\documentclass[11pt,reqno]{amsart}
\usepackage{amsmath}
\usepackage{amsfonts}
\usepackage{amsmath}
\usepackage{amssymb}
\usepackage{amsthm}
\usepackage{array}
\usepackage{calc}
\usepackage{caption}
\usepackage{color}
\usepackage{dashbox}
\usepackage{enumitem}
\usepackage{extarrows}
\usepackage[hang,flushmargin,symbol*]{footmisc}
\usepackage[margin=1in]{geometry}
\usepackage{graphicx}
\usepackage[colorlinks=true, citecolor=black, urlcolor=black, linkcolor=black]{hyperref}
\usepackage[capitalise]{cleveref}
\usepackage{mathdots}
\usepackage{mathrsfs}
\usepackage{mathtools}
\usepackage{setspace}
\usepackage{stmaryrd}
\usepackage{tikz-cd}
\usetikzlibrary{arrows,shapes,positioning,calc}
\usetikzlibrary{decorations.markings}
\usetikzlibrary{patterns}
\usepackage[normalem]{ulem}
\usepackage{url}
\usepackage{verbatim}
\usepackage[colorinlistoftodos,loadshadowlibrary]{todonotes} %\usepackage[disable]{todonotes}

\newtheorem{theorem}{Theorem}[section]

\newtheorem{lemma}[theorem]{Lemma}

\newtheorem{corollary}[theorem]{Corollary}
\newtheorem{proposition}[theorem]{Proposition}

\theoremstyle{definition}
\newtheorem{definition}[theorem]{Definition}
\newtheorem{example}[theorem]{Example}
\newtheorem{remark}[theorem]{Remark}

\newtheorem{situation}[theorem]{Situation}

\newtheorem{warning}[theorem]{Warning}
\newtheorem{condition}[theorem]{Condition}

\newcommand{\scr}[1]{\ensuremath{\mathscr{#1}}}

\renewcommand{\int}{{\rm int}}

\newcommand{\msout}[1]{\text{\sout{\ensuremath{#1}}}}
\newcommand{\ZZ}{\mathbb{Z}}
\newcommand{\CC}{\mathbb{C}}
\newcommand{\NN}{{\mathbb{N}}}
\newcommand{\QQ}{{\mathbb{Q}}}
\newcommand{\OO}{\mathcal{O}}

\newcommand{\RR}{\mathbb{R}}
\newcommand{\Aff}{{\mathbb{A}}}
\newcommand{\PP}{\mathbb{P}}
\newcommand{\GG}{\mathbb{G}}
\newcommand{\Gtrop}{\GG_{\rm trop}}
\newcommand{\Glog}{\GG_{\rm log}}

\newcommand{\Nl}[1]{{N_{#1}^{\ell}}}

\newcommand{\Spec}{{\rm{Spec}\:}}

\newcommand{\Hom}{{\rm Hom}}

\newcommand{\Exal}{\text{Exal}}
\newcommand{\Ext}{\text{Ext}}

\newcommand{\ccx}[1]{{\mathbb{L}_{#1}}}

\newcommand{\lccx}[1]{{\mathbb{L}^\ell_{#1}}}

\newcommand{\gp}[1]{#1^{\rm gp}}
\newcommand{\vir}[1]{ {[{#1}]^{vir}} }
\newcommand{\lvir}[1]{ {[{#1}]^{\ell vir}} }

\newcommand{\Cl}[1]{{C_{#1}^\ell}}

\newcommand{\Tl}[1]{{T^{\rm log}_{#1}}}
\newcommand{\lkah}[1]{\Omega^{\rm log}_{#1}}

\newcommand{\Log}{{\mathcal{L}og}}

\newcommand{\lpb}{{\arrow[dr, phantom, very near start, "\ulcorner \ell"]}}
\newcommand{\lpbstrict}{{\arrow[dr, phantom, very near start, "\ulcorner \msout{\ell}"]}}
\newcommand{\pb}{{\arrow[dr, phantom, very near start, "\ulcorner"]}}
\newcommand{\num}[1]{\langle #1 \rangle}

\def\overnorm#1{\overline{#1}\vphantom{#1}}
\renewcommand{\bar}[1]{\ensuremath{\overnorm{#1}}}
\newcommand{\Ms}[1][{g, n}]{\overline{{\cal M}}_{#1}}
\newcommand{\Mp}[1][{g, n}]{{\frak{M}_{#1}}}
\newcommand{\Ml}[1][{g, n}]{\overline{M}^{\rm log}_{#1}}

\newcommand{\dr}[1][{P, \vec a}]{\overline{\rm DR}^{\rm log}_{#1}}
\newcommand{\drp}[1][{P, \vec a}]{{\rm DR}^{\rm log, ps}_{#1}}
\newcommand{\dd}[1][{P, \vec a}]{\overline{\rm D}_{#1}}
\newcommand{\ddp}[1][{P, \vec a}]{{\rm D}^{\rm ps}_{#1}}

\newcommand{\DR}[1][{g, n, \vec a}]{\overline{\rm tDR}^{\rm log}_{#1}}
\newcommand{\pDR}[1][{g, n, \vec a}]{\overline{\rm DR}^{\rm log}_{#1}}
\newcommand{\D}[1][{g, n, \vec a}]{\overline{\rm D}_{#1}}
\newcommand{\Dirr}[1][{g, n, \vec a}]{\overline{\rm DR}^{\rm irr}_{#1}}

\newcommand{\DRp}[1][{g, n, \vec a}]{{\rm tDR}^{\rm log, ps}_{#1}}
\newcommand{\pDRp}[1][{g, n, \vec a}]{{\rm DR}^{\rm log, ps}_{#1}}
\newcommand{\Dp}[1][{g, n, \vec a}]{{\rm D}^{\rm ps}_{#1}}

\newcommand{\DRA}[1][{g, n, A}]{\DR[#1]}
\newcommand{\pDRA}[1][{g, n, A}]{\pDR[#1]}

\newcommand{\Pic}{\underline{\rm Pic}}
\newcommand{\pic}{{\rm Pic}}
\newcommand{\qpic}{{\rm Pic}^{qs}}
\newcommand{\LogPic}{\underline{\rm LogPic}}
\newcommand{\logpic}{{\rm LogPic}}
\newcommand{\LP}[1][{g, n}]{\underline{\rm LP}_{#1}}
\newcommand{\lp}[1][{g, n}]{{\rm LP}_{#1}}
\newcommand{\LPp}[1][{g, n}]{\underline{\rm LP}^{\rm ps}_{#1}}
\newcommand{\lpp}[1][{g, n}]{{\rm LP}^{\rm ps}_{#1}}

\newcommand{\GL}{{\rm GL}}

\newcommand{\aj}{{\rm aj}}

\newcommand{\CHl}{{\rm logCH}_*}

\newcommand{\Kl}{{\operatorname{log}\!K}_\circ}
\newcommand{\cKl}{{\rm logK}^\circ}

\def\Sym#1#2{[\operatorname{Sym}^{#1}#2]}
\def\Symc#1#2{\operatorname{Sym}^{#1}#2}

\DeclareMathOperator*{\colim}{colim}
\newcommand{\et}{{\text{\'et}}}

\newcommand{\bra}[1]{\left[{#1}\right]}
\newcommand{\pra}[1]{\left({#1}\right)}

\newcommand{\AF}[1][{}]{\Theta_{{#1}}}
\newcommand{\af}[1]{\AF[#1]}
\newcommand{\cal}[1]{\ensuremath{\mathcal{#1}}}
\newcommand{\pt}{{\Spec \ZZ}}

\newcommand{\rk}{{\rm rk}}

\renewcommand{\int}[1]{{\rm int}\,#1}

\newcommand{\Alt}{{\rm Alt}}

\newcommand{\Mons}{(Mon)}
\newcommand{\longsimeq}{\overset{\sim}{\longrightarrow}}

\renewcommand{\vert}[1]{#1^{\rm vert}}

\renewcommand{\tilde}[1]{\widetilde{#1}}
\renewcommand{\hat}[1]{\widehat{#1}}

\newcommand{\Mbar}{{\bar M}}

\usepackage[foot]{amsaddr}

\usepackage{xparse}

\NewDocumentCommand{\rt}{m o}{%
  \IfNoValueTF{#2}
    {#1^{1/r}}  % if no optional value, default to r
    {#1^{1/#2}} % else use provided value
}

\usepackage{longtable}
\usepackage{array}

\title{On the $K$-theoretic logarithmic double ramification class}

\author{Kamyar Amini}
\author{You-Cheng Chou}
\author{Leo Herr}
\author{David Holmes}
\author{Irit Huq-Kuruvilla}
\author{Yuan-Pin Lee}
\date{\today}

\begin{document}

\maketitle

\begin{abstract}
The logarithmic double ramification cycle is the virtual fundamental class of the locus where a line bundle on a family of curves is fiberwise trivial.
We construct a K-theoretic logarithmic double ramification class and prove a product formula and a \(\mathrm{GL}_r(\mathbb Z)\)-invariance property.
We also give an explicit formula for this class in terms of a Grothendieck polynomial via a novel $K$-theoretic Thom--Porteous formula for vector bundles on algebraic stacks.
\end{abstract}

\setcounter{tocdepth}{1}
\tableofcontents

\section{Introduction}

We write $\Ms^{\rm irr} \subseteq \Ms$ for the open locus of irreducible curves. Given an $n$-tuple of integers $\vec a = (a_1, \dots, a_n) \in \ZZ^n$ which sum to zero $\sum_i a_i = 0$, define
\[
 \Dirr := \left\{ (C, p_i) \in \Mbar^{irr}_{g,n} \mid \OO_C(\sum_i a_i p_i) \cong \OO_C  \right\} \subseteq \Mbar^{irr}_{g,n}
\]
as the locus where the line bundle $\OO_C(\sum a_i p_i)$ is fiberwise trivial\footnote{Our results apply equally well to the locus where any other fixed line bundle on $C$ is fiberwise trivial, such as twists of the dualizing sheaf $\omega^k_{C}(\sum a_i p_i)$.}.     
The problem of extending this construction to stable curves and deriving a formula for the resulting class was originally proposed by Y.~Eliashberg in 2001 within the framework of symplectic field theory. This was resolved in 2017 by Janda, Pandharipande, Pixton, and Zvonkine \cite{DRcycleonmoduliofcurves}, who proved that DR classes are tautological classes on the moduli space of stable curves. 
Holmes, Pixton and Schmitt \cite{drmultbchow} constructed a refinement called the \emph{log DR class} which satisfies a formula proven in \cite{LogDRCycle} via fixed point localization techniques; a more direct proof was given in \cite{chiodoholmesDRroots}.

\subsection{The double ramification class}

We outline two constructions of these virtual classes and their logarithmic lifts, which we will lift to $K$-theory. As the name suggests, the line bundle $\OO_C(\sum_i a_i p_i)$ corresponds to a ramified cover $f: C \to \PP^1$. The preimages of infinity and zero are determined by the signs of the integers $a_i$:
\[f^* \infty = \sum_{i, a_i < 0} -a_i p_i \qquad 
f^* 0 = \sum_{i, a_i > 0} a_i p_i.\] 
We can view $\Dirr$ as a space of stable (rubber) maps to $\PP^1$, and a virtual class can be defined via the standard obstruction theory on rubber stable maps to $\PP^1$ \cite{junli1,RdC}. 

Alternatively, consider the Picard scheme of degree zero line bundles on the universal curve $\cal U^{irr}/\Ms^{irr}$:
\[
   \pi: \operatorname{Pic}_{g,n}^{0, irr} \coloneqq \operatorname{Pic}^{0}(\cal U^{irr}/\Ms^{irr}) \to \Mbar^{irr}_{g,n}. 
\]
This family of semi-abelian varieties of relative dimension $g$ comes with two natural sections; a ``zero section'' 
\[
e: \Mbar^{irr}_{g,n} \to \operatorname{Pic}_{g,n}^{0, irr}, \qquad e(C, p_i) = (C, \OO_C)
\]
defined by the trivial line bundle, and the ``Abel--Jacobi'' section determined by the vector $\vec a \in \ZZ^n$
\[
 \aj_{\vec a} : \Mbar^{irr}_{g,n} \to \operatorname{Pic}_{g,n}^{0, irr}, \qquad \aj_{\vec a} (C, p_i) = (C, \OO_C (\sum_i a_i p_i)).
\]
Then $\Dirr = e^{-1} (\aj_{\vec a}) = \aj^{-1}_{\vec a}(e)$ is the intersection of these two sections. These two sections do not intersect transversely, but a virtual class can be defined via Gysin pullback 
\[
    [\Dirr]^{vir} = e^!(\aj_{\vec a}) = \aj^!_{\vec a}(e).
\]
These two constructions define the same virtual fundamental class in $CH_*(\Dirr)$. 

\begin{definition}\label{def:K_DR}
    The DR class $[\Dirr]^{vir}$ in $K$-theory is the $K$-theoretic Gysin pullback along $\aj_{\vec a}$ of the structure sheaf $e_* \OO_{\Mbar^{irr}_{g,n}}$ of the zero section $e$:
    \[
        \aj_{\vec a}^! [e_* \OO_{\Mbar^{irr}_{g,n}}] \qquad \in K_\circ(\Dirr).
    \]
\end{definition}

This $K$-theoretic DR class may be ``different'' from that in Chow, in that they need not be identified under Grothendieck--Riemann--Roch isomorphisms. This is a refinement of the Chow version, as described in \S \ref{ss:whykthy?}.

\subsubsection{Extension to $\Ms$ using logarithmic schemes}
The spaces $\Ms^{\rm irr}$ are noncompact, which prevents their intersection theory from being well-behaved. Extending these objects to $\Ms$ is necessary for most applications in symplectic topology and algebraic geometry. This extension can be achieved by incorporating logarithmic structures via two different approaches. In the first approach, we view $\pDR$ as a space of log stable maps, akin to maps to ``rubber $\PP^1$,'' equipped with its natural log perfect obstruction theory. See Remark~\ref{rmk:comparisonwithrubber} for a detailed comparison. In the second approach, log Picard stacks are utilized to define a log perfect obstruction theory on a compactification $\pDR$. Both of these approaches are well-established in Chow, but we lift them to $K$-theory in a manner analogous to Definition \ref{def:K_DR}. The first main result of this work establishes the equivalence of these two perspectives.

\begin{theorem}[=Corollary \ref{cor:samelogPOT}]
The two log perfect obstruction theories are equivalent. In particular, they define the same virtual fundamental class $\lvir{\pDR}$ in cohomology, Chow, and $K$-theory.
\end{theorem}
We refer to this $K$-theoretic class as the ``logarithmic virtual structure sheaf'' or the ``logarithmic virtual fundamental class in $K$-theory'' and write $\OO^{\rm vir}_{\pDR} = \lvir{\pDR}$.

In Chow, logarithmic DR classes push forward to the usual (non-logarithmic) double-ramification class $[{\rm DR}_{g, n, \vec a}]^{vir}$ on $\Ms$, as defined in  \cite{DRcycleonmoduliofcurves,earlyholmesDRpaper,wisemarcuslogabeljacobi}. Analogously, we may define the (non-logarithmic) $K$-theoretic DR class as the pushforward of the logarithmic $K$-theoretic DR class.

\subsection{Product formula}

Two line bundles $\mathcal L$ and $\mathcal M$ are simultaneously trivial if and only if $\mathcal L$ and $\mathcal L \otimes\mathcal M$ are simultaneously trivial, so
\begin{equation}
    \Dirr[g, n, \vec a] \cap \Dirr[g, n, \vec b] = \Dirr[g, n, \vec a + \vec b]\cap \Dirr[g, n, \vec b].
\end{equation}
With a little more work, one shows that 
\begin{equation}\label{eq:DDRirr}
    \vir{\Dirr[g, n, \vec a]} \cdot \vir{\Dirr[g, n, \vec b]} = \vir{\Dirr[g, n, \vec a + \vec b]}\cdot \vir{\Dirr[g, n, \vec b]}
\end{equation}
in $K^\circ(\Ms^{\rm irr})$. 

The analogous equality of classes $\vir{\pDR[g,n,-]}$ makes sense in $K^\circ(\Ms)$, but it is \emph{false} (the analogous phenomenon in Chow was observed in \cite{drmultbchow}). 
However, it holds for the logarithmic virtual classes living on sufficiently fine blowups of $\Ms$ (again mirroring the situation in Chow \cite{drmultbchow}). To this end, define the \emph{colimit log $K$-theory ring}
\[
    \cKl(X) \coloneqq \colim_{\tilde X \to X} K^\circ(\tilde X)
\]
as a colimit over the locally free $K$-theories of log alterations $\tilde X \to X$ made precise in \S \ref{ss:logvfccolimitlogK}. 

The log DR classes naturally live on $\Dirr$, but we also consider them on $\Ms$ via an implicit pushforward:

\begin{theorem}
    The log DR classes lie in the \emph{colimit log $K$-theory ring} $\cKl(\Ms)$ of $\Ms$. 
\end{theorem}

Most log virtual fundamental classes live in the larger ambient group of limit log $K$-theory \cite[\S 3]{kthylogprodfmla}. We need to work in the subring of colimit log $K$-theory so products like \eqref{eq:DDRirr} are well defined. We then prove the analogue of \eqref{eq:DDRirr} in $\cKl(\Ms)$, extending from two to arbitrarily many cycles. 

\begin{theorem}[{=Theorem~\ref{t:product}}] \label{t:2} 
    Let $A = [\vec a_i]$ be an $n \times r$ integer matrix where the entries of each column $\vec a_i \in \ZZ^n$ sum to zero. Let $M \in \GL_r(\ZZ)$ be an invertible integer matrix, and denote the columns of $AM = [\vec b_i]$ as $\vec b_i$.

    The products of the corresponding double ramification classes agree in colimit log $K$-theory:
    \[
    \lvir{\pDR[g, n, \vec a_1]} \cdot \cdots \cdot \lvir{\pDR[g, n, \vec a_r]} 
    = 
    \lvir{\pDR[g, n, \vec b_1]} \cdot \cdots \cdot \lvir{\pDR[g, n, \vec b_r]} \qquad \in \cKl(\Ms).
    \]
\end{theorem}

This is the natural lift to $K$-theory of \cite[Theorem 5.3]{holmesschwarzGLrinvarianceprodfmla}. It may also be seen as a version of the \emph{log product formula} of \cite{herrthesis,kthylogprodfmla,logprodfmla}. The equality is proven by equating both sides with a third object, the log virtual fundamental class of a locus $\pDRA$ where $r$-tuples of line bundles are simultaneously trivial.

\subsection{Degeneracy loci on algebraic stacks}

An explicit formula for the logarithmic DR cycle was stated in \cite{LogDRCycle} and proven there by virtual localization techniques; a more direct proof via the Riemann--Roch and Thom--Porteous formulas was given in \cite{chiodoholmesDRroots}. We will follow the latter route, for which we need a Thom--Porteous formula for the
$K$-theoretic degeneracy loci on algebraic stacks.

We fix a map of vector bundles $E \to F$ on an algebraic stack $Y$. Write $r_1 = \rk E, r_2 = \rk F$. We specialize to the case $r_1 \leq r_2$ and describe the $(r_1-1)$-degeneracy locus $X \subseteq Y$ where $E \to F$ does not have maximal rank, but we expect a similar approach to give formulas for all degeneracy loci. 

\begin{definition}\label{def:stackygrothpoly}
    Let $\tilde G$ be the operator on the $K$-theory of an algebraic stack $Y$ defined by sending a difference $E - F$ of vector bundles to:
    \begin{equation}\label{eqn:stackygrothpoly}
    \tilde G(E-F) = [\OO_Y] - [\det E] \cdot \bra{\dfrac{\lambda_t(F^\vee)}{\lambda_{t^{-1}}(E)}}_{{\rm coeff} \, t^{r_1}},
    \end{equation}
    where $r_1 = \rk E$ and $\lambda_t(E) = \sum_{i = 0}^\infty [\wedge^i E] t^i$. 
\end{definition}

We check this formula only depends on the class of $[E]-[F]$ in $K$-theory in Lemma \ref{lem:tildeGwelldefKtheory}.

\begin{theorem}[=Corollary \ref{cor:thomporteousKthytildeGdegenlocusformula}]\label{introthm:thomporteous}
    Let $E \to F$ be a morphism of vector bundles on an algebraic stack $Y$. If the degeneracy locus $e : X \subseteq Y$ is a regular immersion of expected codimension $r_2 - r_1 + 1$, then 
    \[
    e_* [\OO_X] = \tilde G(E-F) \qquad \in K^\circ(Y).
    \]
\end{theorem}

\begin{remark}
    Versions of Theorem \ref{introthm:thomporteous} for the $K$-theory of schemes 
    have been developed by Buch \cite{andersbuchthomporteous}, Anderson \cite{AndersonKtheoreticdegeneracyloci}, Hudson--Ikeda--Matsumura--Naruse \cite{Hudson_Ikeda_Matsumura_Naruse}, and Tamvakis—Wilson \cite{Tamvakis_Wilson_Giambelliformulas}.
    These expressions are formulated as infinite sums that truncate on schemes but diverge on stacks, as elements of virtual rank zero need not be nilpotent.  

    We instead use an explicit finite, locally free resolution of the structure sheaf valid for stacks and in mixed characteristics due to Eagon--Northcott, Buchsbaum--Rim, Eisenbud, and Weyman. See \cite[\S A.2.6.1]{commutativealgebraEisenbud1995} for historical context on the resolution. Section \ref{s:thomporteous} uses this resolution to prove Theorem \ref{introthm:thomporteous}, as foreshadowed by \cite[Remark 1.2]{AndersonKtheoreticdegeneracyloci}.
\end{remark}

\subsection{Formulas for logarithmic double ramification cycles}
We use the above results to compute the logarithmic $K$-theoretic double ramification cycle. Fix a compactified Jacobian $\mathcal{J} \to \Ms$ as in Definition~\ref{d:2.20}. It is a proper algebraic stack which contains the stack of multidegree-zero line bundles $\pic^{[0]}_{C/S}$ on a family of curves $C/S$ as a dense open. Here we assume that $g>0$, otherwise $\cal J \to \Ms$ is an isomorphism and $\lvir{\pDR}$ is just its fundamental class. 

\begin{theorem}[{=Theorem \ref{thm:rthrootGrothendieckdegenlocusDR} + Proposition \ref{prop:drclassgrothpoly}}]\label{introthm:degeneracylocuszerolocus}   
    Write $\pi': \mathscr{C}' \to \mathcal{J}$ is the universal curve (quasistable model, cf.\ Definition~\ref{def:rootjacquasistable}) and $\mathscr{F} \to \mathscr{C}'$ is the universal line bundle. The structure sheaf of the zero section $e : \Ms \to {\cal J}$ is the value of the operator $\tilde G$ on the $K$-theoretic class of the complex $[R\pi'_* {\scr F}] = [R^0 \pi'_*{\scr F}] - [R^1 \pi'_* {\scr F}]$ 
    \[
        [e] = e_*[\OO_{\Ms}] = \tilde G([R\pi'_* {\scr F}]) \qquad \in K_\circ({\cal J}).
    \]
    As a result, the $K$-theoretic log DR class can be expressed as a class of degeneracy locus
    \begin{align*}
        \lvir{\pDR} &= \aj_{\vec a}^! \tilde G(R\pi'_*{\scr F})         
                \qquad \in K_\circ(\tilde \Ms) 
    \end{align*}
    on the log alteration $\Ms \times_{\aj_{\vec a}, \lp}^\ell \cal J \to \Ms$. 
\end{theorem}

In Theorem \ref{thm:rthrootGrothendieckdegenlocusDR} we generalize Theorem \ref{introthm:degeneracylocuszerolocus} to triviality loci of $r$th roots of line bundles. This gives a precursor to Pixton's formula in log $K$-theory spelled out in \S \ref{sss:kthypixtonkamyarlastmorsel}
which is parallel to the one in Chow \cite[Main Theorem 2]{chiodoholmesDRroots}. The $r$-independence is a feature of log geometry.

\subsection{Novelties and challenges}

The technical core of this work builds upon the strategies in \cite{drmultbchow}, \cite{holmesschwarzGLrinvarianceprodfmla} and \cite{chiodoholmesDRroots} within the setting of Chow or ordinary cohomology. 
Adapting them to the present context leads to several technical challenges and novelties. We highlight a few below.

\begin{enumerate}
\item While various definitions of the double ramification space $\pDR$, its perfect obstruction theories, and its virtual class $[\pDR]^{\rm vir}$ exist across the literature, their equivalences are not immediately apparent. This is addressed in this work.
\item Working with refined classes in \emph{colimit} log $K$-theory is not only possible but essential for the \emph{log product formula} in $K$-theory. This stems from the fact that there is no ring structure on \emph{limit} log $K$-theory due to the absence of well-defined product map $\Kl(X) \times \Kl(Y) \dashrightarrow \Kl(X \times Y)$.
\item The Artin fan is not functorial in general \cite[5.4.1]{skeletonsfansabramchenmarcusulrischwise}. However, the construction of both colimit log $K$-theory and colimit Chow theory relies on specific, bounded functoriality properties of Artin fans. We provide the proof for these required properties in Appendix~\ref{Appendix A}.
\item Existing $K$-theoretic Thom–Porteous formulas in the literature do not admit a straightforward generalization to algebraic stacks. Versions of Theorem~\ref{introthm:thomporteous} for the $K$-theory of schemes have been developed by Buch \cite{andersbuchthomporteous}, Anderson \cite{AndersonKtheoreticdegeneracyloci}, Hudson–Ikeda–Matsumura–Naruse \cite{Hudson_Ikeda_Matsumura_Naruse}, and Tamvakis–Wilson \cite{Tamvakis_Wilson_Giambelliformulas}. 

These expressions are formulated as infinite sums which naturally truncate on schemes. However, they could diverge in the $K$-theory of algebraic stacks because elements of virtual rank zero are not necessarily nilpotent.
To overcome this issue, we instead employ an explicit, finite, locally free resolution of the structure sheaf that remains valid for stacks and in mixed characteristics. This resolution is rooted in the classical work of Eagon--Northcott \cite{eagonnorthcottoriginal}, Buchsbaum–-Rim, Eisenbud, Bruns--Vetter \cite{brunsvetteroriginal}, and Weyman (see \cite[\S A.2.6.1]{commutativealgebraEisenbud1995} for historical context). 
Anderson \cite[Remark 1.2]{AndersonKtheoreticdegeneracyloci} foreshadowed the utility of such a resolution, which we use to prove Theorem \ref{introthm:thomporteous} in Section \ref{s:thomporteous}. 
\end{enumerate}

One remaining open technical challenge is the \emph{consistent} utilization of the logarithmic Picard functor $\logpic$ \cite{logpic}, as many of its foundational properties remain undeveloped. While a fully realized theory of $\logpic$ would offer a more uniform approach to the subject, the current lack of a robust Brill--Noether theory on $\logpic$ prevents its full implementation here. Consequently, in \S\ref{s:5} we follow existing literature by resorting to birational models of \logpic, i.e., compactified Jacobians.

\subsection{Why (quantum) $K$-theory?}\label{ss:whykthy?}

Quantum $K$-theory counts geometric curves in a target variety just like quantum cohomology, but in a refined way \cite{giventalquantumKtheory,quantumkthyypleethesis} :

\begin{itemize}
    \item Refined Enumerative Geometry: Standard Gromov-Witten theory (quantum cohomology) counts how many algebraic curves of a certain degree pass through given constraints. Quantum $K$-theory doesn't just count them; it computes the Euler characteristics of coherent sheaves on the moduli spaces of these curves.

    \item Representation Theory: Quantum $K$-theory groups often carry actions of large algebraic structures, like quantum groups or affine algebras. Mathematicians use it to solve hard problems in representation theory by translating them into geometric questions about flag manifolds and other quiver varieties. Pushforward along torsors and gerbes in $K$-theory also contains representation-theoretic information, so the virtual structure sheaves often do as well.  
\end{itemize}

Quantum $K$-theory also has its own version of mirror symmetry coming from physics. For theoretical physicists, quantum $K$-theory naturally arises in string theory and gauge theory when studying the quantum mechanics of universes with extra dimensions. 
Its influence on theoretical physics is largely its relation to 3-dimensional topological field theory. See the pioneering works of N.~Nekrasov, H.~Jockers, P.~Mayr etc.\ (\cite{jockerspeterquantumkphysics}, \cite{jockerspeterquantumkphysics2}, and references therein.) 
\begin{itemize}
    \item Wilson Loops and Gauge Theories: In $3D$ supersymmetric gauge theories, physicists study line operators called Wilson loops (see e.g. \cite{mihalceawilsonlinestuff},\cite{wilsonloopjockersmayr}). The algebra of these Wilson loops, when wrapped around certain spaces, is captured by the quantum $K$-theory of the Higgs branch of the theory.

    \item String Theory and $D$-Branes: In string theory, $D$-branes are physical objects where open strings can end. It was observed by Witten \cite{wittendbranes} that the ``charges'' of these $D$-branes are not measured by standard cohomology, but by $K$-theory. Quantum $K$-theory helps physicists calculate the quantum corrections (from worldsheet instantons) to these $D$-brane states.

    \item Quantum Integrable Systems: Quantum $K$-theory rings are isomorphic to the state spaces of certain quantum integrable systems (like the $XXZ$ spin chain model, discussed in e.g. \cite{xxzreference}). It allows physicists to use geometry to find exact solutions for complex, interacting quantum systems.
\end{itemize}

As with other virtual fundamental classes, the $K$-theoretic DR classes are in general
different from %that 
those in Chow. Even though there are Grothendieck--Riemann--Roch isomorphisms $K_\circ(X)_\QQ \simeq CH_*(X)_\QQ$ for a smooth, projective scheme $X$ 
\cite{Fultonintersectiontheory}, 
the (virtual) structure sheaves $[\OO_S] \in K_\circ(X)_\QQ$ of subschemes $S \subseteq X$ do not in general map to (virtual) fundamental classes in Chow $[S] \in CH_*(X)$.
These problems worsen for nice algebraic stacks $X$, where Grothendieck--Riemann--Roch \cite{KawasakiRR,toenGRR,edidingrahamequivariantGRR,Joshua2003} identifies $K_\circ(X)_\CC$ with the Chow groups of the inertia stack $IX$ instead of $CH_*(X)_\CC$.

\subsection{Background and Conventions}
\subsubsection*{Background}
This work assumes only a basic understanding of $K$-theory, such as the introductory material found in the opening pages of Fulton--Lang's book \textit{Riemann--Roch Algebra} \cite{RiemannRochAlgebra}. On the other hand, we utilize more sophisticated tools from logarithmic geometry. Familiarity with log schemes at the level of \cite{katooriginal, ogusloggeom} is assumed throughout the exposition. %We outline some of the basics in Appendix \ref{Appendix B}. 
To assist the reader, we have summarized the necessary definitions, terminology, and fundamental results in Appendix~\ref{Appendix B}.

\subsubsection*{Conventions}

We work exclusively with \emph{f.s.} (log) schemes and stacks locally of finite type over $\Spec \ZZ$. The stack $\Log$ therefore parameterizes \emph{f.s.} log structures (denoted $\cal T or$ in \cite{logstacks}). When we consider $r$th roots of line bundles in Section \S \ref{s:5}, we work over $\Spec \ZZ[\frac{1}{r}]$. 
Any log scheme $S$ has a short exact sequence of sheaves on its small strict-\'etale site: %\david{which etale site? Small strict?}
\[
1 \to \OO^*_S \to \gp M_S \to \gp{\bar M}_S \to 0.
\]
Write $\Glog$ and $\Gtrop$ for the sheaves on the big strict-\'etale site of log schemes: 
\[
\Glog: S \mapsto \Gamma(S, \gp M_S),
\]
\[
\Gtrop: S \mapsto \Gamma(S, \gp{\bar M}_S).
\]

Sections $X \to \Gtrop$ are called ``piecewise linear functions.'' They can be interpreted as piecewise linear functions on the tropicalization of $X$ via Lemma \ref{lem:charmonPLfns}. If $X/S$ is a family of curves with marked point $p : S \to X$ at which $\gp{\bar M}_{X/S, p} = \ZZ$, the restriction of a piecewise linear function $\alpha$ on $X$ to $\gp{\bar M}_{X/S, p}$ is called the \emph{slope} of $\alpha$ at $p$. 

We will often consider vectors $\vec a \in \ZZ^n$ of integers which sum to zero $\sum a_i = 0$, as well as $n \times r$-matrices $A = (a_{ij})$ whose columns each sum to zero $\sum_i a_{ik} = 0$. Write $\ZZ^n_0 = \ker(+ : \ZZ^n \to \ZZ)$ for the set of vectors in $\ZZ^n$ whose coordinates sum to zero. 
Write $\underline G = \underline G_S$ for the constant sheaf with value $G$ on %a scheme 
$S$. If $G$ is a sheaf of groups acting on a sheaf $H$ in a topos $X$ and $P$ is a $G$-torsor on $X$, write 
% $G \to H$ is a morphism of sheaves of groups on %a scheme
% $S$ and $P$ is a $G$-torsor, write 
\[
    P \times^G H \coloneqq \pra{P \times H}/G
\]
quotient by the antidiagonal action of $G$. For example, $G \to H$ could be a morphism of sheaves of groups. 

We write $K_\circ(X)$ for the Grothendieck group of coherent sheaves (a.k.a.\ ``$G$-theory'') on an algebraic stack $X$ and $K^\circ(X)$ for that of locally free sheaves.

\subsection{Glossary of notation}
\,

\vspace{10pt}

\setlength{\LTpre}{0pt}
\setlength{\LTpost}{0pt}
\begin{longtable}{>{\raggedright\arraybackslash}p{0.22\textwidth}
                  >{\raggedright\arraybackslash}p{0.58\textwidth}
                  >{\raggedright\arraybackslash}p{0.18\textwidth}}

\textbf{Notation} & \textbf{Meaning} & \textbf{Where introduced} \\
\hline
\endfirsthead

\textbf{Notation} & \textbf{Meaning} & \textbf{Where introduced} \\
\hline
\endhead

\hline
\endfoot

\multicolumn{3}{l}{\emph{General conventions}}\\
\hline

$g,n$ &
Genus and number of marked points &
Throughout \\

$\vec a=(a_1,\dots,a_n)$ &
Integer vector with $\sum_i a_i=0$ (contact orders) &
Introduction; \S\ref{s:dr} \\

$\ZZ^n_0$ &
Kernel of $\ZZ^n \xrightarrow{+} \ZZ$, i.e.\ vectors summing to zero &
Conventions \\

$A=(a_{ij})$ &
$n\times r$ integer matrix with each column summing to zero &
\S\ref{higherranksection} \\

\multicolumn{3}{l}{\emph{Moduli of curves and log geometry}}\\
\hline

$\Ms$ &
Moduli stack of stable curves &
Throughout \\

$\Mp$ &
Moduli stack of prestable curves &
\S\ref{s:dr} \\

$C/S$ &
A (log) prestable curve over a log scheme $S$ &
Throughout \\

$p_i$ &
Marked sections of a curve &
Throughout \\

$M_X$ &
Log structure on a scheme or stack $X$ &
Appendix~\ref{Appendix B} \\

$\bar M_X$ &
Characteristic monoid $M_X/\OO_X^\times$ &
Appendix~\ref{Appendix B} \\

$\gp P, \gp M_X, \gp{\bar M}_X$ &
Associated group of a monoid or a log structure &
Conventions \\

$\vert C$ &
Vertical log curve (log structure only at nodes) &
\S\ref{s:dr} \\

$\af{X}$ &
Artin fan of a log stack $X$ &
Appendix~\ref{Appendix B} \\

$\Glog$ &
Sheaf $S\mapsto \Gamma(S,\gp M_S)$ &
Conventions \\

$\Gtrop$ &
Sheaf $S\mapsto \Gamma(S,\gp{\bar M}_S)$ &
Conventions \\

\multicolumn{3}{l}{\emph{Double ramification loci}}\\
\hline

$\pDR \subseteq \pDRp$ &
Pre/stable log double ramification space &
Definition in \S\ref{s:dr} \\

$\DR \subseteq \DRp$ &
Log blowup of the log DR space by a log algebraic stack &
Definition~\ref{def:tDR} \\

$\D \subseteq \Dp$ &
Fiber product defining the pre/stable log DR locus via Abel--Jacobi map &
Definition~\ref{def:DAJkernel} \\

$\DRA,\pDRA$ &
Higher-rank DR spaces associated to a matrix $A$ &
\S\ref{higherranksection} \\

\multicolumn{3}{l}{\emph{Abel--Jacobi maps and Picard stacks}}\\
\hline

$\aj_{\vec a}$ &
Log Abel--Jacobi map with contact order $\vec a$ &
Definition~\ref{def:logAJmap} \\

$e=\aj_{\vec 0}$ &
Zero section of the log Picard stack &
Definition~\ref{def:logAJmap} \\

$\LP$ &
Log Picard stack of the universal vertical curve &
\S\ref{def:LP} \\

$\lp$ &
Log Picard sheaf (coarse version of $\LP$) &
\S\ref{def:LP} \\

$\cal J$ &
Compactified Jacobian (log alteration of $\lp$) &
\S\ref{rmk:compactjacobiansexist} \\

$\scr F$ &
Universal admissible line bundle on $\cal J$ &
\S\ref{s:5} \\

$\pi' : {\cal C}' \to {\cal J}$ &
Universal (quasistable) curve over $\cal J$ &
\S\ref{s:5} \\

\multicolumn{3}{l}{\emph{$K$-theory, obstruction theories, and virtual classes}}\\
\hline

$K_\circ(X)$ &
Grothendieck group of coherent sheaves on $X$ &
Background \\

$\Kl(X)$ &
Log $K$-theory: \emph{limit} over log alterations &
Definition~\ref{def:colimitlimitlogKtheory} \\

$\cKl(X)$ &
Log $K$-theory: \emph{colimit} over log alterations &
Definition~\ref{def:colimitlimitlogKtheory} \\

$\CHl(X)$ &
Log Chow groups: limit over log alterations &
\S\ref{ss:logvfcdoubleramificationdef} \\

$\mathbb E$, $\mathbb E^\vee$ &
Hodge bundle and its dual &
\S\ref{ss:DRlogpothodgebundle} \\

$\lvir{X}$ &
Log virtual structure sheaf / virtual class &
\S\ref{ss:logvfcdoubleramificationdef} \\

$e^{\mathrm{log},!}$ &
Log Gysin pullback along the zero section &
\S\ref{s:5} \\

$\Delta^{\mathrm{log},!}$ &
Log Gysin pullback along the diagonal &
Theorem~\ref{t:product} \\

\multicolumn{3}{l}{\emph{Grothendieck polynomials and degeneracy loci}}\\
\hline

$\tilde G(-)$ &
Operator on $K$-theory &
Definition \ref{def:stackygrothpoly} \\

$W \subseteq \bar {\rm Mat}_{r_2 \times r_1}$ &
Universal degeneracy locus in a matrix stack &
Remark~\ref{rmk:BrillNoetherdegenlocus} \\

\multicolumn{3}{l}{\emph{Root stacks and twisted curves}}\\
\hline

$\Ms(r)$ &
Moduli of twisted curves (log root stack of order $r$) &
Definition~\ref{def:twistednodalcurvesrthroot} \\

$\cal J^{1/r}$ &
Stack of $r$th roots of line bundles on $\cal J$ &
Definition~\ref{def:rootjacquasistable} \\

$\scr F^{1/r}$ &
Universal $r$th root line bundle &
Definition~\ref{def:rootjacquasistable} \\

$E^{1/r} \subseteq {\cal J}^{1/r}$ &
Zero locus where $\scr F^{1/r}$ is trivial &
Definition~\ref{def:rootjacquasistable} \\

\end{longtable}

\subsection*{Acknowledgments}

Many of the ideas in the present paper are based on conversations with Jonathan Wise. 
L.H.\ is also grateful to Milo Bechtloff-Weising, Peter Haine, Thibault Poiret, and Pim Spelier for explanations and inspiration. 
We are grateful to Anders Buch, Leonardo Mihalcea, Daniel Orr, and Mark Shimozono for discussions about Grothendieck polynomials. These results were presented at IMJ-PRG and in ICTS-TIFR. We thank A.~Chiodo, P.~Georgieva, C.~Ravi and B.~Sreedhar for inviting us to deliver a talk on this topic.

AI (ChatGPT and Claude) made us aware of the Eagon--Northcott complex, created the glossary, checked for typos, and substantially helped with editing. The proofs of Lemmas \ref{lem:presheafmonoscheme}, \ref{lem:factorthroughpresheafmonoscheme} were refined in collaboration with AI upon noticing a mistake in an earlier version.

\section{The double ramification locus and the Abel-Jacobi map} \label{s:dr}

The \emph{logarithmic double ramification (DR) class/cycle} has multiple definitions in the literature. We establish our preferred formulation before comparing both the resulting moduli spaces and their virtual classes.

\subsection{Log line bundle}

\begin{definition}\label{def:loglinebundle}
    A \emph{log line bundle} $P$ on a log scheme (or log algebraic stack) $X$ is a strict-\'etale $\gp M_X$-torsor, also known as a $\Glog$-torsor. We say that $P$ has \emph{bounded monodromy} if there is a log blowup $\tilde X \to X$ such that the pullback $P|_{\tilde X}$ has trivial associated $\Gtrop$-torsor. Equivalently, this means that $P|_{\tilde{X}}$ is induced from a (classical) line bundle $L$ on $\tilde X$ as in the following diagram
    \[
    \begin{tikzcd}
                &H^1(X, \gp{M}_X) \ar[d]       \\
        H^1(\tilde X, \OO^*_{\tilde X}) \ar[r]       &H^1(\tilde X, \gp M_{\tilde X}) \ar[r]        &H^1(\tilde X, \gp{\bar M}_{\tilde X})       \\
        L \ar[r, mapsto]      &P|_{\tilde X} \ar[r, mapsto]      &0
    \end{tikzcd}    
    \]
\end{definition}

If $L/C$ is a line bundle, we implicitly take its associated log line bundle $L \times^{\OO^*_C} \gp M_{\vert C}$ to arrive at Situation \ref{sit:logsmoothbase}. We adopt the terminology of log curves from \cite{fkatomodulilogcurves}.  We call a log curve $C/B$ \emph{standard} if the stalk of the characteristic monoid $\bar M_C$ at the marked points $p_i$ is rank one over that of the base 
\[
    \bar M_{C, p_i} = \bar M_B \oplus \underline \NN. 
\]
There is another canonical log structure $M_{\vert C}$ on $C$ which agrees with $M_B|_C$ away from the nodes, and which satisfies 
\[
    M_C \simeq M_{\vert C} \oplus_{\OO^*_C} \bigoplus M_{p_i}; 
\]
we call $\vert C = (C, M_{\vert C})$ a \emph{vertical} log curve, or the \emph{verticalization} of $C$. See Appendix \ref{ss:logcurves} for more discussion.

\begin{situation}\label{sit:logsmoothbase}
    Let $B$ be a log smooth log algebraic stack with a prestable family of standard log curves $C/B$ of genus $g$ with $n$ marked points $p_i$ and fix a log line bundle $P$ on $\vert C$. 
\end{situation}

\begin{definition}\label{def:dr}
    Fix $P/C/B$ as in Situation \ref{sit:logsmoothbase}. The \emph{(prestable log) double ramification locus} is the stack $\drp[P]$ over $B \times B\Glog$ whose fiber over an object $Q\colon S \to B \times B\Glog$ is the set of isomorphisms of $\Glog$-torsors $Q|_{C_S} \longsimeq P|_{C_S}$: 
    \begin{equation}\label{eqn:drpoints}
        \left\{
    \begin{tikzcd}
                &\drp[P] \ar[d]       \\
        S \ar[r, "Q", swap] \ar[ur, dashed]       &B \times B\Glog
    \end{tikzcd}
    \right\}    
    \coloneqq 
    \{\text{Isomorphisms } Q|_{C_S} \longsimeq P|_{C_S}\}.
    \end{equation}
    
    Write $\dr[P] \subseteq \drp[P]$ for the open substack on which $C_S/S$ is stable.
\end{definition}

\begin{remark}
    Let $\vec a \in \ZZ^n$ be a vector of integers summing to either zero $\sum a_i = 0$ or some number of the form $\sum a_i = k(2-2g)$. 
    The main example of Situation \ref{sit:logsmoothbase} is the universal standard log curve $C$ on $B = \Ms$ with $P$ the log line bundle associated to 
    \[
        \OO_C(\sum a_i p_i) \quad \text{or}\quad  \omega^k_{C/\Ms} (\sum a_i p_i), 
    \] 
    respectively. The condition that $\OO_C(\sum a_ip_i)$ is trivial as a log line bundle defines a log proper monomorphism $\dr[{\OO_C(\sum a_ip_i)}] \to B$. This subfunctor generically consists of curves $C$ admitting a meromorphic function with poles at $\{p_i\ |\  a_i<0\}$ with multiplicity $-a_i$ and zeros at $\{ p_i \ | \ a_i>0 \}$ with multiplicity $a_i$. Its virtual class is dubbed ``Double Ramification cycle'' due to the prescribed ramification at $0$ and $\infty$. 
    We specialize to this case in \S \ref{ss:triviallinebundle}.

    The condition that $\omega^k(\sum a_i p_i)$ is trivial for $k > 0$ defines the subfunctor of \emph{$k$-canonical divisors} and is discussed in
    \cite{farkaspandharipande-twisteddivisors}. Its virtual fundamental class is known as the ``twisted Double Ramification cycle.''

\end{remark}

\subsubsection{Isomorphisms of $\Glog$-torsors on $C$ vs.\ $\vert C$}

There is a canonical identification 
\[\gp M_{C}/\gp M_{\vert C} = \bigoplus p_{i*} \underline \ZZ,\]
fitting into a short exact sequence 
\begin{equation}\label{eqn:sesmoncurve}
    1 \to \gp M_{\vert C} \to \gp M_C \to \bigoplus p_{i*} \underline \ZZ \to 0.
\end{equation}
This yields a canonical splitting of the characteristic monoid
\begin{equation}\label{eqn:splittingcharmoncurve}
\bar M_C = \bar M_{\vert C} \oplus \bigoplus p_{i*} \underline \NN, 
\end{equation}
but the same global splitting need not exist for $\gp M_C$.

\begin{example}
    Let $L$ be a line bundle on $S$ and let $C = \PP_S(L \oplus \OO_S)$ be the projective completion. There is a pullback square 
    \[
    \begin{tikzcd}
        C \ar[r] \ar[d] \lpbstrict       &\bra{\PP^1/\GG_m} \ar[d]      \\
        S \ar[r]       &B\GG_m.
    \end{tikzcd}    
    \]
    The verticalization $\vert C$ is strict over $S$.
    Let $p \subseteq C$ be the marked point given by the zero section of the line bundle; its underlying scheme is isomorphic to $S$ and $\bar M_p = \bar M_S \oplus \underline \NN$. 

    Let $\alpha \in \Gamma(p, \gp{\bar M}_p) = \Gamma(S, \gp{\bar M}_S) \oplus \ZZ$ be the section $(0, 1)$ generating the factor of $\ZZ$. One may check that $\OO_S(\alpha) \simeq L$. So $M_p \simeq M_S \oplus \underline \NN$ if and only if the line bundle $L/S$ is trivial.
\end{example}

\begin{lemma}\label{lem:Glogtorsorverticalvsstandard}
Let $\vec a \in \ZZ^n$ be a vector and $C/S$ a prestable log curve. 
\begin{enumerate}
    \item Regard $\vec a$ as a constant section of 
    \[\Gamma(C, \bigoplus p_{i*} \underline \ZZ) = \Gamma(S, \underline \ZZ^n).\]
    Its image in $H^1(\vert C, \gp M_{\vert C})$ under the boundary map coming from \eqref{eqn:sesmoncurve} is the $\gp M_{\vert C}$-torsor associated to the line bundle $\OO_C(\sum a_i p_i)$. 

    \item\label{it:CvsCvert} Given a pair of log line bundles $P, P'$ on $\vert C$, an isomorphism $P|_C \longsimeq P'|_C$ of their restrictions along $C \to \vert C$ is equivalent to 
    \begin{itemize}
        \item A section $\vec a \in \Gamma(S, \underline \ZZ^n)$, and
        \item An isomorphism of $P'$ with the twist of $P$ by the $\gp M_{\vert C}$-torsor associated to $\OO_C(\sum a_i p_i)$:
    \[
        \OO_C(\sum a_i p_i) \times^{\OO_C^*} P \longsimeq P'.
    \]
    \end{itemize}

    \item If the isomorphic $\Glog$-torsors $P, P'$ are pulled back from the base $S$ in the previous bullet point, the coordinates of $\vec a$ must sum to zero $\sum a_i = 0$. 
\end{enumerate}
\end{lemma}

\begin{proof}
    For the first point, use the long exact sequence associated to \eqref{eqn:sesmoncurve}. For the third point, we apply the balancing condition.

    Replace $P'$ by $P'-P$ and $P$ by $\gp M_{\vert C}$ to reduce the claim \eqref{it:CvsCvert} to the case where $P'$ is trivial. The proof follows from the long exact sequence associated to \eqref{eqn:sesmoncurve}. 
\end{proof}

A point \eqref{eqn:drpoints} is an isomorphism $Q|_{C_S} \longsimeq P|_{C_S}$ of $\Glog$-torsors on the standard log curve $C_S/S$. By Lemma \ref{lem:Glogtorsorverticalvsstandard}, this isomorphism is equivalent to a vector $\vec a \in \ZZ^n_0$ and an isomorphism 
\[
    \OO_C(\sum a_i p_i) \times^{\OO_C^*} Q|_{\vert C_S} \longsimeq P|_{\vert C_S}. 
\]
The choice of vector $\vec a \in \ZZ^n_0$ is locally constant, defining a discrete invariant called the contact order and yielding a decomposition 
\[
    \drp[P] = \bigsqcup_{\vec a \in \ZZ^n_0} \drp.
\]

\subsubsection{The double ramification locus is not a closed substack}\label{ss:drnotclosed}

We conclude with an example showing the proper log monomorphism $\dr \to B$ need not be a closed or open substack, or even a monomorphism on underlying stacks. The example and Figures \ref{fig:blowupstrictdr} and \ref{fig:tropicalAFblowupdr} are identical to those of \cite[\S 5.1]{kthylogprodfmla}, but with a slightly different setup.

Let $D, E = \PP^1$ and let $C_0 = D \cup E$ be the curve obtained by joining these rational curves along two nodes. Let $p \in D$ and $q \in E$ be marked points disjoint from the nodes. Let $s$ be a log geometric point over which $C_0$ is a log curve and write $\ell_1, \ell_2 \in \bar M_s$ for the smoothing parameters of the nodes. For example, $s$ could have the basic log structure $\bar M_s = \NN^2$ generated by free smoothing parameters $\ell_i = e_i$ for the nodes of $C_0$. 

\begin{lemma}\label{lem:DRcomparableedgelengths}
    If $\OO_{C_0}(2p - 2q) \simeq \OO_{C_0}(\alpha)$ is the line bundle associated to a piecewise linear function $\alpha \in \Gamma(C_0, \gp{{\bar M}}_{C_0})$, then the smoothing parameters $\ell_1 = \ell_2$ coincide. 
\end{lemma}

\begin{proof}
    The piecewise linear function $\alpha$ corresponds to a pair of values $\alpha_D, \alpha_E \in \gp{{\bar M}}_s$ such that 
    \[
        \alpha_D - \alpha_E = m_1 \ell_1 = m_2 \ell_2
    \]
    for some integers $m_i \in \ZZ$ \cite[Remark 1.2]{loggw}. These integers must have the same sign.  
    
    The line bundle associated to $\alpha$ on $C_0$ has degrees: 
    \[
        \deg_D \OO_{C_0}(\alpha) = m_1 + m_2, \qquad \deg_E \OO_{C_0}(\alpha) = -m_1 - m_2,
    \]
    up to conventions about incoming/outgoing edges. 
    For it to be isomorphic to $\OO_{C_0}(2p-2q)$ we must have $m_1 + m_2 = 2$. As they have the same sign, we must have $m_1 = m_2 = 1$ and then $\ell_1 = \ell_2$. 
\end{proof}

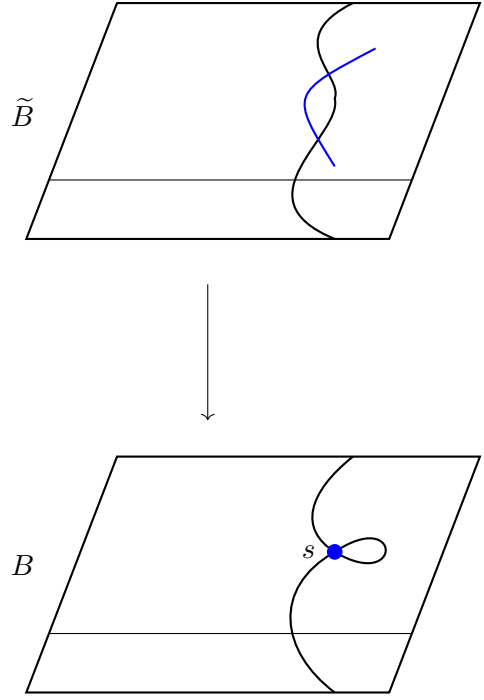
\begin{figure}
    \centering
\begin{tikzpicture}[scale=1.2]
  \coordinate (A) at (-.8,0);
  \coordinate (B) at (3.2,0);
  \coordinate (C) at (4.2,2.6);
  \coordinate (D) at (.2,2.6);

  \draw[thick] (A) -- (B) -- (C) -- (D) -- cycle;

  \draw[thick]
    (2.8,2.6)
    .. controls (2.3,2.2) and (2.2,1.8) .. (2.6,1.55)
    .. controls (3.35,1.1) and (3.35,2.05) .. (2.6,1.55)
    .. controls (1.9,1.15) and (2.0,0.45) .. (2.6,0);

    \filldraw[blue] (2.6, 1.55) circle (.08);
    \node[left] at (2.5, 1.55){$s$};

    \node[left] at ($(A)!.5!(D) + (-.3, .1)$) {$B$};

    \draw ($(A)!0.25!(D)$) -- ($(B)!0.25!(C)$);
    
    \begin{scope}[shift = {(0, 5)}]
        
    \coordinate (A) at (-.8,0);
    \coordinate (B) at (3.2,0);
    \coordinate (C) at (4.2,2.6);
    \coordinate (D) at (.2,2.6);
    
    \draw[thick] (A) -- (B) -- (C) -- (D) -- cycle;

    \draw[thick]
    (2.8,2.6)
    .. controls (2,2.2) and (2.7,1.8) .. (2.6,1.55)
    .. controls (2.7,1.15) and (1.5,0.45) .. (2.6,0);

    \draw[thick, blue]
    (3.05, 2.1)
    .. controls (2.1, 1.6) .. (2.6, .8);
    
    \draw ($(A)!0.25!(D)$) -- ($(B)!0.25!(C)$);
  
    \node[left] at ($(A)!.5!(D) + (-.3, .1)$) {$\tilde B$};
    \draw[->] ($(A)!.5!(B) + (0, -.5)$) to ($(A)!.5!(B) + (0, -2)$);
    \end{scope}
\end{tikzpicture}
    \caption{The blowup $\tilde B \to B = \Ms[1, 2]$ through which $\dr[\OO_C(2p-2q)]$ factors as a strict closed immersion.}
    \label{fig:blowupstrictdr}
\end{figure}

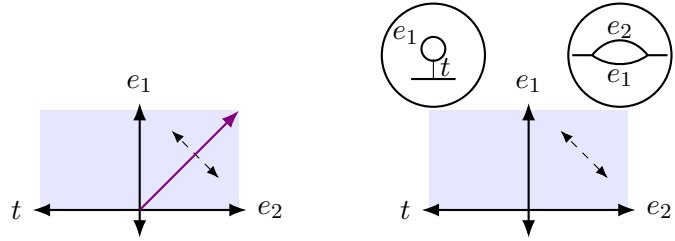
\begin{figure}
\centering
\begin{tikzpicture}[
  scale=1.05,
  >=Latex,
  cone/.style={fill=blue!10, draw=none},
  axis/.style={->, thick},
  purplearrow/.style={->, thick, violet},
  smallpic/.style={thick}
]

\begin{scope}[shift={(4.2,3.75)}]
  \draw[smallpic] (0,0) circle (0.65);
  \draw[smallpic] (0,0.08) circle (0.15);
  \draw[smallpic] (-0.28,-0.3) -- (0.28,-0.3);
  \draw[-] (0,-0.3) -- (0,-.05);

  \node at (-0.35,0.25) {$e_1$};
  \node at (0.15,-0.15) {$t$};
\end{scope}

\begin{scope}[shift={(6.55,3.75)}]
  \draw[smallpic] (0,0) circle (0.65);
  \draw[smallpic] (-0.6,0) -- (-0.35,0);
  \draw[smallpic] (0.35,0) -- (0.6,0);
  \draw[smallpic] (-0.35,0)
    .. controls (-0.15,0.25) and (0.15,0.25) .. (0.35,0)
    .. controls (0.15,-0.15) and (-0.15,-0.15) .. (-0.35,0);

  \node at (0,0.32) {$e_2$};
  \node at (0,-0.28) {$e_1$};
\end{scope}

\begin{scope}[shift={(5.4,1.8)}]
  \fill[cone] (-1.25,0) rectangle (1.25,1.25);

  \draw[axis] (0,0) -- (0,1.35) node[above] {$e_1$};
  \draw[axis] (0,0) -- (1.35,0) node[right] {$e_2$};
  \draw[axis] (0,0) -- (-1.35,0) node[left] {$t$};
  \draw[axis] (0,0) -- (0,-0.35);

  \coordinate (A) at (.7, .7);
  \coordinate (E) at (.3, -.3);
  \draw[<->, dashed] ($(A) + (E)$) -- ($(A) - (E)$);
\end{scope}

\begin{scope}[shift={(.5,1.8)}]
  \fill[cone] (-1.25,0) rectangle (1.25,1.25);

  \draw[axis] (0,0) -- (0,1.35) node[above] {$e_1$};
  \draw[axis] (0,0) -- (1.35,0) node[right] {$e_2$};
  \draw[axis] (0,0) -- (-1.35,0) node[left] {$t$};
  \draw[axis] (0,0) -- (0,-0.35);

  \coordinate (A) at (.7, .7);
  \coordinate (E) at (.3, -.3);
  \draw[<->, dashed] ($(A) + (E)$) -- ($(A) - (E)$);
  \draw[purplearrow] (0,0) -- (1.25,1.25);
\end{scope}

\end{tikzpicture}

\caption{The Artin fans associated to the log blowup $\tilde B \to B$. }
\label{fig:tropicalAFblowupdr}
\end{figure}

We are ready for our example.

\begin{example}
    Suppose $g = 1$, $n = 2$, and $a_1 = - a_2 = 2$. Take $B = \Ms[1, 2]$ with universal curve $C$ and marked points $p, q$ and let $P$ be the log line bundle associated to $\OO_C(2p - 2q)$ on the universal stable curve. Then $\dr = \dr[{\OO_C(2p - 2q)}]$ is the subfunctor of curves $C$ on which $\OO_C(2p - 2q)$ is fiberwise trivial as a log line bundle, meaning that it is associated to a piecewise linear function. 

    The Artin fan of $B$ consists of two cones joined along a ray, depicted in Figure \ref{fig:tropicalAFblowupdr}. One is $\af{}^2$ and the other is the stack quotient $[{\rm Sym}^2 \af{}] = \bra{\af{}^2/\ZZ/2}$. Define a subdivision of $[{\rm Sym}^2 \af{}]$ by descending the diagonal subdivision of $\af{}^2$ at $(1, 1)$ along the $\ZZ/2$ quotient. Let $\tilde B \to B$ be the pullback of this subdivision depicted in Figure \ref{fig:blowupstrictdr}, which is the subfunctor of $B$ on which the unordered pair of lengths parameterized by $[{\rm Sym}^2 \af{}]$ are comparable. 

    We claim $\dr[{\OO_C(2p - 2q)}] \to B$ factors as a strict closed immersion into $\tilde B$. Note that $\tilde B \to B$ is an isomorphism away from a single closed stacky point corresponding to the union of two rational curves along two nodes $C_0 = D \cup E$ considered above. As before, let $s \to B$ be a log geometric point with image this closed stacky point. 

    It suffices to show that if $s \to B$ factors through $\dr[{\OO_C(2p - 2q)}]$, it factors through $\tilde B$. But Lemma \ref{lem:DRcomparableedgelengths} shows the edge lengths are comparable, so $s$ indeed factors through the blowup $\tilde B$. The map $\dr[{\OO_C(2p - 2q)}] \to \tilde B$ is a strict closed immersion because the log structure has a single generator which is pulled back from the exceptional divisor of $\tilde B \to B$. % Some details omitted. 
\end{example}

There is always a log alteration $\tilde B \to B$ such that the f.s.\ pullback of $\dr$ is a strict closed substack \cite[Proposition 3.7]{kthylogprodfmla}.

\subsection{The kernel of the Abel-Jacobi map}

Recall the log Picard group of \cite{logpic}: 
\begin{definition}
    Let $C/S$ be a prestable, \emph{vertical} log curve $C = \vert C$. The \emph{log Picard stack} $\LogPic_{C/S}$ consists of $\Glog$-torsors on $C$ with bounded monodromy. The \emph{log Picard sheaf} $\logpic_{C/S}$ is the sheaf of equivalence classes of $\Glog$-torsors $P$ on $C$ with bounded monodromy, where two torsors  are equivalent $P_1 \sim P_2$ if, locally in $S$, there is a $\Glog$-torsor $Q$ on $S$ and an isomorphism 
    \[
        Q|_C \otimes P_1 \simeq P_2.
    \]
\end{definition}

\begin{remark}\label{rmk:logpicgerbeband}
    The natural map $\LogPic_{C/S} \to \logpic_{C/S}$ is a gerbe banded by $\Glog$. There is a short exact sequence \cite[Theorem D]{logpic} 
    \begin{equation}\label{eqn:logpicses}
        1 \to \pic^{[0]}_{C/S} \to \logpic_{C/S} \to {\rm TroPic}_{C/S} \to 0,
    \end{equation}
    where 
    \begin{itemize}
        \item $\pic^{[0]}_{C/S}$ is the ``multidegree zero'' part of the Picard sheaf where each component of a geometric fiber has degree zero, a semiabelian scheme, and 
        \item ${\rm TroPic}_{C/S}$ is the tropical Picard functor, a tropical abelian variety\footnote{Think of a real torus $\RR^n/\ZZ^n$ or a family of such degenerating over a cone complex.} which contains obstructions to a given log line bundle coming from a genuine line bundle on $C/S$. 
    \end{itemize}
    There are stacky variants of the above which we underline similarly without mention. 
\end{remark}

We are interested in the universal log Picard group, that of the universal log curve $\cal U \to \Ms$. This log curve is not vertical; write $\vert{\cal U} \to \Ms$ for the verticalization of $\cal U$ without the log structure at the marked points. 

\begin{definition}\label{def:LP}
    Let $\LP \coloneqq \LogPic_{\vert{\cal U}/\Ms}$ and $\lp \coloneqq \logpic_{{\vert {\cal U}}/\Ms}$ be the log Picard stack and sheaf of the universal \emph{vertical} curve $\vert{\cal U} \to \Ms$. Write $\LPp, \lpp$ for the analogous constructions for the universal prestable vertical curve $\vert{{({\cal U}^{\rm ps})}}/\Mp$.
\end{definition}

\begin{example}\label{ex:tropicgenusg=0}
    The tropical Jacobian for genus $g = 0$ is the trivial family of abelian varieties ${\rm TroPic}_{\cal U/\Ms[0, n]} = \Ms[0, n]$ by definition \cite[Definition 3.6.1]{logpic}. This is because the tropicalization $\scr X$ of a genus zero curve is a tree and so $H_1(\scr X) = 0$.
    Thus 
    \[
        \lp[0, n] = \pic^{[0]}_{\cal U/\Ms[0, n]} = \Ms[0, n].
    \]
\end{example}

\begin{example}\label{ex:tropicalabelianvarietytropic}
    Take $g = n = 1$ and work over $\ZZ[\frac{1}{2}, \frac{1}{3}]$ so that $\Ms[1, 1] = \PP(4, 6)$
    is the compactification of the $j$-line parameterizing elliptic curves. Then $\lp[1, 1]$ is birational to the universal curve $\Ms[1, 2] \to \Ms[1, 1]$ and isomorphic away from the boundary corresponding to the nodal curve. 

    The tropical Jacobian of the nodal curve $N/S$ (depicted in Figure \ref{fig:tropicalabelianvariety}) comes from the bilinear pairing 
    \[
        \ZZ \times \ZZ \to \gp \NN = \ZZ; \qquad (x, y) \mapsto xy.
    \]
    Choose an element $p \in P$ as a map $\NN \to P$ or a map of affine monoschemes $\Spec P \to \Spec \NN$ \cite{ogusloggeom} and write $\bar P = P/\num{p}$ for the initial sharp monoid under $P$ in which $p$ maps to zero. The tropical Jacobian takes the value 
    \[
        {\rm TroPic}_{N/S}(\Spec P) \coloneqq \dfrac{\ker\pra{\gp P \to \gp {\bar P}}}{\ZZ \cdot p}.
    \]
    On the element $1 \in \RR_{\geq 0}$ for example, we have the quotient 
    \[
        {\rm TroPic}_{N/S}(\Spec {\RR_{\geq 0}}) \coloneqq \dfrac{\RR}{\ZZ} = S^1.
    \]
    This tropical Jacobian is a family of real circles over a monoscheme in general, and is extended to a functor on log schemes in the usual way \cite{tropicalcurvesmodulicones}. The circle can be thought of as a blow-down of the tropicalization of the nodal cubic curve. 
\end{example}

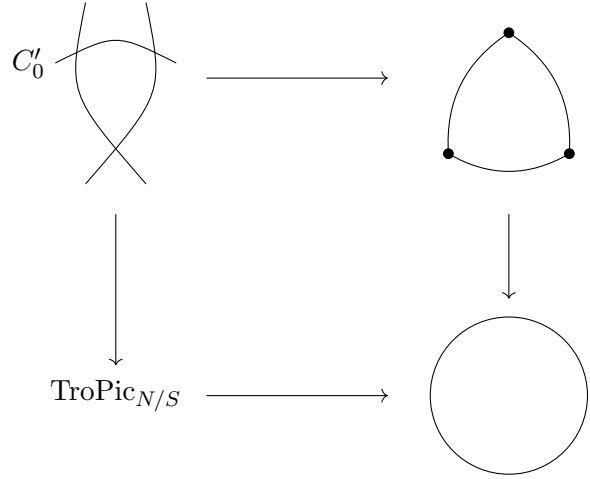
\begin{figure}
    \centering
    \begin{tikzpicture}[scale = .8]
        \node[left] at (-.5, 2){$C'_0$};
        \draw (0, 0) .. controls (1.3, 1.5) .. (1, 3);
        \draw (1, 0) .. controls (-.3, 1.5) .. (0, 3);
        \draw (-.5, 2) .. controls (.5, 2.5) .. (1.5, 2);
        \begin{scope}[shift = {(6, 0)}]
            \coordinate (A) at (1, 2.5);
            \coordinate (B) at (0, .5);
            \coordinate (C) at (2, .5);
            \draw[bend left] (A) to (C);
            \draw[bend right] (A) to (B);
            \draw[bend right] (B) to (C);
            \filldraw (A) circle (.08);
            \filldraw (B) circle (.08);
            \filldraw (C) circle (.08);
        \end{scope}
        \begin{scope}[shift = {(6, -5)}]
            \draw (1, 1.5) circle (1.3);
        \end{scope}
        \node at (.5, -3.5){${{\rm TroPic}_{N/S}}$};

        \draw[->] (.5, -.5) to (.5, -3);
        \draw[->] (7, -.5) to (7, -1.9);
        \draw[->] (2, 1.75) to (5, 1.75);
        \draw[->] (2, -3.5) to (5, -3.5);
    \end{tikzpicture}
    \caption{The tropical abelian variety from Example \ref{ex:tropicalabelianvarietytropic} is the ``blow down'' in log geometry of a nodal cubic or the curve $C'_0$ from Figure \ref{fig:divisorlimit}. }
    \label{fig:tropicalabelianvariety}
\end{figure}

For a line bundle $L$ written as $L = \OO(D)$ or $L = \OO(\alpha)$, we write $\gp M(D)$ or $\gp M(\alpha)$ for the associated $\Glog$-torsor. We can now compactify the Abel-Jacobi map. 

\begin{definition}\label{def:logAJmap}
    Let $\vec a \in \ZZ^n_0$ be a vector of integers which sums to zero $\sum a_i = 0$. Define a map $\aj_{\vec a} : \Ms \to \LP$ by endowing a log curve $C$ with the $\Glog$-torsor
    \begin{equation}\label{eqn:AJloglinebundle}
        \gp M_{\vert C}(\sum a_i p_i) \coloneqq \OO^*_C(\sum a_i p_i) \times^{\OO^*_C} \gp M_{\vert C} \qquad \in H^1(\vert C, \gp M_{\vert C})
    \end{equation}
    associated to the line bundle $\OO_C(\sum a_i p_i)$ on its verticalization $\vert C$. Write $e \coloneqq \aj_{\vec 0} : \Ms \to \LP$ for the trivial section, the map sending a curve $C$ to the trivial log line bundle on $C$. 
\end{definition}

The map $e : \Ms \to \lp$ is strict and factors through the dense open $\pic^{[0]}_{\cal U/\Ms}$.

\begin{remark}[Comparison with usual Abel-Jacobi map]
Recall that the (usual) Abel-Jacobi map is a map 
\[
    \aj_{\vec a}: {\cal M}_{g,n} \rightarrow \Pic, \quad \aj_{\vec a}(C, p_i) = \OO_C(\sum a_ip_i),
\]
where $\Pic$ is the ordinary Picard stack of the universal smooth curve. It does not extend to a multidegree-zero line bundle on $\Ms$. 

Definition~\ref{def:logAJmap} gives a compactified map to $\LP$. We have a diagram
    \[
    \begin{tikzcd} 
        {\cal M}_{g,n} \ar[r] \ar[d] \lpbstrict & \Pic^{[0]} \ar[r] \ar[d] \lpbstrict          &{\cal M}_{g,n} \ar[d] \\
        \Ms \ar[r, "{\aj_{\vec a}}", swap]         &\LP \ar[r]        &\Ms.
    \end{tikzcd}
    \]
This can be seen by the short exact sequence \eqref{eqn:logpicses} and the vanishing of $\underline{\rm TroPic}$ on the locus ${\cal M}_{g,n} \subseteq \Ms$ where the log structure vanishes. 
\end{remark}

\begin{definition}\label{def:DAJkernel}
    Consider $P/C/B$ as in Situation \ref{sit:logsmoothbase}. Define $\dd \subseteq \ddp$ as the f.s.\ pullback 
    \begin{equation}\label{eqn:AJkernel}
    \begin{tikzcd}
        \dd \ar[r] \ar[d] \lpbstrict         &\ddp \ar[r] \ar[d] \lpb        &B \ar[d]         \\
        \Ms \ar[r]         &\Mp \ar[r, "{\aj_{\vec a}}", swap]        &\lpp.
    \end{tikzcd}
    \end{equation}
    In the case $\vec a = \vec 0$, the map $e : \Mp \to \lpp$ factors through ${\rm Pic}^{[0], {\rm ps}}$.
\end{definition}

A map $S \to \ddp$ entails a log line bundle $Q$ on $S$ together with an isomorphism 
\begin{equation}\label{eqn:pbdrloglinebundles}
    \OO_{C_S}(\sum a_i p_i) \times^{\OO_C^*} Q|_{\vert C_S} \longsimeq P|_{\vert C_S}. 
\end{equation}

\begin{remark}
    We can twist $P$ instead of $Q$ in \eqref{eqn:pbdrloglinebundles} to obtain an isomorphism 
    \[
        \ddp[P, \vec a] \longsimeq \ddp[P(-\sum a_i p_i), \vec 0]
    \]
    and reduce to the case of $\vec a = \vec 0$. 
\end{remark}

\begin{remark}\label{rmk:AJBmapGlogtorsordifference}
    The maps $\aj_{\vec a}, B \to \lpp$ in \eqref{eqn:AJkernel} both factor through the $\Glog$-gerbe $\LPp \to \lpp$. We can expand \eqref{eqn:AJkernel} to an f.s.\ cartesian diagram 
    \[
    \begin{tikzcd}
        \ddp \ar[r] \ar[dd] \lpb       &B \times B\Glog \ar[r] \ar[d] \lpb        &B \ar[d]      \\
                &\LPp \times B\Glog \ar[r] \ar[d] \lpb         &\LPp \ar[d]       \\
        \Ms \ar[r]         &\LPp \ar[r]       &\lpp.
    \end{tikzcd}
    \]
    On $S$-points of $\ddp$, the map $\ddp \to \LPp \times B\Glog \to B\Glog$ parameterizes the $\Glog$-torsor $Q$ which satisfies
    \[
        Q|_{\vert C_S}(\sum a_i p_i) \longsimeq P|_{\vert C_S}. 
    \]
    That is, $Q$ on $S$ pulls back to the $\Glog$-torsor of isomorphisms $\gp M_{\vert C_S}(\sum a_i p_i) \longsimeq P|_{\vert C_S}$ along $\vert C_S \to S$. The map $B \times B\Glog \to \LPp \times B\Glog \to \LPp$ sends $(P/C/B, Q)$ to the $\Glog$-torsor of isomorphisms 
    \[
        \underline {\rm Isom}(Q|_{\vert C_S}, P|_{\vert C_S}),
    \]
    which is the difference $P|_{\vert C_S} - Q|_{\vert C_S}$ in the group law on $\Glog$-torsors on $\vert C_S$. 
\end{remark}

\begin{remark}\label{rmk:drpbbaseB}
    One can also factor the vertical map $B \to \lpp$ through $\logpic_{C/B}$ to get a diagram 
    \[
    \begin{tikzcd}
        \Dp[P, \vec a] \ar[r] \ar[d] \lpb         &B \ar[d, "e"]      \\
        B \ar[r] \ar[d] \lpb       &\logpic_{C/B} \ar[r] \ar[d] \lpb      &B \ar[d]      \\
        \Mp \ar[r, "{\aj_{\vec a}}"]         &\lpp \ar[r]        &\Mp.
    \end{tikzcd}
    \]
    The lower left square can be seen to be Cartesian as it is commutative and the outer bottom rectangle is Cartesian. So all squares in the diagram are f.s.\ pullbacks and the upper left corner may be taken as the definition of $\Dp[P, \vec a]$. 
\end{remark}

\begin{corollary}\label{cor:pDRisD}
    There is an isomorphism 
    \[
        \Phi : \drp \longsimeq \ddp
    \]
    over $B \times B\Glog$ which restricts to an isomorphism $\dr \longsimeq \dd$. Each of $\dd, \ddp, \dr, \drp$ is representable by a log algebraic stack. 
\end{corollary}

\begin{proof}
    The stability conditions on the two spaces are the same, so they can be ignored. Write $Q$ for the $\Glog$-torsor corresponding to the map to $B \times B\Glog$. A point of $\ddp$ entails an isomorphism \eqref{eqn:pbdrloglinebundles} on $\vert C$, which Lemma \ref{lem:Glogtorsorverticalvsstandard} equates with an isomorphism \eqref{eqn:drpoints} on $C$. Representability results from the fiber product \eqref{eqn:AJkernel} and representability of the diagonal of $\lpp$ by \cite[Theorem 4.12.1]{logpic}.

\end{proof}

\subsubsection{Compactified Jacobians}

Neither the log Picard stack nor the log Picard sheaf are representable by schemes, algebraic spaces, or algebraic stacks, even those equipped with log structures \cite{logpic}. Instead of a smooth cover by a scheme, they admit \emph{log smooth} covers by a log scheme. Such objects are called ``logarithmic spaces'' or ``logarithmic stacks'' in the literature in parallel with the terms ``algebraic space'' and ``algebraic stack.'' 

We will need specific log alterations ${\cal J} \to \lp$ of the log Picard sheaf called ``compactified Jacobians'' \cite{LogDRCycle}.

\begin{definition}\label{def:quasistablemodeladmissible}
    A \emph{quasi-stable model} for a log curve $C/S $ is a morphism of log curves $f:C' \to C$ such that the geometric fibers of $f$ are points or smooth rational curves, called exceptional curves, meeting the remainder of the curve at 2 points. If $C' \to C$ is a quasi-stable model and $\mathcal{F}$ is a line bundle on $C'$, then $\mathcal{F}$ is \emph{admissible} if it has degree 1 on every exceptional curve.
\end{definition}

\begin{definition}
    Let $\qpic$ be the moduli problem on log schemes $S$ parameterizing:
    \begin{itemize}
        \item A log curve $\pi: C \to S$;
        \item A quasi-stable model $f: C' \to C $ for $\pi$;
        \item An admissible line bundle $\mathcal{F}$ of \emph{total} degree 0 on $C'$.\footnote{Total degree zero means the degrees on each component of a geometric fiber sum to zero. Multidegree zero means furthermore that the degree on each component is zero.} 
    \end{itemize}
\end{definition}

\begin{remark}
    The moduli problem $\qpic$ contains the functor $\pic^{[0]}$ of multidegree zero line bundles of Remark \ref{rmk:logpicgerbeband} as a log \'etale monomorphism. On this locus, the quasistable model $C' \to C$ is the identity and there are no exceptional curves. 
\end{remark}

\begin{definition} \label{d:2.20}
     A \emph{compactified Jacobian} over $\Ms$ is an open substack $\mathcal{J} \subset \qpic$ such that
    \begin{itemize}
        \item $\mathcal{J}$ is proper over $\Ms$,
        \item $\mathcal{J}$ contains the subfunctor $\pic^{[0]}$ of line bundles of multidegree zero of Remark \ref{rmk:logpicgerbeband} as a strict open substack,
        \item If $(C' \to C, \mathcal{F})$ correspond to a geometric point $S \to \cal J$,
        either $\mathcal{F}$ is trivial or $\Gamma(C',\mathcal{F}) = 0$.
    \end{itemize}
\end{definition}

\begin{remark}\label{rmk:compactjacobiansexist}
    A stability condition $\theta$ is a systematic choice of how to allow line bundles to degenerate in families of curves. We extend a stability condition on $C$ to a quasistable model $C' \to C$ by assigning zero to each exceptional curve. If a stability condition $\theta$ is ``small'' and ``nondegenerate,'' the moduli problem of $\theta$-stable line bundles is a compactified Jacobian \cite{Holmes_2018}. In the space of stability conditions, nondegenerate ones are open and dense \cite[Section 5]{KassPaganistability} and small ones are nonempty and open, 
    so there exists a small, nondegenerate stability condition and hence a compactified Jacobian. 

\end{remark}

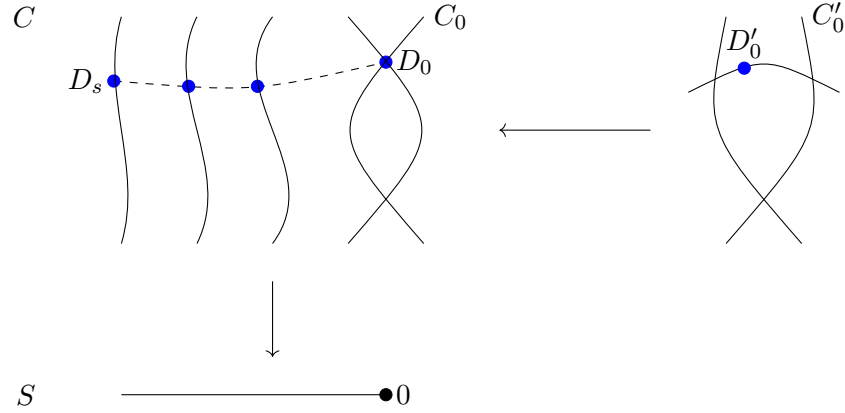
\begin{figure}
    \centering
    \begin{tikzpicture}
        \draw[dashed] (-.1, 2.15) .. controls (1.75, 2) .. (3.5, 2.4);
        \node[left] at (-.1, 2.15){$D_s$};
        \node[right] at (3.5, 2.4){$D_0$};
        \filldraw[blue] (-.1, 2.15) circle (.08);
        \filldraw[blue] (.89, 2.08) circle (.08);
        \filldraw[blue] (1.8, 2.08) circle (.08);
        \filldraw[blue] (3.5, 2.4) circle (.08);
        \node[left] at (-1, 3){$C$};
        \draw (0, 0) .. controls (.3, 1) and (-.3, 2) .. (0, 3);
        \begin{scope}[shift = {(1, 0)}]
        \draw (0, 0) .. controls (.5, 1) and (-.5, 2) .. (0, 3);
        \end{scope}
        \begin{scope}[shift = {(2, 0)}]
        \draw (0, 0) .. controls (.75, 1) and (-.75, 2) .. (0, 3);
        \end{scope}
        \begin{scope}[shift = {(3, 0)}]
        \node[right] at (1, 3){$C_0$};
        \draw (0, 0) .. controls (1.3, 1.5) .. (0, 3);
        \draw (1, 0) .. controls (-.3, 1.5) .. (1, 3);
        \end{scope}
        \draw[->] (2, -.5) to (2, -1.5);
        \draw (0, -2) to (3.5, -2);
        \filldraw (3.5, -2) circle (.08);
        \node[right] at (3.5, -2){$0$};
        \node[left] at (-1, -2){$S$};

        \draw[->] (7, 1.5) to (5, 1.5);
        
        \begin{scope}[shift = {(8, 0)}]
        \node[right] at (1, 3){$C'_0$};
        \draw (0, 0) .. controls (1.3, 1.5) .. (1, 3);
        \draw (1, 0) .. controls (-.3, 1.5) .. (0, 3);
        \draw (-.5, 2) .. controls (.5, 2.5) .. (1.5, 2);
        \filldraw[blue] (.24, 2.32) circle (.08);
        \node[above] at (.24, 2.32){$D'_0$};
        \end{scope}
    \end{tikzpicture}
    \caption{A family of divisors $D_s$ on a curve $C/S$. These smooth points $D_s \in C_s$ correspond to a line bundle $L_s$ on $C_s$. Taking $D_0 \in C_0$ to be the limit of the points $D_s$, we see it is not a smooth point and does not correspond to a line bundle on $C_0$. Blow up the point $D_0 \in C$ to obtain a quasistable model $C' \coloneqq Bl_{D_0} C \to C$ which admits a limiting divisor $D'_0 \in C'_0$ of the family $D_s$. }
    \label{fig:divisorlimit}
\end{figure}

\subsubsection{Generalization to higher rank}\label{higherranksection}

We generalize $\dr$, and $\dd$ to higher rank. 

\begin{situation}\label{sit:logsmoothbasehigherrank}
    Let $B_r$ be a log smooth log algebraic stack with a prestable standard log curve $C/B_r$ and consider an $r$-tuple of $\Glog$-torsors $\vec P \coloneqq (P_i)_i$ on $\vert C$. 
\end{situation}

\begin{definition}
    Let $B_r$ be as in Situation \ref{sit:logsmoothbasehigherrank}. Write $\drp[\vec P]$ for the moduli problem on log schemes $S$ over $B_r \times B\Glog^r$ defined by $r$-tuples of isomorphisms
    \[
    \left\{
    \begin{tikzcd}
                &\drp[\vec P] \ar[d]       \\
        S \ar[r] \ar[ur, dashed]       &B_r \times B\Glog^r
    \end{tikzcd}
    \right\}
    \coloneqq 
    \{r\text{ isomorphisms } Q_i|_{C_S} \longsimeq P_i|_{C_S}\}.
    \]
    Let $\dr[\vec P] \subseteq \drp[\vec P]$ be the open substack on which $C$ is a stable curve.

    Let $A = (a_{ij}) \in {\rm Mat}_{n \times r}(\ZZ)$ be an $n \times r$ matrix of integers such that the columns sum to zero $\sum_i a_{ik} = 0$. Define the higher-rank Abel-Jacobi map 
    \[
        \aj_A : \Ms \to \LP^{\times_{\Ms} r}
    \]
    to the fiber product of $r$ copies of $\LP$ over $\Ms$ by endowing the verticalization $\vert C$ of a log curve $C$ with $r$ $\Glog$-torsors 
    \[
        \left(\gp M(\sum a_{i1} p_i), \cdots, \gp M(\sum a_{ir} p_i)\right) \qquad \in H^1(\vert C, \gp M_{\vert C}).
    \]
    Let $ \dd[\vec P, A] \subseteq \ddp[\vec P, A]$ be the f.s.\ pullback 
    \[
    \begin{tikzcd}
        \dd[\vec P, A] \ar[r] \ar[d] \lpbstrict     &\ddp[\vec P, A] \ar[r] \ar[d] \lpb       &B_r \ar[d]      \\
        \Ms \ar[r]         &\Mp \ar[r, "{\aj_A}"]        &{\lpp{}^{\times_{\Ms} r}}.
    \end{tikzcd}
    \]

    An $S$-point of $\dd[\vec P, A]$ is a log stable curve $C/S$ with genus $g$ and $n$ marked points such that each log line bundle $P_k (-\sum a_{ik} p_i)$ is pulled back from the base $S$. 
    
\end{definition}

One can prove Corollary \ref{cor:pDRisD} in higher rank as well, which we leave to the reader. 

\begin{proposition}
    There is an isomorphism $\drp[\vec P, A] \longsimeq \ddp[\vec P, A]$ over $B_r \times B\Glog^r$ which identifies the open stable substacks $\dr[\vec P, A] \longsimeq \dd[\vec P, A]$. 
\end{proposition}

\subsection{Specialization to the case $P = \OO_C$}\label{ss:triviallinebundle}

Let $\vec a \in \ZZ^n_0$ be a vector of integers summing to zero $\sum a_i = 0$. In this section, we take $B = \Ms$ and equip its universal curve $C/B$ with the trivial log line bundle. We compare with another version of the double ramification space more common in the literature. 

\begin{remark}\label{rmk:canonicaltrivpsiPLsumsections}
    Let $\psi_{\vec a} \in \Gamma(C, \gp {\bar M}_C)$ be the piecewise linear function with slope $a_i$ at the marked point $p_i$ and value $0$ on each edge. Then $\OO_C(\psi_{\vec a}) = \OO_C(\sum a_i p_i)$, so we have a canonical trivialization 
    \[\pra{\gp M_{\vert C}(\sum a_i p_i)}|_C \longsimeq \gp M_C\]
    on $C$. The log line bundle $\gp M_{\vert C}(\sum a_i p_i)$ need not be trivial on $\vert C$. 
\end{remark}

\begin{lemma}
    The stack $\drp[\OO_C, \vec a]$ with $P = \gp M_{\vert C}$ represents squares 
    \begin{equation}\label{eqn:pDRpoints}
    \begin{tikzcd}
        C \ar[r] \ar[d]       &\pt \ar[d]        \\
        S \ar[r, "Q", swap]       &B\Glog
    \end{tikzcd}
    \end{equation}
    with contact order $\vec a$. That is, $\drp[\OO_C, \vec a]$ parameterizes $\Glog$-torsors $Q$ on $S$ together with trivializations of the pullback $Q|_C$ to $C$ with contact order $\vec a$.  
\end{lemma}

\begin{proof}
    By definition, $\drp[\OO_C, \vec a]$ parameterizes $\Glog$-torsors $Q$ on $S$ with isomorphisms 
    \[
        \OO_C(\sum a_i p_i) \times^{\OO_C^*} Q|_{\vert C_S} \longsimeq \gp M_{\vert C_S}.
    \]
    Remark \ref{rmk:canonicaltrivpsiPLsumsections} provides a canonical trivialization 
    \[\gp M_C(\sum a_i p_i) \longsimeq \gp M_C,\]
    and Lemma \ref{lem:Glogtorsorverticalvsstandard} equates the above isomorphism with a trivialization fitting in the square \eqref{eqn:pDRpoints} with contact order $\vec a$. 
\end{proof}

\begin{proposition}
    The space $\drp[\OO_C, \vec a]$ is the (f.s.\ and ordinary) pullback of the Abel-Jacobi section $\aj_{\vec a}$ along the trivial section 
    \[
    \begin{tikzcd}
        \drp[\OO_C, \vec a] \ar[r] \ar[d] \lpbstrict         &\Mp \ar[d, "e"]        \\
        \Mp \ar[r, "{\aj_{\vec a}}", swap]         &\lpp.
    \end{tikzcd}
    \]
\end{proposition}

There are analogous results in higher rank.

\subsubsection{Line bundles coming from piecewise linear functions}

We highlight a variant $\DR$ of the log double ramification locus more common in the literature which admits a log alteration 
\[
    \DR \to \dr[\OO_C, \vec a].
\]

Let $C/S$ be a log curve and consider commutative squares 
\begin{equation}\label{eqn:DRpoints}
    \begin{tikzcd}
    C \ar[r, "\alpha"] \ar[d]       &\Gtrop \ar[d]         \\
    S \ar[r, "L", swap]       &B\GG_m.
    \end{tikzcd}
\end{equation}
Such a square entails:
\begin{enumerate}
    \item\label{it:DRcurve} a log prestable curve $C \to S$ with $n$ marked points $p_i$,
    \item a ``piecewise linear function'' $\alpha : C \to \Gtrop$ %with slope $a_i$ at the marked point $p_i \in C$,
    \item a line bundle $L$ on $S$, and 
    \item\label{it:DRisomlinebundle} an isomorphism $\varphi : \OO_C(\alpha) \simeq L|_C$.
\end{enumerate}

We need to impose an extra condition on the above data.

\begin{condition}\label{cond:DRtotalorderzero}
    Consider the above data over a geometric point $S$. The piecewise linear function $\alpha$ is determined by its set of values $\{\alpha(C_i)\} \in \Gamma(\gp{\bar M}_S)$ on the irreducible components $C_i \subseteq C$ and the vector $\vec a$ of its slopes at the points $p_i$. 
    
    Suppose that 
    \begin{enumerate}
        \item\label{it:DRtotalorder} the subset $\{\alpha(C_i)\} \subseteq \Gamma(\gp{\bar M}_S)$ of values on irreducible components $C_i \subseteq C$ is totally ordered in the natural order on $\Gamma(\gp{\bar M}_S)$, and 
        \item\label{it:DRminzero} the minimal value $\alpha(C_i)$ obtained is $0 \in \Gamma(\gp{\bar M}_S)$. 
    \end{enumerate}
\end{condition}

\begin{definition}\label{def:tDR}
    Let $\DRp[g, n]$ be the moduli problem on log schemes $S$ consisting of points \eqref{eqn:DRpoints} subject to Condition \ref{cond:DRtotalorderzero}. Let $\DR[g, n] \subseteq \DRp[g, n]$ be the open substack in which the curve $C/S$ is stable and 
    \[
        \DR \subseteq \DR[g, n], \qquad \DRp \subseteq \DRp[g, n]
    \]
    for the connected components on which the slope of $\alpha : C \to \Gtrop$ at the marked point $p_i$ is $a_i$. 
\end{definition}

There is a natural morphism
\begin{equation}\label{eqn:pDRDRlogaltn}
    \DR \to \dr[\OO_C, \vec a]
\end{equation}
given by appending the pullback square 
\[
\begin{tikzcd}
    \Gtrop \ar[r] \ar[d] \lpb      &\pt \ar[d]        \\
    B\GG_m \ar[r]      &B\Glog
\end{tikzcd}
\]
to a point \eqref{eqn:DRpoints} to get a point \eqref{eqn:pDRpoints}. 

\begin{proposition}\label{prop:DRlogaltnpDR}
    The natural morphism \eqref{eqn:pDRDRlogaltn} is a log blowup.
\end{proposition}

We need a lemma. 

\begin{lemma}
    Let $D' \to \drp[\OO_C, \vec a]$ be the subfunctor consisting of points \eqref{eqn:pDRpoints} such that, for each strict geometric point $\bar s \to S$ and any trivialization $Q|_{\bar s} \simeq \Glog$, the values of the piecewise linear function 
    \begin{equation}\label{eqn:localinducedgeometricfiberplfunction}
        C_{\bar s} \to \Glog \to \Gtrop
    \end{equation}
    on the components of $C_{\bar s}$ are totally ordered as in \eqref{it:DRtotalorder}. Then $D' \to \drp[\OO_C, \vec a]$ is a log blowup. 
\end{lemma}

\begin{proof}
    Consider a point \eqref{eqn:pDRpoints} with $S$ a log scheme. The condition is independent of the choice of trivialization $Q|_{\bar s} \simeq \Glog$ as well as choice of geometric point with the same image in $S$. 
    
    We claim the condition is also open. It is constructible, so it suffices to argue this locus is closed under generization. If $\bar s \rightsquigarrow \bar t$ is a specialization of geometric points in $S$, the cospecialization map 
    \[
        \Gamma({\bar t}, \gp{\bar M}_{S, {\bar t}}) \longrightarrow 
        \Gamma({\bar s}, \gp{\bar M}_{S, {\bar s}})
    \]
    sends the set of values on components of $C_{\bar t}$ to the set of values on components of $C_{\bar s}$ surjectively. So the total ordering condition is open in $S$. 

    We can assume $S$ is atomic with geometric point $\bar s \to S$ in the unique closed stratum of $S$. Then $D' \to \drp[\OO_C, \vec a]$ pulls back to the iterative log blowup of $S$ at all subsets of the values of the piecewise linear function \eqref{eqn:localinducedgeometricfiberplfunction} on the components of $C_{\bar s}$. So $D' \to \drp[\OO_C, \vec a]$ is a log blowup. 
\end{proof}

\begin{proof}[{Proof of Proposition \ref{prop:DRlogaltnpDR}}]
    We prove the log blowup $D' \to \drp[\OO_C, \vec a]$ is isomorphic to $\DR$. Let \eqref{eqn:pDRpoints} be a point of $\drp[\OO_C, \vec a]$ that lies in $D'$. Write $\bar Q$ for the $\Gtrop$-torsor associated to the $\Glog$-torsor $Q$ on $S$. Because the values of $C \to \bar Q|_C$ on the components of the fibers of $C$ are totally ordered, the min and max sections are fiberwise well defined and compatible under specialization. These give sections of $\bar Q$ defined on $S$, which trivialize $\bar Q$. We choose the min section. 

    We therefore have dashed factorizations 
    \[
    \begin{tikzcd}
        C \ar[rr, bend left] \ar[r, dashed] \ar[d]      &\Gtrop \ar[r] \ar[d] \lpb         &\pt \ar[d]        \\
        S \ar[dr, bend right, dashed] \ar[r, dashed]       &B\GG_m \ar[r] \ar[d] \lpb         &B\Glog \ar[d]         \\
            &\pt \ar[r]        &B\Gtrop.
    \end{tikzcd}
    \]
    Because we chose the min section to trivialize $\bar Q$, the minimal value on a component of a fiber of $C$ is always zero as in Condition \ref{cond:DRtotalorderzero}. So the point $S \to D'$ canonically factors through $\DRp \to \drp[\OO_C, \vec a]$. As $D'$ and $\DRp$ are subfunctors of $\drp[\OO_C, \vec a]$ which factor through each other, they are isomorphic. 
\end{proof}

\begin{remark}\label{rmk:comparisonwithrubber}
    One can replace \eqref{eqn:DRpoints} by commutative squares 
    \[
    \begin{tikzcd}
        C \ar[r] \ar[d]       &\bra{\PP^1/\GG_m} \ar[d]        \\
        S \ar[r]       &B\GG_m
    \end{tikzcd}    
    \]
    satisfying the same Condition \ref{cond:DRtotalorderzero} to get another space $D$. As the morphism $\bra{\PP^1/\GG_m} \to \Gtrop$ is a log alteration, so is the map $D \to \DR$ by the same argument as in the absolute case \cite{birationalinvarianceabramovichwise}. Such commutative squares are what is typically meant by ``rubber stable maps to $\PP^1$,'' omitting the typical accordions/expansions of the target. See \cite{RdC} for comparison with accordions. 
\end{remark}

\section{The double ramification class in \texorpdfstring{$K$}{K}-theory and the product formula} \label{s:3}

\subsection{The log perfect obstruction theories}

A family of prestable maps is stable on a strict open subscheme of the base, so we ignore stability while discussing log perfect obstruction theories. We continue to use the language of Situation \ref{sit:logsmoothbase}. 

We use (a logarithmic version of) the language of obstruction groupoids \cite[\S 5]{relorbgwobstructiongroupoids}, \cite{obstthies}. Fix a strict squarezero nilimmersion $S \subseteq S'$ of affine\footnote{This is for convenience.} log schemes defined by a quasicoherent ideal $J$ 
\begin{equation}\label{eqn:squarezeroextn}
0 \to J \to \OO_{S'} \to \OO_S \to 0.
\end{equation}
One obtains a ``usual'' log perfect obstruction theory in the sense of Behrend-Fantechi \cite{intrinsic} or \cite{herrthesis} by plugging in $J = \OO_S$. We will unpack this %at the end of the proof 
for the reader's convenience.

\subsubsection{Usual log perfect obstruction theory on log stable maps}

Consider log stable maps $\Ml(X)$ to a log smooth target $X$ and let $\pi : \cal C \to \Ml(X)$ be the universal curve. Recall the usual log perfect obstruction theory \cite{intrinsic, loggw, herrthesis} given by 
\[
    \frak E \coloneqq h^1/h^0\pra{R\pi_*\pra{\lkah{X}|_{\cal C} \otimes \omega_\pi}}.
\]
This represents the stacky Weil restriction of the log tangent space $\Tl{X}|_{\cal C}$ of $X$ along $\pi$. 

Consider a square 
\[
\begin{tikzcd}
    S \ar[r] \ar[d]       &S' \ar[d] \ar[dl, dashed]         \\
    \Ml(X) \ar[r]      &\Mp
\end{tikzcd}
\]
with corresponding curves $C \subseteq C'$ over $S \subseteq S'$. The obstruction to lifting given by the above log perfect obstruction theory is given by the lifting problem 
\begin{equation}\label{eqn:usualliftingproblem}
    \begin{tikzcd}
    C \ar[r, "f"] \ar[d]      &X      \\
    C', \ar[ur, dashed]
\end{tikzcd}
\end{equation}
viewed as a squarezero algebra extension 
\[\Exal(f^{-1}\OO_X, J) = \Ext^1(\lkah{X}, J).\]
If $J = \OO_S$, the group of algebra extensions $\Exal(f^{-1} \OO_X, \OO_C)$ is the group of torsors for $\Tl{X}|_C$.

\subsubsection{Log stable maps}\label{sss:logstabmapslogPOT}

Now specialize to the case of interest:
\[
\begin{tikzcd}
    S \ar[r] \ar[d]       &S' \ar[d] \ar[dl, dashed]        \\
    \drp \ar[r, "Q", swap]       &B \times B\Glog.
\end{tikzcd}
\]
One can generalize the previous section to a log smooth morphism $X \to Y$ \cite{RdC}, but we do not need the general case. 

The solid arrows give 
\begin{itemize}
    \item A prestable curve $C' \to S'$ with restriction $\pi : C \coloneqq C' \times_{S'}^{\msout{\ell}} S \to S$,
    \item A $\Glog$-torsor $Q'$ on $S'$ with restriction $Q \coloneqq Q'|_S$, and 
    \item An isomorphism $\varphi : Q|_C \longsimeq P|_{C}$ with contact order $\vec a$%
    \footnote{The discrete data $\vec a$ are locally constant, so they may also be ignored for the sake of computing log perfect obstruction theories.}. 
\end{itemize}
The dashed arrow entails an isomorphism $\varphi' : Q'|_{C'} \longsimeq P|_{C'}$ which extends $\varphi$. 
Lemma \ref{lem:Glogtorsorverticalvsstandard} \eqref{it:CvsCvert} equates these data with an isomorphism 
\[
    \varphi : Q|_{\vert C}(\sum a_i p_i) \longsimeq P|_{\vert C}
\]
which must be lifted to an analogous isomorphism $\varphi'$ on $\vert{{C'}}$. 

We equip $\drp \to B \times B\Glog$ with the standard log perfect obstruction theory of log stable maps. The exact sequence \cite[Proposition IV.2.1.2]{ogusloggeom}
\[
1 \to 1 + J|_C \to \gp M_{\vert{C'}} \to \gp M_{\vert C} \to 1
\]
and the isomorphism $\varphi$ show the set of isomorphisms $\varphi'$ form a $J|_C$-torsor on $C$.%
\footnote{The multiplicative group $1 + J|_C$ is isomorphic to the additive group $J|_C$ because $C \subseteq C'$ is a squarezero extension.}
Choosing an isomorphism $\varphi'$ restricting to $\varphi$ is the same as trivializing this $J|_C$-torsor. This is precisely the obstruction in the usual log perfect obstruction theory \eqref{eqn:usualliftingproblem}, which lies in 
\[
    \Ext^1(\lkah{\pt/B\Glog}|_C, J|_C) = \Ext^1(\OO_C, J|_C) = R^1\pi_* J|_C.
\]

The obstruction groupoid is the sheaf $R^1 \pi_* J|_C$ with isomorphisms given by equalities and the obstruction class is the set of sections of the $\Glog$-torsor on $C'$
\begin{equation}\label{eqn:I'isomsheafQtwistedisomP}
    I' \coloneqq {\rm Isom}(Q'|_{\vert C}(\sum a_i p_i), P'|_{\vert C}) \longsimeq 
    P|_{\vert C}(-\sum a_i p_i) \otimes -Q|_{\vert C} 
\end{equation}
which restrict to the chosen section $\varphi \in I'|_C$.

\subsubsection{The kernel of the Abel-Jacobi section}\label{sss:AJkernellogPOT}

We give the map $\ddp \to B \times B \Glog$ the log perfect obstruction theory coming from the f.s.\ Cartesian diagram
\[
\begin{tikzcd}
    \ddp \ar[r] \ar[d] \lpb      &B \times B \Glog \ar[d] \ar[r] \lpb       &B \ar[d]        \\
    \Mp \ar[r, "\aj_{\vec a}", swap]         &\LPp \ar[r]      &\lpp.
\end{tikzcd}
\]
Use log smoothness of $\Mp$ and $\LPp$ to deduce that the log cotangent complex of $\aj_{\vec a}$ is perfect in degrees $[-1, 0]$. We could reduce to the case of the trivial section $\vec a = \vec 0, \aj_{\vec 0} = e$ by replacing $P$ by $P(-\sum a_i p_i)$. Recall from Remark \ref{rmk:AJBmapGlogtorsordifference} that the map $B \times B\Glog \to \LPp$ sends $(P/C/B, Q)$ to the difference $P - Q|_{\vert C} = \underline{\rm Isom}(Q|_{\vert C}, P)$ in the group law on $\Glog$-torsors on $\vert C$.

A solid diagram 
\[
\begin{tikzcd}
    S \ar[r] \ar[d]       &S' \ar[d] \ar[dl, dashed]         \\
    \Mp \ar[r, "\aj_{\vec a}", swap]         &\LPp
\end{tikzcd}
\]
entails 
\begin{itemize}
    \item A prestable curve $C' \to S'$ with restriction $\pi : C \coloneqq C' \times_{S'}^{\msout{\ell}} S \to S$,
    \item A $\Glog$-torsor $N'$ on $\vert {C'}$, and
    \item An isomorphism of the restriction $N \coloneqq N'|_{\vert C}$ with the log line bundle with slopes $\vec a$: $\varphi : N \longsimeq \gp M_{\vert C}(\sum a_i p_i)$. 
\end{itemize}

A dashed lift in the above square would mean an isomorphism $\varphi' : N' \longsimeq \gp M_{\vert {C'}}(\sum a_i p_i)$ which restricts to $\varphi$. The isomorphisms $\varphi, \varphi'$ may be thought of as trivializations of the $\Glog$-torsors $N(-\sum a_i p_i), N'(-\sum a_i p_i)$ on $\vert C, \vert {C'}$.

As in the previous section, the trivialization $\varphi$ gives a $J|_C$-torsor which induces $N(-\sum a_i p_i)$. A trivialization $\varphi'$ restricting to $\varphi$ is the same as a trivialization of this $J|_C$-torsor. So the obstruction groupoid is $R^1\pi_* J|_C$ again and the obstruction is the class of $N'(-\sum a_i p_i)$ as a $J|_C$-torsor. 

When we restrict this log perfect obstruction groupoid to nilimmersions coming from 
\[
\begin{tikzcd}
    S \ar[r] \ar[d]       &S' \ar[d] \ar[dl, dashed]         \\
    \ddp \ar[r]        &B \times B\Glog,
\end{tikzcd}
\]
the $\Glog$-torsor $N'$ is $P - Q|_{\vert C} = \underline{\rm Isom}(Q|_{\vert C}, P)$. An isomorphism of $N'$ with $\gp M_{\vert C}(\sum a_i p_i)$ is the same as an isomorphism 
\[
    Q|_{\vert C}(\sum a_i p_i) \longsimeq P,
\]
and the obstruction to the dashed lift is again the $J|_C$-torsor of sections of $I'$ from \eqref{eqn:I'isomsheafQtwistedisomP} that restrict to $\varphi$.

\subsubsection{Log perfect obstruction theories}\label{ss:DRlogpothodgebundle}

To get a log perfect obstruction theory on $\drp \simeq \ddp$ in the usual sense \cite{intrinsic}, \cite{herrthesis}, plug in $J = \OO_S$ into the above (log) obstruction groupoids. The result is 
\[
    \mathbb E^\vee : S \mapsto \Gamma(S, R^1\pi_* \OO_C),
\]
which is represented by the dual of the Hodge bundle $\mathbb E$ by Serre duality. 
Here, $\pi$ is the universal curve over $\drp \simeq \ddp$.

\begin{corollary}\label{cor:samelogPOT}
    The two log perfect obstruction theories on $\drp$ introduced in \ref{sss:logstabmapslogPOT} and \ref{sss:AJkernellogPOT} agree. 
\end{corollary}

\begin{remark}
By Remark~\ref{rmk:comparisonwithrubber}, the log obstruction theory in \S~\ref{sss:logstabmapslogPOT} is equivalent to that of log stable maps to $[\PP^1 /\GG_m ]$.
\end{remark}

\begin{remark}
    Consider the log perfect obstruction theories on $\drp \simeq \ddp$ relative to $\Mp$ given by the pullback 
    \[
    \begin{tikzcd}
        \ddp \ar[r] \ar[d] \lpb         &B \ar[d]        \\
        \Mp \ar[r]         &\lpp 
    \end{tikzcd}
    \]
    over the sheaf $\lpp$ instead of the stack $\LPp$. This log perfect obstruction theory has an extra factor of $B \pra{\Aff^1}^\circ$ which is unnecessary. This $\pra{\Aff^1}^\circ$-gerbe is the natural obstruction theory for deformations along the gerbe
    \[
    \begin{tikzcd}
        S \ar[r] \ar[d]       &S' \ar[d] \ar[dl, dashed]         \\
        \LP \ar[r]         &\lp,
    \end{tikzcd}
    \]
    as $\pra{\Aff^1}^\circ$ is the Lie algebra of $\Glog$. 
\end{remark}

\section{The log double ramification class in colimit log \texorpdfstring{$K$}{K}-theory} \label{s:4}

As in \S \ref{ss:drnotclosed}, the log double ramification locus $\dr[g, n, \vec a] \to \Ms$ is a proper monomorphism and not necessarily a closed substack. There is a log blowup ${\tilde M} \to \Ms$ on which the f.s.\ pullback of $\dr[g, n, \vec a]$ is a closed substack, for example the minimal log blowup defined in \cite{earlyholmesDRpaper}.
Thus the DR class is only defined on a sufficiently fine log alteration of $\Ms$. Working on log blowups instead of $\Ms$ is also necessary to get the double-double ramification formula \cite{drmultbchow}. 

We equip each log alteration $\tilde D \to \dr[g, n, \vec a]$ with a log virtual fundamental class 
\begin{equation}\label{eqn:logvfcdrclassintro}
    \lvir{\tilde D} \qquad \in K_\circ(\tilde D).
\end{equation}
The same construction works in Chow groups as well. Writing $\Alt{(\dr[g, n, \vec a])}$ for the poset of log alterations, we show the classes \eqref{eqn:logvfcdrclassintro} are compatible under pushforward, yielding a class in the limit of the $K$-theories of all models 
\[\Kl{}(\dr[g, n, \vec a]) \coloneqq \lim_{\tilde D \in \Alt{\dr}} K_\circ(\tilde D).\]

After restricting to a cofinal subposet $\cal C_{\dr[g, n, \vec a]} \subseteq \Alt{(\dr[g, n, \vec a])}$, we can also define Gysin pullback maps $\varphi^!$ between the models $\varphi : \tilde D_1 \to \tilde D_2$ after L.\ Barrott \cite{barrottlogchow}. We show the log virtual fundamental classes \eqref{eqn:logvfcdrclassintro} are also compatible under Gysin pullback 
\[
    \varphi^! \lvir{\tilde D_2} = \lvir{\tilde D_1} \qquad \in K_\circ(\tilde D_1).
\]
So this system of classes lies in the much smaller group
\[
    \colim_{\tilde D \in \cal C_{\dr[g, n, \vec a]}} K_\circ(\tilde D)
\]
made precise in Definition~\ref{def:colimitlimitlogKtheory}.

Once we push forward along the map $\dr[g, n, \vec a] \to \Ms$, we get a system of classes in the colimit over certain log alterations of $\Ms$
\begin{equation}\label{eqn:colimitlogKMs}
    \colim_{\tilde M \in \cal C_{\Ms}} K_\circ(\tilde M) = \colim_{\tilde M \in \cal C_{\Ms}} K^\circ(\tilde M) \eqcolon \cKl(\Ms),
\end{equation}
which is a \emph{ring}, the colimit log $K$-theory ring of $\Ms$. This ring structure depends on the log smoothness of $\Ms$. 

We can now define products of log double ramification classes in the colimit log $K$-theory ring of $\Ms$, and we prove a ``log product formula'' for them.

\subsection{Defining the log virtual fundamental classes on each model}\label{ss:logvfccolimitlogK}

Fix $f : X \to Y$ an eventually strict, quasicompact morphism between quasicompact logarithmic stacks. We state our results in general for future use, but we will take $f : X \to Y$ to be the morphism $\dr \to \Ms$. Suppose $X$ admits a log alteration by a DM stack and $Y$ is log smooth. 
Equip $f$ with a log perfect obstruction theory, i.e., a closed embedding of $\Aff^1$-stacks
\[
    \Cl{X/Y} \subseteq E
\]
in the sense of Behrend-Fantechi \cite{intrinsic}.

For each log alteration $\tilde X \in \Alt{X}$, we have an induced log perfect obstruction theory \cite[Remark 2.14]{herrthesis}
\[
    \Cl{\tilde X/Y} \subseteq \Cl{X/Y}|_{\tilde X} \subseteq E|_{\tilde X}. 
\]
The structure sheaf $\OO_{\Cl{\tilde X/Y}}$ defines a class in $K_\circ(E|_{\tilde X})$, and we have an isomorphism 
\[
    \pi^* : K_\circ(\tilde X) \longsimeq K_\circ(E|_{\tilde X})
\]
by \cite[Remark 1.6]{kthylogprodfmla}. Define the log virtual fundamental class of each $\tilde X \in \Alt{X}$ as 
\begin{equation}\label{eqn:lvfcsoneachmodel}
    \lvir{\tilde X} \coloneqq {\pi^*}^{-1}[\OO_{\Cl{\tilde X/Y}}] \qquad \in K_\circ(\tilde X).
\end{equation}

\begin{remark}\label{rmk:Chowassumptions}
    The same can be done in Chow groups, provided we assume some log alteration of $X$ admits a stratification by quotient stacks \cite{kreschthesis}. The stack $\Log Y$ is smooth because we assumed $Y$ was log smooth, so $\Cl{\tilde X/Y}$ has equidimensional connected components and the fundamental class 
    \[
        [\Cl{\tilde X/Y}] \subseteq CH(E|_{\tilde X})
    \]
    is well-defined. A cofinal subposet of log alterations $\tilde X \to X$ admit stratifications by quotient stacks. We then have an analogous isomorphism 
    \[
        \pi^* : CH(\tilde X) \longsimeq CH(E|_{\tilde X})
    \]
    and define log virtual fundamental classes in each $CH(\tilde X)$ the same way as in $K$-theory. 
\end{remark}

\begin{remark}\label{rmk:cofinalsubposetfreecones}
Consider the subposet $\cal C \subseteq \Alt{Y}$ given by log alterations $\tilde Y \to Y$ such that
\begin{itemize}
    \item Writing $\tilde X \coloneqq X \times_Y^\ell \tilde Y$, both $\tilde Y$ and $\tilde X$ are representable by log algebraic stacks, 
    \item 
    The map $\tilde f : \tilde X \to \tilde Y$ is strict, and
    \item
    There is a strict map $\tilde Y \to \af{}^n$ for some $n \in \NN$. 
\end{itemize}
This subposet is cofinal in both $\Alt{Y}$ and $\Alt{X}$ by \cite[Theorem 4.6.2]{wisebounded} (See Corollary \ref{cor:wiseblowupforstacks} in the appendix), the fact that $f$ is eventually strict, and the fact that any log alteration $\tilde X \to X$ or $\tilde Y \to Y$ can be refined by the pullback of a log alteration of $\af{Y}$ \cite[Lemma 3.5]{kthylogprodfmla}. 
\end{remark}

Take $X = Y$ in Remark \ref{rmk:cofinalsubposetfreecones} and consider the cofinal subposet $\cal C \subseteq \Alt{X}$. For $\tilde X_1 \to \tilde X_2 \in \cal C$, there is an induced map $\varphi : \af{X_1} \to \af{X_2}$ by Proposition \ref{prop:AFfunctorialifYfree} which is l.c.i.\ because it is a DM-type map of smooth Artin stacks. There results a Gysin pullback map 
\begin{equation}\label{eqn:KthyGysinpullbackcolimitlogK}
    \varphi^! : K^\circ(X_2) \to K^\circ(X_1)
\end{equation}
between their vector-bundle $K$-theories. 

\begin{definition}\label{def:colimitlimitlogKtheory}
    The \emph{colimit log $K$-theory} of a logarithmic stack $X$ is the colimit under the Gysin pullback maps $\varphi^!$ of the log alterations $\tilde X \in \cal C$:
    \[
        \cKl(X) \coloneqq \colim_{\tilde X \in \cal C} K^\circ(\tilde X).
    \]
    The \emph{limit log $K$-theory} \cite[Definition~3.1]{kthylogprodfmla} is the limit under pushforwards of the coherent-sheaf $K$-theories of all log alterations $\tilde X \in \Alt(X)$:
    \[
        \Kl(X) \coloneqq \lim_{\tilde X \in \Alt(X)} K_\circ(\tilde X).
    \]
\end{definition}

In the definition of $\cKl(X), \Kl(X)$, we may freely replace the index of the co/limit by a cofinal subposet. 

\begin{theorem}\label{thm:lvfcsinbothKtheories}
    In the above setting, the log virtual fundamental classes \eqref{eqn:lvfcsoneachmodel} are compatible under both pushforward and Gysin pullback. Scilicet, let $p : X_1 \to X_2 \in \Alt{X}$ be log alterations of $X$ and define their corresponding log virtual fundamental classes as in \eqref{eqn:lvfcsoneachmodel}. Then 
    \[
        p_* \lvir{X_1} = \lvir{X_2} \qquad \in K_\circ(X_2).
    \]
    
    If both $X_i \in \cal C$ lie in the cofinal subposet of Remark \ref{rmk:cofinalsubposetfreecones}, write $\varphi : \af{X_1} \to \af{X_2}$ for the map between their Artin fans provided by Proposition \ref{prop:AFfunctorialifYfree}. Then 
    \[
        \varphi^! \lvir{X_2} = \lvir{X_1} \qquad \in K_\circ(X_1).
    \]
    
    Both statements also hold in Chow under the assumptions of Remark \ref{rmk:Chowassumptions}. 
\end{theorem}

\begin{proof}
    We prove the pushforward statement in Theorem \ref{thm:kvfcsinlogkthy} in the appendix. 

    Let $Y_1 \to Y_2 \in \cal C \subseteq \Alt{Y}$ be a morphism in the cofinal subposet of Remark \ref{rmk:cofinalsubposetfreecones}. Write $X_i = X \times_Y^\ell Y_i$ and $f_i : X_i \to Y_i$ for the maps between them. Form the diagram
    \[
    \begin{tikzcd}
                &\af{X_1} \ar[dr]   \\
        X_1 \ar[r, "f_1"] \ar[d] \lpbstrict \ar[ur]         &Y_1 \ar[d] \ar[r]        &\af{Y_1} \ar[d, "\psi"]       \\
        X_2 \ar[r, "f_2", swap] \ar[dr]         &Y_2 \ar[r]        &\af{Y_2},     \\
                &\af{X_2} \ar[ur]
    \end{tikzcd}
    \]
    with strict Cartesian inner square $X_1 = X_2 \times_{Y_2}^{\msout{\ell}} Y_1$. 
    
    We have morphisms between the Artin fans by Proposition \ref{prop:AFfunctorialifYfree}. Name the maps
    \[\varphi : \af{X_1} \to \af{X_2}, \qquad \psi : \af{Y_1} \to \af{Y_2}\]
    and endow each with its perfect obstruction theory as a representable l.c.i.\ map between smooth Artin stacks.

    We claim 
    \begin{itemize}
        \item 
        The cotangent complexes and hence Gysin pullbacks of the maps on Artin fans between the log alterations are the same 
        \begin{equation}\label{eqn:AFpbsame}
            \ccx{\af{X_1}/\af{X_2}} = \ccx{\af{Y_1}/\af{Y_2}}|_{\af{X_1}}, \qquad \varphi^! = \psi^!,
        \end{equation}
        where $\varphi^!, \psi^!$ are identified where both defined. 
        \item 
        The Gysin pullbacks commute
        \begin{equation}\label{eqn:Gysinpbsquare}
            f_1^! \psi^! = \psi^! f_2^!.
        \end{equation}
    \end{itemize}
    The second claim is standard \cite{Fultonintersectiontheory} and the first claim results from the fact that the morphisms $\af{X_i} \to \af{Y_i}$ are representable and strict \'etale, 
    hence \'etale. 

    Identify the log normal cones 
    \[\Cl{X_i/Y} = C_{X_i/Y_i}\]
    by strictness of the $f_i$. Then 
    \begin{align*}
        \varphi^!\lvir{X_2} &=\varphi^!{\pi^*}^{-1}[C_{X_2/Y_2}]        \\
                &={\pi^*}^{-1}\psi^! f_2^! [Y_2]  &&\eqref{eqn:AFpbsame}      \\
                &={\pi^*}^{-1}f_1^! \psi^! [Y_2]   &&\eqref{eqn:Gysinpbsquare}     \\
                &={\pi^*}^{-1}f_1^! [Y_1]   &&\text{because $Y$ log smooth}      \\
                &={\pi^*}^{-1}[C_{X_1/Y_1}]         \\
                &=\lvir{X_1}.
    \end{align*}    
    We have written $\pi$ abusively for the appropriate pullbacks of $\pi : E \to X$. 
\end{proof}

\begin{lemma}\label{lem:pushforwardcolimitlogKring}
    Use notation as above. Let $s : X \to X'$ be a proper, eventually strict map to a log smooth log algebraic stack $X'$. Then the pushforward of the system of log virtual classes of log alterations of $X$ lies in the colimit log $K$-theory ring of $X'$:
    \[
        s_* \pra{\lvir{X_i}}_{X_i \in \cal C} = \pra{s_{i*}\lvir{X_i}}_{X_i \in \cal C} \qquad \in \cKl(X') \subseteq \Kl(X').
    \]
\end{lemma}

\begin{proof}
    Theorem \ref{thm:lvfcsinbothKtheories} showed this system of classes is compatible under the Gysin maps $\varphi^!$ associated to the maps between Artin fans $\varphi : \af{X_1} \to \af{X_2}$. The eventually strict hypothesis means we can refine $\cal C$ by a cofinal subposet which is pulled back from subdivisions of both $X'$ and $Y$, arguing as in the appendix. Then the pushforwards $s_i : X_i \to X'_i$ between models are compatible with Gysin pullbacks, so the pushforward 
    \[\pra{s_{i*}\lvir{X_i}}_{X_i \in \cal C}\] 
    is compatible with Gysin pullbacks along the Artin fans of the models of $X'$. But these models are log smooth and assumed to have free log structure, so they are smooth and $K_\circ(X'_i) = K^\circ(X'_i)$. This colimit is then precisely the colimit log $K$-theory ring of $X'$.   
\end{proof}

\subsection{Specializing to the log double ramification class in log $K$-theory} \label{ss:logvfcdoubleramificationdef}

\begin{situation}\label{sit:3}
    Suppose $B$ is proper and log smooth and that the morphism $B \to \Mp$ parameterizing the family $C/B$ of curves is eventually strict. 
\end{situation}

For example, $B$ could have the minimal log structure associated to the family of curves $C/B$. We can make the same requirement for a tuple $\vec P/C/B_r$ of $\Glog$-torsors as in Situation \ref{sit:logsmoothbasehigherrank}.

Let $Q \to \dr \in \Alt(\dr)$ be a log alteration by a log algebraic stack $Q$. Recall from \S \ref{ss:DRlogpothodgebundle} that the dual of the Hodge bundle $\mathbb E^\vee$ serves as a log perfect obstruction theory for the map 
\[
\dr \to B \times B \Glog.
\]

\begin{remark}
If $\frak E$ is a log perfect obstruction theory on $Y$ relative to some base and $X \to Y$ is a log alteration, the restriction $\frak E|_X$ provides a natural log perfect obstruction theory for $X$ as in \cite[Remark 2.14, \S 3]{herrthesis}. By \cite[Remark 1.6]{kthylogprodfmla}, the $K$-theory groups are isomorphic under pullback: 
\begin{equation}\label{eqn:adeelkhanvbkthyisom}
    K_\circ(Q) \longsimeq K_\circ(\frak E|_Q).
\end{equation}
Check that a cofinal system of such log alterations $Q$ admits stratifications by global quotient stacks, so we can assume the same isomorphism holds for the Chow groups of a vector bundle stack \cite[Proposition 4.3.2]{kreschthesis}. 
\end{remark}

\begin{definition}
    In the above setting, define the log virtual fundamental class of $Q$ in $K$-theory (resp.\ Chow groups) to be the image of $[\OO_{\Cl{Q/B \times B \Glog}}]$ (resp.\ $[\Cl{Q/B \times B \Glog}]$) under the inverse  
    \[
        K_\circ(\mathbb E^\vee|_Q) \longsimeq K_\circ(Q) \qquad (\text{resp.\ } CH_*(\mathbb E^\vee|_Q) \longsimeq CH_*(Q))
    \]
    of the isomorphism \eqref{eqn:adeelkhanvbkthyisom}. 
\end{definition}

\begin{remark}
    The class $\lvir{Q} \in K_\circ(Q)$ just defined coincides with the virtual structure sheaf of \cite{quantumkthyypleethesis} by \cite[\S 2.2]{Kthyvirtpbs}. 
\end{remark}

\begin{theorem}
    For each log alteration $Q \in \Alt{\dr}$, equip $Q \to \dr \to B \times B\Glog$ with the pullback \cite[Remark 2.14, \S 3]{herrthesis} of the log perfect obstruction theory $\mathbb E^\vee$ from \S \ref{ss:DRlogpothodgebundle} for 
    \[
        \dr \to B \times B\Glog
    \]
    and write $\lvir{Q} \in K_\circ(Q)$ for the resulting log virtual fundamental class/structure sheaf. 

    The system of classes $\pra{\lvir{Q}}_{Q \in \Alt{\dr}}$ is compatible under pushforward, defining the $K$-\emph{theoretic double ramification} class in limit log $K$-theory 
    \[
        \lvir{\dr} \coloneqq \pra{\lvir{Q}}_{Q \in \Alt{\dr}} \qquad \in \Kl(\dr).
    \]

    The pushforward of this class along the proper map $p : \dr \to \Ms$ lies in the smaller \emph{colimit} log $K$-theory ring 
    \[
        p_* \lvir{\dr} \in \cKl(\Ms) \subseteq \Kl(\Ms). 
    \]
\end{theorem}

The last point is important because it lets us take products of the double ramification classes after pushing forward to $\Ms$, as $\cKl(\Ms)$ is a ring. 

\begin{proof}
    The theorem results from Theorem \ref{thm:lvfcsinbothKtheories} and Lemma \ref{lem:pushforwardcolimitlogKring}, where 
    \begin{itemize}
        \item The map $f : X \to Y$ is $\dr \to B \times B\Glog$ equipped with the log perfect obstruction theory $E = \mathbb E^\vee$, and 
        \item The proper, eventually strict map $s : X \to X'$ is $\dr \to \Ms$. 
    \end{itemize}
\end{proof}

\begin{remark}\label{rmk:logGysinmapkthyintnprod}
    We recall the log Gysin map/intersection product in $K$-theory constructed in \cite{kthylogprodfmla}. Let $f : X \to Y$ be an eventually strict, quasicompact morphism of quasicompact log algebraic stacks, i.e., one which fits into a commutative square 
    \[
    \begin{tikzcd}
        \tilde X \ar[r] \ar[d]        &\tilde Y \ar[d]       \\
        X \ar[r]       &Y
    \end{tikzcd}
    \]
    with $\tilde f : \tilde X \to \tilde Y$ strict and vertical maps log alterations. Equip $f$ with a log perfect obstruction theory $\Cl{f} \subseteq E$. 

    A class $\alpha = (\alpha_i) \in \Kl(Y)$ is a system of classes $\alpha_i \in K(\tilde Y_i)$ on log alterations $\tilde Y_i \to Y$ of $Y$ which are compatible under pushforward. We define a map 
    \[
        f^\dagger : \Kl(Y) \to \Kl(X)
    \]
    by sending $\alpha = (\alpha_i)$ to the system $(\tilde f^!_i \alpha_i) \in \Kl(X)$, where $\tilde f_i : \tilde X_i \to \tilde Y_i$ is the cofinal system of log alterations of $X$ which are 
    \begin{itemize}
        \item pulled back from log alterations of $Y$, and 
        \item lead to strict maps $\tilde f_i : \tilde X_i \to \tilde Y_i$. 
    \end{itemize}
    This system $(\tilde f_i^! \alpha_i)$ is automatically compatible with pushforward between the log alterations $\tilde X_i \to \tilde X_j$ by compatibility with the usual Gysin map $\tilde f_i^!$ with pushforward \cite[Theorem 6.2]{Fultonintersectiontheory} and its generalization to representable proper maps between algebraic stacks \cite[Appendix B]{BSS-I}. 
\end{remark}

\begin{remark}\label{rmk:GysinpbpreservescolimKthy}
    Let $\cal C \subseteq \Alt{Y}$ be a cofinal subposet consisting of log alterations $\tilde Y \to Y$ which admit a strict map to some $\af{}^n$. If $Y_1 \to Y_2 \in \cal C$, there is a corresponding map $\psi : \af{Y_1} \to \af{Y_2}$ on Artin fans by Proposition \ref{prop:AFfunctorialifYfree}. Suppose $\alpha = (\alpha_i) \in \Kl(Y)$ has the property that 
    \begin{equation}\label{eqn:Gysinpbstability}
        \psi^! \alpha_2 = \alpha_1 \qquad \in K_\circ(Y_1)
    \end{equation}
    for each such map $Y_1 \to Y_2 \in \cal C$. Then $f^\dagger \alpha \in \Kl(X)$ also has that property for a sufficient cofinal subposet of $\Alt{X}$. If $Y$ is log smooth and $\cal C$ consists of models $Y_i \to Y$ which have locally free log structures, then \eqref{eqn:Gysinpbstability} is exactly the condition that cuts out the subring 
    \[
        \cKl(Y) \subseteq \Kl(Y).
    \]
\end{remark}

\begin{proposition}\label{prop:logintnproductcolimitkthy}
    Let $X$ be a log smooth, quasicompact logarithmic stack. Its colimit log $K$-theory $\cKl(X)$ is a ring with the ``intersection product'' defined by log Gysin pullback along the diagonal of $X$. 
\end{proposition}

\begin{proof}
    The diagonal $X \to X \times X$ is log l.c.i., in that its log cotangent complex $\lccx{X/X \times X}$ is perfect in degrees $[-1, 0]$. This log perfect obstruction theory engenders a log Gysin pullback 
    \[
        \Delta^\dagger : \Kl(X \times X) \to \Kl(X).
    \]
    The intersection product on $\cKl(X)$ is defined by 
    \[
        \cKl(X) \times \cKl(X) \to \cKl(X \times X) \to \Kl(X \times X) \overset{\Delta^\dagger}{\longrightarrow} \Kl(X),
    \]
    remarking that the composite factors through the subring $\cKl(X) \subseteq \Kl(X)$ by Remark \ref{rmk:GysinpbpreservescolimKthy}, as $\Delta^\dagger$ preserves the property \eqref{eqn:Gysinpbstability}. 
\end{proof}

\begin{warning}
    The log intersection product on $\cKl(X)$ is defined using the log Gysin map 
    \[
        \Delta^\dagger : \Kl(X \times X) \to \Kl(X).
    \]
    However, there is no analogous intersection product defined on \emph{limit} log $K$-theory $\Kl(X)$. This is because there is no suitable map 
    \[
        \Kl(X) \times \Kl(X) \dashrightarrow \Kl(X \times X). 
    \]
\end{warning}

As in classic intersection theory \cite{Fultonintersectiontheory}, the most common example of this is the diagonal of a suitably smooth object. This is still not quite enough to get an ``intersection product,'' however. 

\begin{example}\label{ex:logdiagintnprod}
    If $X, Y$ are quasicompact, log smooth log algebraic stacks, the diagonal 
    \[
        \Delta : X \to X \times X
    \]
    is eventually strict and admits a canonical log perfect obstruction theory $\Cl{\Delta} \subseteq \Nl{\Delta}$. By Remark \ref{rmk:logGysinmapkthyintnprod}, we have a log Gysin map 
    \[
        \Kl(X \times X) \to \Kl(X).
    \]

    Unlike in ordinary Chow and $K$-theory, we do not get a log intersection product on $\Kl(X)$ from this map. There is no map 
    \[
        \Kl(X) \times \Kl(Y) \dashrightarrow \Kl(X \times Y),
    \]
    as there are far more log alterations of $X \times Y$ than those pulled back from either factor. They each project onto the smaller limit of ``rectangular subdivisions'' which are intersections of those pulled back from either factor: 
    \[
        \Kl(X) \times \Kl(Y) \to \lim_{\text{rectangular } B \to X \times Y} K_\circ(B) \leftarrow \Kl(X \times Y). 
    \]
    
\end{example}

\begin{example}
    Let $X = \Aff^1$ and let $K'(Z)$ be the limit over only log blowups of $Z$ instead of $\Kl$ for simplicity. Then $K'(\Aff^1) = K_\circ(\Aff^1) = \ZZ$ is the usual $K$-theory of $\Aff^1$, as $\Aff^1$ is valuative and admits no nontrivial log blowups. 

    On the other hand, $K'(X \times X) = K'(\Aff^2)$ is a limit over the $K$-theories of all log blowups of $\Aff^2$. These correspond to (projective) subdivisions of the first quadrant. For convenience, we range over the cofinal system of smooth log blowups $B \to \Aff^2$. Then \cite[Theorem 1.6]{AndersonPayneOperationalKTheory} or \cite[\S 2.4]{BrionVergneEquivariantRiemannRoch} compute the \emph{equivariant} $K$-theory ring of such a blowup $B \to \Aff^2$ as 
    \begin{align*}
        K_\circ^T(B) &= \{\text{piecewise exponential functions on the fan }\NN^2 \text{ of }\Aff^2\}       \\
                &= \{(f_\sigma) \in \prod_{\sigma \in \Sigma_B} R(T_\sigma) \, | \, f_\sigma|_\tau = f_\tau \text{ whenever } \tau \preceq \sigma\}.
    \end{align*}
    To get the non-equivariant version $K_\circ(B)$, you quotient by the ideal generated by differences of characters. 

    In this case, there is a product map which factors through the inclusion of the subring given by the colimit 
    \[K'(\Aff^1) \times K'(\Aff^1) = K^\circ(\Aff^1) \times K^\circ(\Aff^1) \to \colim K^\circ(\Aff^2) \subseteq K'(\Aff^2).\]
\end{example}

\begin{example}
    The group $K'(\Aff^2)$ for $X = \Aff^2$ is not a ring. It's defined as a limit under the pushforwards between models of $\Aff^2$, but these pushforward morphisms are almost never ring homomorphisms. For example, consider 
    \[\pi : Q \coloneqq Bl_{\vec 0} \Aff^2 \longrightarrow \Aff^2\] 
    and let $D_1, D_2 \subseteq Q$ be the strict transforms of the $x-$ and $y-$axes $V(y), V(x) \subseteq \Aff^2$. Then 
    \[
        [D_1] \cdot [D_2] = 0, \qquad \pi_*([D_1]) \cdot \pi_*([D_2]) = [V(y)] \cdot [V(x)] = [\vec 0],
    \]
    so $\pi_*([D_1] \cdot [D_2]) \neq \pi_*([D_1]) \cdot \pi_*([D_2])$. 
\end{example}

\subsection{Product formula} \label{ss:pf}

Let $M = (m_{ij}) \in \GL_r(\ZZ)$ be an invertible matrix with integer entries and consider an $n \times r$ integer matrix $A$ whose columns sum to zero as in \S \ref{higherranksection}. Obtain an isomorphism 
\[M : B\Glog^r \longsimeq B\Glog^r\]
by sending a tuple of log line bundles $\vec Q = (Q_k)_r$ on a log scheme $T$ to the tuple of products 
\[
    \vec Q.M \coloneqq \pra{
    \bigotimes_k Q_k^{m_{k1}}, \bigotimes_k Q_k^{m_{k2}}, \cdots, \bigotimes_k Q_k^{m_{kr}}
    } \qquad \in B\Glog^r(T).
\]

Fix an $r$-tuple of $\Glog$-torsors $\vec P = (P_i)_r$ on a \emph{stable} curve $C/B_r$ as in Situations \ref{sit:logsmoothbasehigherrank} and \ref{sit:3}. 
The above yields a commutative square 
\begin{equation}\label{eqn:GLractiondr}
\begin{tikzcd}
    \dr[\vec P, A] \ar[r, "\sim"] \ar[d]        &\dr[\vec P.M, AM] \ar[d]       \\
    B_r \times B\Glog^r \ar[r, "M", swap]         &B_r \times B\Glog^r
\end{tikzcd}    
\end{equation}
with horizontal arrows isomorphisms, which we make explicit. Note the columns of $AM$ also sum to zero. 

An $S$-point of the left upper corner $\dr$ parameterizes a tuple of isomorphisms 
\begin{equation}\label{eqn:isomsdrforGLinvariance}
    \OO_C(\sum a_{ik} p_i) \times^{\OO_C^*} Q_k|_{\vert C_S} \longsimeq P_k|_{\vert C_S}
\end{equation}
for $1 \leq k \leq r$. If $(m_{1t}, m_{2t}, \cdots, m_{rt})$ is the $t$th column of $M$, the $t$th trivialization on the right upper corner $\dr[\vec P.M, AM]$ is the isomorphism
\[
    \OO_C(\sum m_{it} \vec a_i) \times^{\OO^*_C} \bigotimes Q_i^{m_{it}}|_{\vert C_S} \longsimeq \bigotimes P_i^{m_{it}} 
\]
obtained by tensoring together $m_{it}$-multiples of the isomorphisms \eqref{eqn:isomsdrforGLinvariance}. Check that the obstruction theories of the vertical maps in \eqref{eqn:GLractiondr} are identified under the isomorphisms, so 
\begin{equation}\label{eqn:GLractiondrlvir}
    M_*\lvir{\dr} = \lvir{\dr[\vec P.M, AM]}
\end{equation}
as in \cite[Theorem 5.3]{holmesschwarzGLrinvarianceprodfmla}. 

One could instead replace $P_k$ by $P_k(-\sum a_i p_i)$ to assume the matrix $A = [0]$ is zero, but the above makes it clear why the resulting contact orders are $AM$. 

\begin{example}
    Fix $C/B_r$ and assume $S = \Spec \CC$ with trivial log structure and the torsors $P_i \simeq \gp M_{\vert C_S}$ are trivial for convenience. Suppose 
    \[
    A = \begin{bmatrix}
        -1 &3 \\
        1 &-3
    \end{bmatrix}, \qquad 
    M = \begin{bmatrix}
        -5 &2\\
        -3 &1
    \end{bmatrix},
    \]
    so $AM = \begin{bmatrix}
        -4&1 \\
        4 &-1
    \end{bmatrix}$. 

    Let $C$ be a smooth curve over $\CC$ with two marked points $p, q$. Its vertical log structure is also trivial. The $S$-point of $\Ms[g, 2]$ lies in $\dr$ if and only if the line bundles associated to 
    \[
        \OO_C(-p + q), \qquad \OO_C(3p - 3q)
    \]
    have trivializing sections $s_1, s_2$. The isomorphism \eqref{eqn:GLractiondr} sends these to the line bundles with sections
    \[
        (\OO_C(-4p + 4q), s_1^{-5} \otimes s_2^{-3}), \qquad (\OO_C(p-q), s_1^2 \otimes s_2). 
    \]
\end{example}

We establish a combined log product formula and $\GL_r(\ZZ)$-invariance statement \cite[Theorem 5.3]{holmesschwarzGLrinvarianceprodfmla} for the log DR class. 

Let $A = [\vec a_i]$ be an $n \times r$ integer matrix such that each column $\vec a_i$ has entries summing to zero. Let $M \in \GL_r(\ZZ)$ be an invertible integer matrix and write the columns of $AM = [\vec b_i]$ as $\vec b_i$. Fix $\vec P/C/B_r$ satisfying Situation \ref{sit:logsmoothbasehigherrank} and suppose $B_r$ is proper and log smooth and $B_r \to \Ms$ is eventually strict as in Situation \ref{sit:3}. 
Write $N_t = \bigotimes P_i^{m_{it}}$.

\begin{theorem} \label{t:product}
    In the above situation, the products of DR classes on $B_r$ agree: 
    \[
    \lvir{\dr[P_1, \vec a_1]} \cdot \cdots \cdot \lvir{\dr[P_r, \vec a_r]} =
    \lvir{\dr[N_1, \vec b_1]} \cdot \cdots \cdot \lvir{\dr[N_r, \vec b_r]} \qquad \in \cKl(B_r).
    \]
    Both sides are equal to $\lvir{\dr[\vec P, A]} \in \cKl(B_r)$.
\end{theorem}

We omit the pushforward of the classes from $\dr[\vec P, A]$ to $B_r$ in the notation of the theorem, but all classes are colimit log $K$-theory classes on $B_r$. 

\begin{proof}
    It suffices to show 
    \[
    \lvir{\dr[\vec P, A]} = \lvir{\dr[P_1, \vec a_1]} \cdot \cdots \cdot \lvir{\dr[P_r, \vec a_r]}
    \]
    in $\cKl(B_r)$ because of the isomorphism \eqref{eqn:GLractiondrlvir}. 

    By Remarks \ref{rmk:AJBmapGlogtorsordifference}, \ref{rmk:drpbbaseB} and their analogue in higher rank, we have a pullback square
    \[
    \begin{tikzcd}
        \dr[\vec P, A] \ar[r] \ar[d]      &\prod \dr[P_i, \vec a_i] \ar[d]      \\
        B_r \ar[r, "\Delta_B"]       &(B_r)^r.
    \end{tikzcd}
    \]
    The log perfect obstruction theories of the vertical arrows are the same. By \cite[Remark 1.8]{kthylogprodfmla}, it results that 
    \[
        \Delta_B^\dagger \pra{\boxtimes_i \lvir{\dr[P_i, \vec a_i]}} = \lvir{\dr[\vec P, A]}.
    \]
    By Proposition \ref{prop:logintnproductcolimitkthy}, the pullback $\Delta_B^\dagger \lvir{\prod \dr[P_i, \vec a_i]}$ defines the product of the classes $\dr[P_i, \vec a_i]$ on $B_r$.

\end{proof}

We do not need the more complicated diagram \cite[Equation (1)]{kthylogprodfmla} adapted from \cite{KontsevichManin1994GromovWitten, BehrendManin1996StacksStableMaps, prodfmla, herrthesis} because stability in $\dr \subseteq \drp$ is equivalent to stability of the underlying curve.

\begin{example}
    Take $B_r = \Ms$, $C/B_r$ to be the universal curve, and $P_i = \gp M_C$ the trivial log line bundles. Then Theorem \ref{t:product} states that 
    \[
        \prod \lvir{\dr[g, n, \vec a_i]} = \prod \lvir{\dr[g, n, \vec b_i]}
    \]
    on $\cKl(\Ms)$ if $\vec b_i$ are the columns of the matrix $AM$ with $A = [\vec a_1 \cdots \vec a_r]$. This universal case was done in \cite[Theorem 5.3]{holmesschwarzGLrinvarianceprodfmla}. 
\end{example}

\section{Thom--Porteous formula in \texorpdfstring{$K$}{K}-theory for stacks}\label{s:thomporteous}

Recognize $\Aff^{m \times n} = {\rm Mat}_{n \times m}$ as the universal base of a morphism $A \to B$ of trivial vector bundles of ranks $m, n$. Quotienting by $\GL_m \times \GL_n$ we get the universal base of a morphism of not-necessarily-trivial vector bundles. Suppose $m \geq n$ and write $i : V \subseteq \Aff^{m \times n}$ for the closed subscheme on which the map $A \to B$ does not have full rank; that is $\rk(A \to B) < n$. 

There is an exact sequence $C_\bullet$ resolving the structure sheaf $\OO_V$ with terms: 
\begin{equation}\label{eqn:buchsbaumrimjerzyweymanresolution}
    C_i = \wedge^{n+i-1} A \otimes \wedge^n (B^\vee) \otimes D_{i-1}(B^\vee).
\end{equation}
The formula in loc.\ cit.\ has an extra factor of $-\otimes_\ZZ \ZZ(-n-i+1)$ shifting the degree, but we do not need the grading.% and the $\GL_m \times \GL_n$ action preserves degree. 
Versions of this formula are due to Buchsbaum--Rim, Eagon--Northcott, Eisenbud, Weyman; see \cite[\S A.2.6.1]{commutativealgebraEisenbud1995} for historical context.

\begin{theorem}[{\cite[Proposition 6.1.7]{jerzyweymansyzygiesvectorbundles}}]\label{thm:jerzyweymanresolution}
    The complex $C_\bullet$ is a $\GL_m \times \GL_n$-equivariant minimal free resolution of $\OO_V$.  
\end{theorem}

For a vector bundle $A$ on an algebraic stack $X$ define the total divided power series $D_t$ and total exterior power series $\lambda_t$ as: 
\[
    D_t(A) \coloneqq \sum_{i = 0}^\infty [D_i A] t^i, \quad \lambda_t(A) \coloneqq \sum_{i=0}^{\infty}   [\wedge^i A] t^i.
\]
Here, $D_i A$ is the submodule of the tensor power $A^{\otimes i}$ consisting of all tensors that are invariant under the permutation action of the symmetric group $\mathfrak{S}_i$.

There is always a natural map from the divided power to the symmetric power:
\[
\phi: D_k A \hookrightarrow A^{\otimes k} \twoheadrightarrow \operatorname{Sym}^k A.
\]
In characteristic zero, one can construct an inverse map by averaging tensors, which gives $D_i A \cong {\rm Sym^k}A$.
In characteristic $p$, $\phi$ can have nontrivial kernel and cokernel. In the $K$-group, they are the same:

\begin{lemma}[{\cite[Appendix~A.2.6]{commutativealgebraEisenbud1995}}] \label{lem:equiv_sym_divide}
 Let $A$ be a vector bundle on an algebraic stack $X$. The classes of the divided powers $D_k A$ and the symmetric powers $\Symc k A$ coincide in $K$-theory:
 \[
\Sym k A = [D_k A].
 \]
 As a result, $D_t(A)$ and $\lambda_{-t}(A)$ are inverses:
 \begin{equation}\label{eqn:dividedpowergenfunction}
 D_t(A) \cdot \lambda_{-t}(A) =1.
 \end{equation}
\end{lemma}

\begin{proof}

Recall the canonical Koszul complex resolving the symmetric power (dubbed $K(\varphi')_k$ on \cite[pg. 591]{commutativealgebraEisenbud1995}):
\[
 0 \to \wedge^k A \to \wedge^{k-1} A \otimes \operatorname{Sym}^1 A \to \dots \to \wedge^1 A \otimes \operatorname{Sym}^{k-1} A \to \operatorname{Sym}^k A \to 0.
\]
Dually, we have the canonical Koszul complex resolving the exterior power in terms of divided powers:
\[
 0 \to D_k A \to D_{k-1} A \otimes \wedge^1 A \to \dots \to D_1 A \otimes \wedge^{k-1} A \to \wedge^k A \to 0.
\]
Taking alternating sums of both complexes, we get equalities in $K$-theory:
\[
 [\operatorname{Sym}^k A] = \sum_{i=1}^k (-1)^{i-1} [\wedge^i A] [\operatorname{Sym}^{k-i} A] \qquad
[D_k A] = \sum_{i=1}^k (-1)^{i-1} [\wedge^i A] [D_{k-i} A].
 \]
By induction, $[\operatorname{Sym}^{k-i} A] = [D_{k-i} A]$ for all $1 \le i < k$. The base case of $k = 1$ is immediate, and the illustrative case $k = 2$ is addressed in \cite{goodwillieMOanswer}. The right-hand sides of both recursive formulas are identical, which gives $[\operatorname{Sym}^k A] = [D_k A]$.

For the formula \eqref{eqn:dividedpowergenfunction}, it suffices to prove it for a line bundle. In that case,
\[
\frac{1}{\lambda_{-t}(A)} = \frac{1}{1-tA} = \sum_{i=0}^{\infty}t^iA^i = D_tA,
\]
since $A^i = \operatorname{Sym}^iA$ for a line bundle. Now, by the splitting principle, the result follows.
\end{proof}

\begin{corollary}
The class of $[\OO_V]$ in the $\GL_m \times \GL_n$-equivariant $K$-theory of ${\rm Mat}_{n \times m}$ is
\begin{equation}\label{eqn:firstappearanceGoperator}
    \begin{aligned}
    i_*[\OO_V] &= [\OO_{{\rm Mat}_{n \times m}}] + \sum_{i = 1}^\infty (-1)^i[\wedge^{n+i-1} A] [\wedge^n(B^\vee)] [D_{i-1}(B^\vee)] 
    \\
    & = [\OO_{{\rm Mat}_{n \times m}}] - [\det(B^\vee)] \left[ \lambda_t(A)D_{-t^{-1}}(B^\vee) \right]_{t^n} 
    \\
    & = [\OO_{{\rm Mat}_{n \times m}}] - [\det(B^\vee)] \left[ \frac{\lambda_t(A)}{\lambda_{t^{-1}}(B^\vee)} \right]_{t^n}
\end{aligned}
\end{equation}
Here, $\left[ - \right]_{t^n}$ means to take the coefficient of $t^n$ in the expression.
\end{corollary}

\begin{proof}
Lemma~\ref{lem:equiv_sym_divide} gives the last equality.
Expand their products and take $t^n$-coefficient:
\[
\begin{split}
\left[ \lambda_t(A) D_{-t^{-1}}(B^{\vee}) \right]_{t^n}
&= \sum_{i=0}^{\infty} (-1)^i [\wedge^{n+i} A] [D_i B^{\vee}] = - \sum_{i=1}^{\infty} (-1)^i [\wedge^{n+i-1} A][D_{i-1} B^{\vee}] .
\end{split}
\]
This is exactly the resolution \eqref{eqn:buchsbaumrimjerzyweymanresolution} of $[\mathcal{O}_V]$ given by Theorem \ref{thm:jerzyweymanresolution}, twisted by $\wedge^n(B^{\vee}) = \det (B^{\vee})$.
\end{proof}

In our setting, we have a map $\varphi : E \to F$ of vector bundles of ranks $r_1, r_2$ with $r_1 \leq r_2$. We are interested in the locus $j : W \subseteq \bar{\rm Mat}_{r_2 \times r_1}$ where the rank is not maximal, i.e., when it drops below $r_1$. This is the opposite setting of the above results, but they are related by taking duals. 

There is an isomorphism 
\begin{equation} \label{eqn:dualvbisom}
    {\rm Mat}_{r_2 \times r_1} \longsimeq {\rm Mat}_{r_1 \times r_2}
\end{equation}
which sends a map $E \to F$ to the dual map of vector bundles $F^\vee \to E^\vee$ and identifies $W$ with the $V$ discussed above after plugging in 
\[
    A = F^\vee, \quad B = E^\vee, \quad m = r_2, \quad n = r_1.
\]

The attentive reader will recognize the right hand side of \eqref{eqn:firstappearanceGoperator} as the operator $\tilde G$ of \ref{def:stackygrothpoly} applied to $B^\vee \to A^\vee$. We show this operator is well-defined in $K$-theory. 

\begin{lemma}\label{lem:tildeGwelldefKtheory}
    Formula $\tilde G$ from Definition \ref{def:stackygrothpoly} only depends on the class of $E - F$ in $K$-theory. 
\end{lemma}

\begin{proof}
If $H$ is another vector bundle on $Y$, we claim $\tilde G(E + H - (F + H)) = \tilde G(E-F)$. This suffices for the statement because, if $[E] = [F]$ in $K^\circ$-theory, then $E \oplus H = F \oplus H$ for some $H$. 

Note that
\[
    \frac{\lambda_t(H^{\vee})}{\lambda_{t^{-1}}(H)} = \det(H)^{-1} t^{r_H},
\]
which follows by assuming the splitting of $H = \oplus_{i=1}^{r_H} L_i$:
\[
    \frac{\lambda_t(H^{\vee})}{\lambda_{t^{-1}}(H)} = \frac{ \prod_{i=1}^{r_H} (1+ t  [L_i^{-1}] ) }{\prod_{i=1}^{r_H} (1+t^{-1} [L_i]) } = \prod_{i=1}^{r_H} (t [L_i^{-1}]) = \det(H)^{-1} t^{r_H}.
\]
Now compute 
\[
\begin{split}
    \tilde G((E+H) - (F+H)) &= [\OO] - [\det (E+H)] \cdot \bra{\dfrac{\lambda_t((F+H)^\vee)}{\lambda_{t^{-1}}(E+H)}}_{{\rm coeff} \, t^{r_1+r_H}}
    \\
    & = [\OO] - [\det(E+H)] \bra{\dfrac{\lambda_t(F^\vee)}{\lambda_{t^{-1}}(E)}  \det (H)^{-1} t^{r_H} }_{{\rm coeff} \, t^{r_1+r_H}}
    \\
    & = [\OO] - [\det(E)] \bra{\dfrac{\lambda_t(F^\vee)}{\lambda_{t^{-1}}(E)}}_{{\rm coeff} \, t^{r_1}} = \tilde G(E-F). 
\end{split}
\]
\end{proof}

\begin{corollary}\label{cor:thomporteousKthytildeGdegenlocusformula}
    The class $j_*[\OO_W]$ in the $K$-theory of $\bar{\rm Mat}_{r_2 \times r_1}$ is given by the alternating sum 
    \begin{align*}
        j_*[\OO_W] &= [\OO_{\bar {\rm Mat}_{r_2 \times r_1}}] + \sum_{i = 1}^\infty (-1)^i 
    [\wedge^{r_1+i-1} F^\vee] [\wedge^{r_1}(E)] [D_{i-1}(E)]     \\
        &=[\OO_{\bar {\rm Mat}_{r_2 \times r_1}}] - [\det E] \cdot \Big[\frac{\lambda_t(F^\vee)}{\lambda_{t^{-1}}(E)}\Big]_{{\rm coeff} \, t^{r_1}} \\
        &=\tilde G(E - F).
    \end{align*}

    Consider a pullback square 
    \begin{equation*}%\label{eqn:pbdeterminantallocus}
    \begin{tikzcd}
        X \ar[r] \ar[d, "e", swap] \pb       &W \ar[d, "j"]      \\
        Y \ar[r]       &\bar{\rm Mat}_{r_2 \times r_1}
    \end{tikzcd}
    \end{equation*}
    where $e : X \subseteq Y$ is a regular immersion of algebraic stacks. Suppose $Y$ is smooth and $X/Y$ has the same codimension $r_2-r_1 + 1$ as $W \subseteq \bar{\rm Mat}_{r_2 \times r_1}$. Here, $Y \to \bar{\rm Mat}_{r_2 \times r_1}$ parameterizes a map of vector bundles $E \to F$ of ranks $r_1, r_2$ with $r_1 \leq r_2$ on $Y$. 
    
    The class of $[\OO_X]$ in the $K$-theory of $Y$ admits a similar formula: 
    \begin{align*}
        e_*[\OO_X] &= [\OO_Y] + \sum_{i = 1}^\infty (-1)^i 
    [\wedge^{r_1+i-1} F^\vee] [\wedge^{r_1}(E)] [D_{i-1}(E)]      \\
        &=[\OO_Y] - [\det E] \cdot \bra{\dfrac{\lambda_t(F^\vee)}{\lambda_{t^{-1}}(E)}}_{{\rm coeff} \, t^{r_1}} \\
        &=\tilde G(E - F).
    \end{align*}
\end{corollary}

\begin{proof}
    The formula for $j_*[\OO_W]$ comes from identifying the loci $W$ and $V$ under the isomorphism of equation \eqref{eqn:dualvbisom}. The second statement results from the acyclicity of determinantal complexes \cite[A2.10]{commutativealgebraEisenbud1995}. 
    As the pullback of the resolution of $j_*\OO_{\bar W}$ on $\bar {\rm Mat}_{r_2 \times r_1}$ gives a resolution for $e_*\OO_X$ on $Y$, we get the same formula upon taking alternating sums. 
\end{proof}

We state this corollary in the setting where $X \subseteq Y$ is a regular immersion of the correct codimension, but it suffices that the grade of the ideal sheaf $I_{X/Y}$ is $r_2 - r_1 + 1$. 

\begin{example}
Consider the case where $E,F$ are both line bundles, i.e. $r_1=r_2=1$. Corollary~\ref{cor:thomporteousKthytildeGdegenlocusformula} gives the following formula for the degeneracy locus:

\[
    [\OO_Y]-[E]\bra{\frac{(1+tF^{-1})}{(1+t^{-1}E)}}_{{\rm coeff \,}t^1}.
\]
Applying the geometric series expansion to the denominator yields:
\[
    [\OO_Y]-[E]\Big[(1+tF^{-1})\sum_{i=0}^\infty (-1)^it^{-i}E^i\Big]_{t^1}=1-[E][F^{\vee}].
\]
This is equivalent to the ``standard'' method for computing the degeneracy locus, which we recall. Observe that a map $E\to F$ is equivalent to a map $\OO_Y \to F \otimes E^\vee$, which is a section of $F\otimes E^\vee$. The degeneracy locus is the vanishing locus of this section, so its structure sheaf is given by the $K$-theoretic Euler class of $F\otimes E^\vee$, which is $1 - [E \otimes F^\vee]$.
\end{example}

\begin{example}
Let $\Aff^n$ be the universal space with a map from rank-one to rank-$n$ trivial vector bundles, with the map 
\[
    \OO \to \OO^{\oplus n}; \qquad 1 \mapsto (1, \cdots, 1)^T.
\]
The resolution gives $[\OO_{\vec 0}] = 0$, and we have 
\[
    1 - \bra{\dfrac{(1+t)^n}{(1+t^{-1})}}_{\text{coeff } t} = 1-1 = 0.
\]
\end{example}

\begin{remark} \label{Remark5.10}
    In \cite{andersbuchthomporteous}, Anders Buch provides a formula for the class $[\OO_X]$ in terms of the (stable double) Grothendieck polynomial of $[E] - [F]$ in the case where $X, Y$ are \emph{schemes}. Write $g = 1-r_1+r_2$.
    The (stable double) Grothendieck polynomial \cite[Theorem~6.11]{Grothendieckpoly} for a single column with $r$ rows is given by:
    \[
    G_{1^g}(E-F) = \sum_{i=g}^{\infty}(-1)^{i-g}\binom{i-1}{i-g}S_{1^i}(E-F),
    \]
    where
    \[
    S_{1^i}(E-F) = \sum_{k=0}^i(-1)^{i-k}e_k(1-E_1,\dots, 1-E_{r_1}) \ h_{i-k}(1-F_1^{\vee},\dots,1-F_{r_2}^{\vee})
    \]
    is the Schur polynomial for a column of length $i$, where $E_i$ and $F_i$ are $K$-theoretic Chern roots of $E$ and $F$ respectively, $h_m$ is the complete homogeneous symmetric function of degree $m$, and $e_k$ is the $k$-th symmetric function. 
    
    When the space is a scheme, elements of rank zero such as $1-F_i^\vee$ in the above formula are nilpotent. This can fail for stacks, so 
    the expression remains a formal power series.

    Buch provides formulas for the other degeneracy loci using the other double stable Grothendieck polynomials evaluated at $[E] - [F]$. These formulas similarly fail for stacks. Using the family of Eagon--Northcott complexes and their generalizations in \cite{jerzyweymansyzygiesvectorbundles}, one can find resolutions for general universal degeneracy loci. These resolutions pull back to describe the degeneracy locus of an arbitrary map of vector bundles $E \to F$ on an algebraic stack $Y$, provided the locus has a long-enough regular sequence \cite[Appendix~A.2]{commutativealgebraEisenbud1995}. For schemes $Y$, this gives an alternative to the double stable Grothendieck polynomials. 
\end{remark}

For completeness, we verify that evaluating this formal power series formally recovers our finite operator $\tilde G$.
\begin{proposition}\label{prop:equivalence_r1}
    For a maximal rank drop $r=1$, the formal series yielded by Buch's Grothendieck polynomial and Anderson's determinantal formula \cite{AndersonKtheoreticdegeneracyloci} is formally equivalent to the finite operator $\tilde G(E-F)$ from Definition \ref{def:stackygrothpoly}.
\end{proposition}
\begin{proof}
Recall the $K$-theoretic Chern class \cite{RiemannRochAlgebra}:
\[
\begin{split}
    \gamma_s (E) &:= \lambda_{s/(1-s)}(E) = \sum_{i\geq 0}\gamma^i(E) s^i;
    \\
    c_k^K(E) & := \gamma^k(E - \rk E).
\end{split}
\]
We start from Anderson's formal series. 
\[
\begin{split}
    &\sum_{m \geq 0} \binom{r_2-r_1+m}{m} (-1)^m c^K_{r_2-r_1+1+m}(F-E) 
    \\
    &= \sum_{m \geq 0} \binom{r_2-r_1+m}{m} (-1)^m \left[ (1+s)^{r_2-r_1} \frac{\lambda_t(F^\vee)}{\lambda_t(E^\vee)} \right]_{\text{Coeff } s^{r_2-r_1+1+m}} \quad \left( t := \frac{-s}{1+s} \right) 
    \\
    &= \left[ \frac{1}{1+s} \frac{\lambda_t(F^\vee)}{\lambda_t(E^\vee)} \right]_{\text{Coeff } s^0, |s|>1} 
    \\
    &= \left[ \frac{1}{1+s} \frac{\lambda_t(F^\vee)}{\lambda_t(E^\vee)} \right]_{\text{Coeff } s^0, |s|<1} + \operatorname{Res}_{s=-1} \left( \frac{1}{1+s} \frac{\lambda_t(F^\vee)}{\lambda_t(E^\vee)} \frac{ds}{s} \right) \quad \text{(Residue theorem on }\mathbb{P}^1\text{)}
    \\
    &= 1 + \operatorname{Res}_{s=-1} \left( \frac{1}{1+s} \frac{\lambda_t(F^\vee)}{[\det E^\vee] t^{r_1} \lambda_{t^{-1}}(E)} \frac{ds}{s} \right) \quad \Big(\text{Evaluate } s=0 \text{; Apply }\lambda_t(E^\vee) = t^{r_1} \lambda_{t^{-1}}(E)\Big)
    \\
    &= 1 + \operatorname{Res}_{t=\infty} \left( [\det E] \frac{\lambda_t(F^\vee)}{t^{r_1+1} \lambda_{t^{-1}}(E)} dt \right) 
    \\
    &= 1 - [\det E] \left[ \frac{\lambda_t(F^\vee)}{\lambda_{t^{-1}}(E)} \right]_{\text{Coeff } t^{r_1}}.
\end{split}
\]
This precisely matches the formula for $\tilde G(E-F)$ in Definition \ref{def:stackygrothpoly}.
\end{proof}

\section{The double ramification class as a degeneracy locus} \label{s:5}

The goal of this section is to prove Theorem \ref{introthm:degeneracylocuszerolocus}. We prove a stronger statement allowing $r$th roots for some $r \in \NN$. We work over $\ZZ[\frac{1}{r}]$ in this section. 

Fix a compactified Jacobian $\cal J$ as in Definition~\ref{d:2.20}; it is a log alteration $\cal J \to \logpic$ which is representable by a log algebraic stack and contains as an open substack the space $\pic^{[0]}$ of multidegree-zero line bundles. We work with $\cal J$ because its points parameterize bona fide line bundles on quasistable models of a curve, and we cannot yet make sense of sections, vanishing theorems, and Brill-Noether theory for log line bundles on log curves.

There is a universal admissible bundle $\scr F$ on $\cal J$. The map assigning the trivial log line bundle uniquely factors through a closed immersion into the compactified Jacobian 
\[e : \Ms \to {\cal J} \to \lp.\]
We will present this closed substack as the degeneracy locus of a map of vector bundles for $g \neq 0$% 
, pulled back from a universal case:

\begin{remark}\label{rmk:BrillNoetherdegenlocus}

As in equation \eqref{eqn:dualvbisom}, there is a universal stack 
\[\bar{\rm Mat}_{r_2 \times r_1} \coloneqq \bra{\Aff^{r_1 \times r_2}/\GL_{r_1} \times \GL_{r_2}}\]
endowed with a map of vector bundles $\varphi : E \to F$ of ranks $r_1, r_2$. We suppose $r_1 \leq r_2$ and are interested in the locus $j : W \subseteq \bar{\rm Mat}_{r_2 \times r_1}$ on which the map $\varphi$ does \emph{not} have maximal rank. i.e., the rank of $\varphi$ is at most $r_1 - 1$.

\end{remark}

We work in the universal case but generalize to a general base $B$ in Proposition \ref{prop:arbitrarybaseB}. The proofs are the same, and root stacks in the universal case require more careful analysis.

\subsubsection{Adding $r$th roots}

We generalize to spaces ${\cal J}^{1/r}$ over ${\cal J}$ which parameterize $r$th roots of line bundles. For this, we need stacky curves.

\begin{definition}\label{def:twistednodalcurvesrthroot}
    Write $N_{xy = t} = \Spec \CC[x, y]/(xy - t)$ for the standard node. A \emph{(balanced) twisted nodal curve} $\hat C/S$ \cite{Twistedcurves} of order $r \in \NN$ is a proper DM type map $\hat C \to S$ with coarse moduli space $C/S$ a nodal curve such that
    \begin{itemize}
        \item The map $\hat C \to C$ is an isomorphism over the smooth locus over $S$, and 
        \item After pulling back $\hat C \to C$ along a strict-\'etale neighborhood $C_{loc} \to C$ of a node, there is a pullback square 
        \[
        \begin{tikzcd}
            \hat C_{loc} \ar[r] \ar[d]      &\bra{N_{uv = s}/\mu_r} \ar[d, "{x = u^r, y = v^r}"]         \\
            C_{loc} \ar[r]       &N_{xy = t}
        \end{tikzcd}    
        \]
        where the horizontal maps are \'etale and the coordinates $s, t$ are pulled back from the base $S$. 
    \end{itemize} 
    We do not allow stack structure at marked points and all our nodes have the same $\mu_r$ stack structure, unlike \cite{Twistedcurves}. If the coarse moduli space $C/S$ is a log curve, equip $\hat C$ with the log structure given by $u, v$ at the nodes and the pullback of that of $C$ everywhere else. 
    
    Say a twisted nodal curve is \emph{stable} if its coarse moduli space $C/S$ is. Write $\Ms(r)$ for the moduli space of stable twisted nodal curves. 
\end{definition}

\begin{proposition}\label{prop:twistedcurvesrthlogrootstack}
    The moduli space $\Ms(r)$ is the log root stack (Definition \ref{d:la}) of $\Ms$ to degree $r$. At a geometric point $\bar x \to \Ms$ corresponding to a curve with $k$ nodes, the fiber of the map taking coarse moduli spaces $b'' : \Ms(r) \to \Ms$ is a $\mu_r^k$-gerbe. 
\end{proposition}

This does not mean the root stack at the single nodal divisor $\sqrt[r]{\Ms, D}$, which only has a $\mu_r$-gerbe at the nodal divisor. It is closer to a simultaneous root stack at all the components of the nodal divisor $D \subseteq \Ms$, but differs at self-intersections of these components. 

We first prove a lemma. 

\begin{lemma}\label{lem:rootstacksAFpullback}
    If $\hat C$ is a twisted nodal curve of order $r$ and its coarse moduli space $C/S$ is a \emph{vertical} log curve, we can identify the Artin fans $\af{\hat C} \simeq \af{C}$. Under this identification, there is a pullback square 
    \[
    \begin{tikzcd}
        \hat C \ar[r] \ar[d] \lpbstrict     &\af{C} \ar[d]        \\
        C \ar[r]       &\af{C}
    \end{tikzcd}
    \]
    with right vertical arrow $\af{C} \to \af{C}$ the $r$th power map and horizontal arrows the Artin fan maps.  
\end{lemma}

\begin{proof}
    Implicit in the statement is that there \emph{is} an associated map $\af{\hat C} \to \af{C}$. We can localize in $S$ to assume $S$ is atomic, and \cite[Corollary 3.17]{loghochschild} lets us assume $\af{C} = \af{P} \times S$ and $C$ is atomic. As $C$ is vertical, the relative characteristic sheaf $\gp{\bar M}_{C/S}$ is concentrated at the nodes, so we have reduced to an atomic \'etale neighborhood of a node of $C$. The claim results from the local description in Definition \ref{def:twistednodalcurvesrthroot}. 
\end{proof}

\begin{proof}[{Proof of Proposition \ref{prop:twistedcurvesrthlogrootstack}}]
    The description at geometric points results from the first statement. 

    Stability and marked points are irrelevant, so it suffices to show the same claim for prestable curves $\Mp[g, 0](r) \to \Mp[g, 0]$ with no marked points. \'Etale localize to replace $\Mp[g, 0]$ by $S \coloneqq \Aff^k$, with coordinates $t_1, \cdots, t_k$ corresponding to smoothing parameters of the $k$ nodes of a \emph{vertical} log curve $C/S$. There is exactly one root stack $\hat C \to C$ described by the pullback in Lemma \ref{lem:rootstacksAFpullback}. 
    
    The problem is, $\hat C$ is not a log twisted prestable curve over $\Aff^k$ and its underlying space is not a twisted curve over $\Aff^k$. The map $\hat C \to \Aff^k$ is not saturated. Its universal saturation is plainly the $r$th log root stack of $\Aff^k$ obtained by adjoining roots of $t_1, \cdots, t_k$ as sections of $\Gamma(\Aff^k, \bar M_{\Aff^k})$. I.e., the construction of $\hat C$ results in a twisted nodal curve $\hat C/S$ if and only if the base $S$ contains roots $s_i$ of the smoothing parameters $t_i$.

\end{proof}

\begin{corollary}\label{cor:pfwdgerbe1/r}
    Recall that we work over $\Spec \ZZ[\frac{1}{r}]$ in this section. We have $b''_*[\OO_{\Ms(r)}]=[\OO_{\Ms}]$.
\end{corollary}

\begin{proof}
    The map $b'' : \Ms(r) \to \Ms$ is stratified into gerbes for finite groups with invertible orders, so the excision exact sequence and \cite[Proposition 1.9]{highergenuskthycostellochouherrlee} complete the proof. 
\end{proof}

We can now consider $r$th roots of line bundles.  

\begin{definition}\label{def:rootjacquasistable}
    Let ${\cal J}^{1/r}$ be the log stack parameterizing 
    \begin{itemize}
        \item A family of stable twisted curves $\hat C/S$ in $\Ms(r)$, 
        \item A quasistable model $\hat C' \to \hat C$, 
        \item An admissible line bundle $\scr F$ of total degree zero on $\hat C'$ such that $(\hat C' \to \hat C/S, \scr F)$ lies in $\cal J$, and 
        \item A line bundle $\scr F^{1/r}$ with an isomorphism ${\scr F^{1/r}}^{\otimes r} \simeq \scr F$. 
    \end{itemize}
    Write $\pi' : {\scr C}' \to {\scr C} \to {\cal J}^{1/r}$ for the universal stacky quasistable model viewed as a prestable family of curves over ${\cal J}^{1/r}$. Let $e^{1/r} : E^{1/r} \subseteq {\cal J}^{1/r}$ be the locus where the $r$th root ${\scr F}^{1/r}$ is trivial. 
\end{definition}

\begin{remark}
    We need to allow stack structure at the nodes of our curves to work with roots of line bundles. The stack of $r$th roots of the trivial line bundle $\OO_C$ for example is equivalent to maps to $B\mu_r$ on $C$. As in \cite{Twistedcurves}, this space is not proper unless we place stack structure at the nodes of $C$.  
\end{remark}

\begin{figure}
    \centering
    \begin{tikzcd}
        E^{1/r} \ar[r, "{e^{1/r}}"] \ar[d, "{b'}"] \ar[dd, bend right=40, "b", swap]         &{\cal J}^{1/r} \ar[d, "{\epsilon'}"] \ar[dd, bend left=40, "\epsilon"]        \\
        \Ms(r) \ar[r, "{e'}", swap] \ar[d, "{b''}"] \lpbstrict      &{\cal J}_{\Ms(r)} \ar[d, "{\epsilon''}"] \ar[r] \lpbstrict        &\Ms(r) \ar[d, "{b''}"]         \\
        \Ms \ar[r, "e", swap]         &{\cal J} \ar[r]        &\Ms
    \end{tikzcd}
    \caption{A commutative, noncartesian diagram of the zero sections of ${\cal J}$ and ${\cal J}^{1/r}$.}
    \label{fig:gerbesandrootstacks}
\end{figure}

\begin{remark}\label{rmk:E1/rtrivialgerbe}

    The commutative diagram in Figure \ref{fig:gerbesandrootstacks} is not cartesian. The pullback instead parameterizes $r$th roots of the trivial line bundle \emph{on the curve}, i.e., stable maps to $B\mu_r$. 

    The map $b' : E^{1/r} \to \Ms(r)$ has a section. As $b'$ is a $\mu_r$-gerbe \cite[Lemma 4.5]{chiodoholmesDRroots}, it is in fact the trivial $\mu_r$-gerbe $E^{1/r} \simeq \Ms(r) \times B\mu_r$. 
\end{remark}

\subsubsection{Proof of the theorem for $g \neq 0$}

We want to understand the Gysin pullback along the zero section 
\[
    e^{1/r} : E^{1/r} \to {\cal J}^{1/r}.
\]

\begin{theorem}\label{thm:rthrootGrothendieckdegenlocusDR}
    Use notation as in Definition \ref{def:rootjacquasistable} and assume $g \neq 0$. The structure sheaf of the zero section $e^{1/r} : E^{1/r} \to {\cal J}^{1/r}$ is given by evaluating the operator $\tilde G$ of Definition \ref{def:stackygrothpoly} on the $K$-theoretic class of the complex $[R\pi'_* {\scr F^{1/r}}] = [\pi'_*{\scr F^{1/r}}] - [R^1 \pi'_* {\scr F^{1/r}}]$:
    \[
        [e^{1/r}] = e^{1/r}_* [\OO_{E^{1/r}}] = \tilde G([R\pi'_* {\scr F^{1/r}}]) \qquad \in K({\cal J}^{1/r}).
    \]
\end{theorem}

We handle the case $g = 0$ in 
\S \ref{sss:g=0case} and explain why it is so different in Remark \ref{rmk:whyisg=0weird}. 
We construct a square
\begin{equation}\label{eqn:triviallbBrillNoetherlocus}
\begin{tikzcd}
    E^{1/r} \ar[r] \ar[d, "{e^{1/r}}", swap] \lpbstrict      &W \ar[d]      \\
    {\cal J}^{1/r} \ar[r]         &\bar {\rm Mat}_{r_2 \times r_1}.
\end{tikzcd}
\end{equation}
The closed substack $W \subseteq \bar {\rm Mat}_{r_2 \times r_1}$ is from Remark \ref{rmk:BrillNoetherdegenlocus}. We show the pullback is the locus where the universal $r$th root on ${\cal J}^{1/r}$ admits sections, which \cite[Lemma 4.4]{chiodoholmesDRroots} then identifies with the zero locus $e^{1/r} : E^{1/r} \to {\cal J}^{1/r}$ and so this square is cartesian in both the f.s.\ and ordinary categories. 

Recall $\pi' : {\scr C}' \to {\cal J}^{1/r}$
is the universal quasistable model and $\scr F^{1/r}$ is its universal $r$th root line bundle. Given a divisor $D$ on ${\scr C}'$, we have an exact sequence 
\[
    0 \to \OO_{\scr C'} \to \OO_{\scr C'}(D) \to \OO_D(D) \to 0
\]
which twists by $\scr F^{1/r}$ to become 
\begin{equation}\label{eqn:FtwistDses}
    0 \to \scr F^{1/r} \to {\scr F^{1/r}}(D) \to {\scr F}^{1/r}_D(D) \to 0.
\end{equation}
Fix a sufficiently ample divisor $D$ on ${\scr C}'$ such that 
\[
    R^p \pi'_* {\scr F^{1/r}}(D) = 0 \quad \text{and} \quad R^p \pi'_* {\scr F}^{1/r}_D(D) = 0 \qquad \text{for } p \neq 0. 
\]
Then we have an exact sequence
\begin{equation}\label{eqn:Ftwistlongexactseq}
    0 \to \pi'_* \scr F^{1/r} \to \pi'_* {\scr F^{1/r}}(D) \to \pi'_* {\scr F}^{1/r}_D(D) \to R^1 \pi'_* {\scr F}^{1/r} \to 0.
\end{equation}

\begin{lemma}
    The sheaves $\pi'_* {\scr F^{1/r}}(D), \pi'_* {\scr F}^{1/r}_D(D)$ are vector bundles on ${\cal J}^{1/r}$ of some ranks $r_1, r_2 \in \NN$. 
\end{lemma}

\begin{proof}
    As ${\scr F^{1/r}}(D), {\scr F}^{1/r}_D(D)$ are line bundles on ${\scr C}'$, they are flat over ${\cal J}^{1/r}$. By \cite[07VK]{sta}, their derived pushforwards $R\pi'_* {\scr F^{1/r}}(D), R\pi'_* {\scr F}^{1/r}_D(D)$ form perfect complexes. Their cohomologies vanish in nonzero degrees by assumption on $D$, so they are perfect complexes in degree zero, or vector bundles. 
\end{proof}

\begin{lemma}\label{lem:r1-r2genus}
    The difference $r_1 - r_2$ of the ranks of the bundles $\pi'_* {\scr F^{1/r}}(D), \pi'_* {\scr F}^{1/r}_D(D)$ described above is 
    \[
        r_1 - r_2 = 1 - g,
    \]
    where $g$ is the genus of the curves in ${\cal J}^{1/r} \to \Ms(r)$. 
\end{lemma}

\begin{proof}
The calculation will result from Riemann-Roch. The numbers $r_1, r_2$ are constant on $\cal J^{1/r}$, so we can specialize to any geometric point $\bar s \to {\cal J}^{1/r}$ to compute them. Suppose $\bar s \to {\cal J}^{1/r}$ lies in the open locus $\pic^{[0]}$ and its image under the map ${\cal J}^{1/r} \to \Ms(r)$ lies in the locus of smooth curves. Then the curves agree $\scr C'_{\bar s} = \scr C_{\bar s}$ and are smooth. The exact sequence \eqref{eqn:Ftwistlongexactseq} specializes at $\bar s$ to give us 
\[
    r_1 - r_2 = h^0({\scr C}'_{\bar s}, {\scr F}^{1/r}_{\bar s}) - h^1({\scr C}'_{\bar s}, {\scr F}^{1/r}_{\bar s}). 
\]
Riemann-Roch equates this value with 
\[
    r_1 - r_2 = \deg(\scr F^{1/r}_{\bar s}) - g + 1,
\]
but $\bar s \to {\cal J}$ lands in the open locus $\pic^{[0]}$ where the multidegrees of $\scr F^{1/r}$ and $\scr F$ are zero, so $r_1 - r_2 = 1 - g$. 
\end{proof}

We are ready to prove the $g \neq 0$ case of the theorem. The $g = 0$ case is handled in the next section. 

\begin{proof}[{Proof of Theorem \ref{thm:rthrootGrothendieckdegenlocusDR}}] 
    By the analysis in \cite[\S 2.2]{Kthyvirtpbs}, the product with $[\OO_{E^{1/r}}]$ in the $K$-theory ring $K({\cal J}^{1/r})$ is the same as the Gysin pullback along $e^{1/r}$. 

    As $g \geq 1$, Lemma \ref{lem:r1-r2genus} shows $r_1 \leq r_2$. The codimension of the $(r_1-1)$-degeneracy locus 
    $w : W \subseteq \bar{\rm Mat}_{r_2 \times r_1}$ from Remark \ref{rmk:BrillNoetherdegenlocus} is 
    \[ 
        (r_1 - (r_1-1))(r_2 - (r_1-1)) = 1 \cdot (r_2 - r_1 + 1) = g.
    \]
    This matches the codimension of $e^{1/r} : E^{1/r} \to {\cal J}^{1/r}$, so 
    Corollary \ref{cor:thomporteousKthytildeGdegenlocusformula}
    (compare with \cite[Theorem 2.3]{andersbuchthomporteous}) 
    computes the structure sheaf as 
    \[
        e^{1/r}_* [\OO_{E^{1/r}}] = \tilde G([\pi'_* {\scr F}^{1/r}(D)] - [\pi'_* {\scr F}^{1/r}_D(D)]).
    \]
    The exact sequence \eqref{eqn:Ftwistlongexactseq} identifies the classes 
    \[
        [\pi'_* {\scr F}^{1/r}(D)] - [\pi'_* {\scr F}^{1/r}_D(D)] = [\pi'_*{\scr F}^{1/r}] - [R^1 \pi'_* {\scr F}^{1/r}]
    \]
    in $K$-theory and concludes our argument. 
\end{proof}

\subsubsection{The case $g = 0$}\label{sss:g=0case}

\begin{lemma}\label{lem:g=0istrivial}
    If $g = 0$ and $\vec a \in \ZZ^n_0$, then 
    \[
        \D[0, n, \vec a] = \lp[0, n] = {\cal J} = \Ms[0, n],
    \]
    where $\cal U \to \Ms[0, n]$ is the universal curve. The stack of $r$th roots is 
    \[
        {\cal J}^{1/r} = \Ms[0, n](r) \times B\mu_r = E^{1/r}.
    \]
\end{lemma}

\begin{proof}
    The claim for $\D[0, n, \vec a]$ results from the others and Definition \ref{def:DAJkernel}. By the exact sequence 
    \[
        1 \to \pic^{[0]}_{\cal U/\Ms[0, n]} \to \lp[0, n] \to {\rm TroPic}_{\vert{\cal U}/\Ms[0, n]} \to 0
    \]
    of sheaves on the big strict \'etale site of abelian groups over $\Ms$ \cite[Theorem 4.14.6]{logpic}, 
    it suffices to show 
    \[
        {\rm TroPic}_{\vert{\cal U}/\Ms[0, n]} = \Ms \quad \text{and} \quad \pic^{[0]}_{\vert{\cal U}/\Ms[0, n]} = \Ms.
    \]
    The latter claim is immediate, as multidegree zero line bundles on genus zero curves are trivial. The former is Example \ref{ex:tropicgenusg=0}. 
    The version for $r$th roots is similar. 
\end{proof}

\begin{remark}\label{rmk:whyisg=0weird}
    The operations $e^{1/r, !}, e^{1/r}_*$ 
    are trivial in genus $g = 0$. But they no longer represent the class of $w : W \subseteq \bar {\rm Mat}_{r_2 \times r_1}$ as in the $g \neq 0$ case. By Lemma \ref{lem:r1-r2genus}, we have $r_1 = r_2 + 1$, so the middle map 
    \[
        \pi'_* {\scr F}^{1/r}(D) \to \pi'_* {\scr F}_D^{1/r}(D)
    \]
    of the exact sequence \eqref{eqn:Ftwistlongexactseq} is generically surjective. The locus $W$ detects where this map ceases to be surjective, not injective. So $W$ is the support of the cokernel $R^1\pi'_* {\scr F}^{1/r}$ instead of the kernel, and it is codimension two. 
\end{remark}

\subsubsection{Towards a $K$-theoretic Pixton formula}\label{sss:kthypixtonkamyarlastmorsel}

Chiodo and Holmes prove a precursor to Pixton's formula for the log DR class \cite[Main Theorem 2]{chiodoholmesDRroots}:
\[
    {\rm LogDR}({\cal L}) \cdot \eta^u = r^{u+1} \epsilon_* c_{g + u}(-R\pi_*{\cal L}^{1/r}).
\]
The LHS is independent of $r$, while the RHS has a factor depending on the exponent $u$. We prove the $K$ theoretic analogue in the case $u = 0$, in which no factor of $r$ appears on either side.

\begin{proposition}\label{prop:drclassgrothpoly}
    Use the notation of Figure \ref{fig:gerbesandrootstacks}. The structure sheaf of the $r$th root degeneracy locus $E^{1/r}$ pushes forward to that of the zero section in ${\cal J}$ in $K$-theory:
    \begin{equation}\label{eqn:pfwdalonggerbes}
        \epsilon_* e^{1/r}_* [\OO_{E^{1/r}}] = e_* b_* [\OO_{E^{1/r}}] = e_* [\OO_{\Ms}] \qquad \in K_\circ({\cal J}).
    \end{equation}

    The resulting Thom--Porteous formula for the log DR class is the same with or without the $r$th roots:
    \begin{align*}
        \lvir{\D} &= \aj_{\vec a}^! \tilde G(R\pi'_*{\scr F})         \\
                &= \aj_{\vec a}^! \epsilon_* \tilde G(R\pi'_*{\scr F}^{1/r}) \qquad \in K_\circ(\tilde \Ms)
    \end{align*}
\end{proposition}

\begin{proof}
    The map $b' : E^{1/r} \to \Ms(r)$ is a gerbe by the proof of \cite[Lemma 4.5]{chiodoholmesDRroots}, so the structure sheaves agree under pushforward $b'_* [\OO_{E^{1/r}}] = [\OO_{\Ms(r)}]$ by \cite[Proposition 1.9]{highergenuskthycostellochouherrlee}. Combined with Corollary \ref{cor:pfwdgerbe1/r}, we get that $b_*[\OO_{E^{1/r}}] = [\OO_{\Ms}]$ and the equalities \eqref{eqn:pfwdalonggerbes} result. 

    Factor the square \eqref{eqn:AJkernel} via the log alteration ${\cal J} \to \lp$:
    \begin{equation}\label{eqn:DRhavetoblowup}
    \begin{tikzcd}
        \D \ar[r] \ar[d] \lpbstrict      &E \ar[d, "e"]        \\
        \tilde \Ms \ar[r, "{\tilde \aj_{\vec a}}"] \ar[r] \ar[d] \lpb      &{\cal J} \ar[d]    \\
        \Ms \ar[r, "{\aj_{\vec a}}"]      &\lp
    \end{tikzcd}
    \end{equation}
    Endow $\tilde \aj_{\vec a}$ with the pullback of the natural log perfect obstruction theory on $\aj_{\vec a}$ so that pulling back along either one is the same operation. As $\lvir{\D} = \aj_{\vec a}^! [\OO_{\Ms}]$ is the pullback of the zero section along the Abel-Jacobi map, we conclude by plugging in our Thom--Porteous formulas for the zero section from Theorem \ref{thm:rthrootGrothendieckdegenlocusDR}. 
\end{proof}

We conclude by generalizing the results of this section to an arbitrary base $B$. The proofs are the same. 

Let $B$ be a log smooth, quasicompact log algebraic stack with a family of log twisted prestable curves $\hat C/B$ with a log line bundle $P/\hat C$ and a vector $\vec a \in \ZZ^n_0$. Define the compactified Jacobian and zero section ${\cal J}_B, E^{1/r}_B$ over $B$ by f.s.\ pullback:
\[\begin{tikzcd}
    E^{1/r}_B \ar[r] \ar[d, "e^{1/r}_B", swap] \lpb       &E^{1/r} \ar[d]        \\
    {\cal J}_B \ar[r] \ar[d] \lpb      &\cal J \ar[d]         \\
    B \ar[r]       &\Ms(r)
\end{tikzcd}\]
Write $\pi'_B : {\scr C}'_B \to {\cal J}_B$ for the pullback of the quasistable model. Suppose the map $B \to \Ms(r)$ (or equivalently the map to $\Ms$) is eventually strict. 

\begin{proposition}\label{prop:arbitrarybaseB}
    In the above setting, suppose $e_B^{1/r} : E^{1/r}_B \to {\cal J}_B$ has the expected codimension $g$. The structure sheaf of the zero section is then given by the Grothendieck operator $\tilde G$:
    \[
    e^{1/r}_{B*} [\OO_{E^{1/r}_B}] = \tilde G(R\pi'_{B*} \scr F^{1/r}) \qquad \in K^\circ(\cal J_B).
    \]

    Pushing forward to $B$, we have an equality 
    \begin{align*}
        \lvir{\dr} &= \aj_{\vec a}^! \tilde G(R\pi'_*{\scr F})         \\
                &= \aj_{\vec a}^! \epsilon_* \tilde G(R\pi'_*{\scr F}^{1/r}) \qquad \in K_\circ(B)
    \end{align*}
\end{proposition}

\begin{proof}
    The proof is the same as that of Theorem \ref{thm:rthrootGrothendieckdegenlocusDR} and Proposition \ref{prop:drclassgrothpoly}. The analogous diagram to \eqref{eqn:DRhavetoblowup} is 
    \[
    \begin{tikzcd}
        \dr \ar[r] \ar[d] \lpb         &E \ar[d, "e"]      \\
        \tilde B \ar[r, "{\tilde \aj_{\vec a}}"] \ar[d] \lpb        &{\cal J}_B \ar[d]    \\
        B \ar[r, "{\aj_{\vec a}}"]       &\logpic_{\hat C/B}.
    \end{tikzcd}
    \]
\end{proof}

\appendix

\section{Background on log structures}\label{Appendix B}

Assuming familiarity with the basics of log geometry of \cite{katooriginal, ogusloggeom}, we review some specific constructions. See also \cite[Appendix A]{firmlogrationalpointsherrpieropanthibaultsara}.

\subsection{Log structures}

\begin{definition}
    Let $X$ be a scheme or algebraic stack and view its structure sheaf $\OO_X$ as a sheaf of monoids under multiplication. A \emph{log structure} on $X$ is a sheaf of monoids $M_X$ and a monoid homomorphism $\alpha: M_X\to \mathcal{O}_X$ %, where the monoidal structure on $\mathcal{O}_X$ is given by multiplication of functions, 
    such that the restriction of $\alpha$ gives an isomorphism
    \[\alpha^{-1}(\mathcal{O}_X^*)\longsimeq \mathcal{O}_X^*.\]
    We identify $\alpha^{-1} \OO_X^*$ and $\OO_X^*$ via this isomorphism. The \emph{characteristic monoid} is the quotient sheaf of monoids $\bar M_X \coloneqq M_X/\OO_X^*$. 
\end{definition}

The main example comes from toric geometry. 

\begin{example}\label{ex:logstraffinetoricvar}
    Let $P$ be a sharp, f.s.\ monoid, engendering an affine toric variety $X = \Aff_P \coloneqq \Spec \ZZ[P]$. The inclusion $P \subseteq \ZZ[P]$ gives a natural map $\alpha : \underline P \to \OO_X$ from the constant sheaf. 

    This map is not yet a log structure, but we can fix it. Define $M_X$ as the pushout 
    \[
    \begin{tikzcd}
        \alpha^{-1} \OO_X^* \ar[r] \ar[d] \ar[dr, phantom, very near end, "\lrcorner"]       &\OO_X^* \ar[d]      \\
        \underline P \ar[r]        &M_X
    \end{tikzcd}
    \]
    in the category of sheaves of monoids and verify that $M_X$ is a log structure on $X$. The stalks of the characteristic monoid $\bar M_X$ are quotients of $P$ by faces. There is a locally closed stratification of $X$ into torus orbits on which $\bar M_X$ restricts to a constant sheaf. 
\end{example}

\begin{definition} 
    A log structure on $X$ is \emph{f.s.}\ (fine and saturated) if, strict \'etale locally on $X$, there is a strict map $X \to \Aff_P$ to the affine toric variety with log structure associated to a sharp, f.s.\ monoid $P$ as in Example \ref{ex:logstraffinetoricvar}. 
\end{definition}

For the purposes of this work, all schemes and algebraic stacks are locally of finite type, and all log structures are f.s.\ unless otherwise stated. An f.s.\ log structure is a sort of ``local toric structure.''

\begin{definition}
A morphism of log algebraic stacks $X\to Y$ is a morphism of stacks $f: X\to Y$ and a morphism of sheaves of monoids $f^{-1}M_Y\to M_X$ compatible with the structure maps $f^{-1}\alpha_Y$ and $\alpha_X$.  
Such a morphism is \emph{strict} if it induces an isomorphism $f^* M_Y \xrightarrow{\sim} M_X$, where $f^*M_Y$ is the log
structure associated to the pre-log structure $ f^{-1}M_Y \rightarrow f^{-1}\OO_Y \rightarrow \OO_X $. 
\end{definition}

\begin{remark}
    An f.s.\ log algebraic stack admits a maximal locally closed stratification $X = \bigsqcup X_\alpha$ such that 
    \begin{itemize}
        \item $X_\alpha \subseteq X$ are locally closed and connected, and 
        \item $\bar M_X|_{X_\alpha}$ is locally constant. 
    \end{itemize}
    These strata can be defined using a local map to a toric variety, pulling back its stratification, and taking their connected components, or by taking finite Boolean combinations of closed subschemes associated to the monoid ideals in $\bar M_X$ \cite[Section~2.2]{logpic}. 
\end{remark}

\begin{definition}
    A log algebraic stack $X$ is \emph{atomic} if 
    \begin{itemize}
        \item it has a unique closed stratum $X' \subseteq X$ which is connected and, 
        \item for any geometric point $\bar x \to X'$, the restriction map 
    \[\Gamma(X,\bar{M}_X) \to \bar{M}_{X,\bar x}\]
    is an isomorphism.
    \end{itemize}
\end{definition}

\begin{proposition}[{\cite[Proposition~2.2.2.5]{molchowiselogetaledescentnote}}]\label{prop:coverbyatomics}
    If $X$ is a locally noetherian log algebraic stack, there is a strict smooth cover $\{V_\alpha \to X\}$ with $V_\alpha$ atomic. 
\end{proposition}

\subsection{The universal stack $\Log$}
An equivalent formulation of the theory of log schemes and stacks can be developed using Olsson's stack $\Log$, parameterizing log structures. We give a brief exposition of this formulation here.

\begin{definition}\cite{logstacks}
 Let $\Log$ be the log algebraic stack that parametrizes f.s.\ log structures with $T$ points:
\[
    \Log(T) :=\{ \text{f.s.\ log structures on }T \}.
\]
Given a log algebraic stack $X$, there is a unique morphism:
\[
f: X\to \Log \text{ such that } f^*M_{\Log}=M_X.
\]
\end{definition}

We assume all log structures are f.s., but emphasize this as part of the definition of $\Log$. 

\begin{definition}
For any log algebraic stack $Y$, define $\Log_Y$ as the algebraic stack parameterizing log structures with a log map to $Y$:
\[
    \Log_Y(T):=\{ \text{a log structure on }T \text{ and a log map }T \rightarrow Y \}.
\]
\end{definition}

The advantage of this perspective is it gives a unified way to define properties of log morphisms. We use the following meta-definition for all such properties:

\begin{definition}
If $P$ is a property of schemes or stacks,
a log scheme (or log stack) $X$ is defined to have property \emph{log $P$} if  the map $X\to \Log$ has property $P$. 

If $P$ is a property of morphisms of schemes or stacks, a morphism $f:X\to Y$ of log schemes (or log stacks) is defined to have property \emph{log $P$} if the map $X\to \Log_Y$ has property $P$.
    
\end{definition}

\subsection{Log curves}\label{ss:logcurves}

A \emph{log curve} over a log scheme $S$ is a log smooth, integral, saturated, proper morphism $\pi: C \rightarrow S$ of f.s.\ log schemes with connected geometric fibers of relative dimension one. It results that the geometric fibers of $\pi$ are reduced, prestable curves.

The relative characteristic  $\gp{\bar M}_{C/S} \coloneqq \gp {\bar M}_{C}/\gp {\bar M}_S|_C$ is supported on the special locus of the curve, i.e. on a geometric point $x$
\[
    \gp{\overline{M}}_{C/S, x} \simeq
    \begin{cases}
        \mathbb{Z} \quad \text{if }x\text{ is a node or marked point},
        \\
        0 \quad \text{otherwise}.
    \end{cases}
\]
The $n$ marked points are part of the data of the log morphism $C \to S$, the points $p \in C$ at which $\bar M_{C, p} \simeq \bar M_S \oplus \NN$. The moduli space of log curves is precisely $\Mp$ equipped with its normal crossings divisor of singular curves \cite{fkatomodulilogcurves}. 

We can equivalently equip $C$ with a different log structure $M_{\vert C}$ such that the relative characteristic $\gp {\bar M}_{\vert C/S}$ is supported only at the nodes. Then $M_C$ is the pushout 
\[
    M_C \simeq M_{C^{\rm vert}} \oplus_{\OO_C^*} \bigoplus M_{p_i}. 
\]

\begin{definition}\label{def:standardverticallogcurves}
We refer to $C/S$ equipped with the log structure $M_C$ as a \emph{standard} log curve and the curve $C^\circ/S$ equipped with $M_{\vert C}$ as a \emph{vertical} log curve, or the \emph{verticalization} $\vert C/S$ of $C/S$. 
\end{definition}

\subsection{Artin fans}
For any f.s.\ log scheme $X$, there exists a natural map to an algebraic stack $\Theta_X$, called the \emph{Artin fan}, which is constructed as the colimit of the Artin cones corresponding to the local characteristic monoids of $X$. The Artin fan is a combinatorial object; here we discuss the basics.

Let $(\mathbf{Mon}^\#)$ denote the category of sharp monoids. Here a monoid is an integral, saturated, finitely generated commutative, unital semigroup. A monoid is \emph{sharp} if its group of units is zero.

\begin{definition}
The category of cones is defined to be the opposite category of sharp monoids:
\[
    (\mathbf{Cones}) := (\mathbf{Mon}^\#)^{\text{op}}.
\]
We write $\operatorname{Cone}P$ for the cone corresponding to a monoid $P$, and $\overline{M}_\sigma$ for the monoid corresponding to a cone $\sigma$.
\end{definition}

\begin{remark}
The assignment $P \mapsto \Hom(P, \NN)$ yields a self-duality on the category of sharp monoids \cite[Theorem I.2.2.3]{ogusloggeom}:
    \[
    \Mons^\# \xrightarrow{\,\sim\,} \Mons^{\# \rm op}; \qquad P \longmapsto \Hom(P, \NN).
    \]
\end{remark}

\begin{definition}
A \emph{face} of a monoid $P$ is a submonoid $Q \subseteq P$ such that for all $a,b \in P$, if $a+b \in Q$, then $a,b \in Q$. A \emph{face localization} is a quotient $P \rightarrow P/Q$. 

In the category of cones, we define \emph{face inclusions} as the dual of face localization, denoted $\operatorname{Cone} P/Q \hookrightarrow \operatorname{Cone} P$. So the face (in the usual sense) of a polyhedral cone is associated to a quotient map of the associated monoids. 
\end{definition}

\begin{definition}
An \emph{Artin cone} is the quotient stack of the affine toric variety $\Aff_P$ by its dense torus $T_P$:
\[
    \Theta_P := [ \Spec \mathbb{Z}[P] / \Spec \mathbb{Z} [P^{\rm gp}] ],
\]
where $P$ is a sharp monoid. The canonical log structure of the toric variety $\Aff_P$ descends to $\Theta_P$. Write $\af{} = \af{\NN} = \bra{\Aff^1/\GG_m}$. 
The category of Artin cones, denoted $\mathbf{ArtinCones}$, is the full subcategory of log algebraic stacks whose objects are Artin cones.
\end{definition}

Artin cones and cones encode the same data:
\begin{theorem}\cite[Theorem~6.11]{tropicalcurvesmodulicones}
The functor
\[
    (\mathbf{Cones}) \xrightarrow\sim (\mathbf{ArtinCones}); \qquad \sigma \mapsto \Theta_{\overline{M}_\sigma}
\]
is an equivalence.
\end{theorem}

Just as a cone complex is constructed by gluing individual cones along their common faces, we can globalize Artin cones into Artin fans.

\begin{definition}
An \emph{Artin fan} is a log algebraic stack $\scr B$ which admits a strict \'etale cover $\{\af{P_i} \to \scr B\}$ by Artin cones.
\end{definition}

The stack $\Log$ is an Artin fan by \cite[Corollary 5.25]{logstacks}. 
Maps between Artin fans are those of log algebraic stacks. If a log algebraic stack $\scr B$ is an Artin fan, the map $\scr B \to \Log$ is \'etale. Over a separably closed field, if $\scr B \to \Log$ is \'etale and representable by algebraic spaces, then it is an Artin fan \cite[Lemma 2.3.1]{birationalinvarianceabramovichwise}. This direction is false without the representability assumption \cite[Remark 2.3.10]{birationalinvarianceabramovichwise} or the separable closure assumption (consider $\scr B = \Spec \ZZ[1/3]$ with trivial log structure).

To construct \emph{the} Artin fan of a log scheme or log algebraic stack $X$, we apply Proposition~\ref{prop:coverbyatomics}, which guarantees the existence of a strict, smooth cover of $X$ by atomics. Consequently, we can find a strict, smooth hypercover $V_{\bullet}$ of $X$.

\begin{definition}[{\cite[Proposition~3.2.1]{wisebounded}}]
    Let $X$ be a log algebraic stack, \emph{the Artin fan} $\af{X}$ of $X$ is the colimit 
    \[
    \af{X} = \text{colim}_{V_\bullet} \af{V_\alpha}
    \]
    in the category of \'etale sheaves on $\Log$, viewed as a $\Log$-algebraic space via the \'espace \'etal\'e construction. 
\end{definition}

The reader can check this construction is independent of the choice of hypercover. 
We have a factorization 
\[
X \to \af{X} \to \Log.
\]
Over $\CC$, this is the initial factorization of the map $X \to \Log$ through a strict, \'etale, $\Log$-algebraic space \cite[Proposition~3.2.1]{wisebounded}. We prefer to take their construction as the definition in mixed characteristic, as opposed to this initial factorization property.

\subsection{Log alterations, modifications, and blowups}
Terminology regarding log alterations and modifications has subtle differences across the literature. To avoid ambiguity, we make a definition here.

\begin{definition} \label{d:la}
A morphism of Artin fans $\pi: \mathcal{A}_1 \rightarrow \mathcal{A}_2$ is a \emph{log alteration} if it is 
proper, birational, and has finite flat diagonal. Furthermore, a log alteration of Artin fans $\pi$ is called a:
\begin{itemize}
    \item \emph{Log modification} if $\pi$ is representable. 
    \item \emph{Log blowup} if $\pi$ is representable and projective. 
    \item \emph{Log root stack} if, for each Artin cone $\tau \to \cal A_2$, the pullback $\sigma \coloneqq \tau \times_{\cal A_2}^{\ell} \cal A_1$ is an Artin cone, and the induced morphism on characteristic monoids 
    \begin{equation}\label{eqn:rootstackdualconemorphism}
        \bar M_{\tau} \to \bar M_{\sigma}
    \end{equation}
    induces an isomorphism on the associated $\QQ$-vector spaces. 
    \item \emph{Log root stack of order $r$} 
    if it is a log root stack and each morphism \eqref{eqn:rootstackdualconemorphism} is the multiplication-by-$r$ map $[r] : \bar M_\tau \to \bar M_\tau$. 
\end{itemize}

A morphism of log algebraic stacks $f: X \rightarrow Y$ is a \emph{log alteration} if, strict \'etale locally on $Y$, it fits into an f.s.\ pullback square
    \[
    \begin{tikzcd}
        X \ar[r] \ar[d, "f"] \lpb       &\mathcal{A}_1 \ar[d, "\pi"]         \\
        Y \ar[r]       &\mathcal{A}_2
    \end{tikzcd}
    \]
where $\pi: \mathcal{A}_1 \to \mathcal{A}_2$ is a log alteration of Artin fans.

Analogously, the morphism $f: X \rightarrow Y$ is a \emph{log modification}, \emph{log blowup}, \emph{log root stack}, or \emph{log root stack of order $r$} if, strict \'etale locally, it is pulled back from a morphism of Artin fans $\pi: \mathcal{A}_1 \to \mathcal{A}_2$ satisfying the corresponding property defined above.
\end{definition}

\begin{remark}
If $\af{X}$ is the Artin fan of $X$, it has a map $[r]: \af{X}\to \af{X}$ given by multiplication by $r$. The log root stack $X^{1/r}$ of order $r$ fits into the following f.s.\ and strict pullback square:
\[
\begin{tikzcd}
X^{1/r} \arrow[d] \arrow[r] & \af{X} \arrow[d, "{[r]}"] \\
X \arrow[r] & \af{X}.
\end{tikzcd}
\]
\end{remark}

\begin{example}
The Artin fan of a log scheme $\Aff^1$ with its toric log structure is $\af{\Aff^1} = \af{} = \bra{\mathbb{A}^1/\GG_m}$. The root stack fits into the diagram:
\[
\begin{tikzcd}
(\Aff^1)^{1/r} \arrow[d] \arrow[r] & \bra{\Aff^1/\GG_m} \arrow[d, "\times r"] \\
\Aff^1 \arrow[r] & \bra{\Aff^1/\GG_m}
\end{tikzcd}
\]
The map $(\Aff^1)^{1/r} \to \Aff^1$ is a coarse moduli space which restricts to a $\mu_r$ gerbe over the origin and is the identity elsewhere. 

More generally, equip a scheme or algebraic stack $X$ with the divisorial log structure associated to a smooth divisor $D \subseteq X$. Then $X^{1/r} = \sqrt[r]{X, D}$ is the (ordinary) $r$-th root stack of $X$ at $D$ in the sense of C.~Cadman \cite{cadman}.
\end{example}

The log root stack and ordinary root stack usually do not coincide, for example if an s.n.c.\ divisor $D \subseteq X$ has multiple intersecting branches. The fiber of the log root stack $X^{1/r} \to X$ over a point $x \in X$ is a $B\mu_r^k$-gerbe, where $k$ is the number of branches of $D$ passing through $x \in X$. The ordinary root stack restricts to a $B\mu_r$-gerbe over the whole divisor $D$. 

\begin{example}
    Consider $X = \Aff^n$, with Artin fan $\bra{\Aff^n/\GG_m^n} = \af{}^n$. The log root stack is the pullback 
    \[
    \begin{tikzcd}
        (\Aff^n)^{1/r} \ar[r] \ar[d] \lpbstrict      &\af{}^n \ar[d, "{[r]}"]        \\
        \Aff^n \ar[r]      &\af{}^n.
    \end{tikzcd}
    \]
    The stacky points of $\af{}^n$ are of the form $B\GG_m^k$, corresponding to a linear subspace of $\Aff^n$ cut out by the vanishing of $k$ variables. Pulling back along the multiplication by $r$ map on $B\GG_m^k$ defines a $\mu_r^k$ gerbe over this locus. 
\end{example}

\section{Pushforward of the log virtual fundamental class in \texorpdfstring{$K$}{K}-theory}\label{Appendix A}

\begin{theorem}[{\cite[Theorem 3.10]{herrthesis}, \cite[Proposition 2.5]{kthylogprodfmla}}]\label{thm:kvfcsinlogkthy}
    Let $f : X \to Y$ be a morphism of log algebraic stacks with $X$ quasicompact and DM and $Y$ log smooth. Suppose $f$ is equipped with a log perfect obstruction theory $\Cl{f} \subseteq E$. If $\pi : \tilde X \to X$ is a log alteration, we equip $\tilde X$ with the induced log perfect obstruction theory as in \cite[Remark 2.14, \S 3]{herrthesis}:
    \[\Cl{\tilde X/Y} \subseteq \Cl{X/Y}|_{\tilde X} \subseteq E|_{\tilde X}.\]
Then the corresponding $K$-theoretic virtual structure sheaves satisfy the following pushforward identity
    \[
    \pi_* \OO^{\rm vir}_{\tilde X/Y} = \OO^{\rm vir}_{X/Y} \qquad \in K_\circ(X).
    \]
    Consequently, the collection of virtual structure sheaves ranging over all log alterations $\pra{\OO^{\rm vir}_{\tilde X_i/Y}}_{\tilde X_i \in {\rm Alt}(X)}$ defines a class in log $K$-theory $\Kl(X)$. 
\end{theorem}

We establish several technical lemmas about log alterations, from which the theorem follows by applying \cite[Proposition 2.5]{kthylogprodfmla}. For the convenience of the reader, we begin by recalling a key structural result regarding the subdivision of Artin fans.

\begin{theorem}[{\cite[Theorem 4.6.2]{wisebounded}}]\label{thm:wiseboundedsubdiv}
    Any quasicompact Artin fan $\scr B$ admits a projective subdivision $\tilde {\scr B} \to \scr B$ which admits a strict, \'etale, representable map $\tilde{\scr B} \to \af{}^N$ for some $N \in \NN$. 
\end{theorem}

While the statement of theorem in loc.\ cit.\ only explicitly asserts that $\tilde {\scr B} \to \af{}^N$ is strict, the remaining properties follow directly from the construction.

\begin{corollary}\label{cor:wiseblowupforstacks}
    Let $X$ be a quasicompact log algebraic stack and let $\tilde {\scr B} \to \scr B = \af{X}$ be a projective subdivision of its Artin fan as in Theorem \ref{thm:wiseboundedsubdiv}. Write $\tilde X \coloneqq \tilde{\scr B} \times_{\scr B}^{\msout{\ell}} X$. Then the Artin fan of $\tilde X$ admits a strict, \'etale, representable map $\af{\tilde X} \to \af{}^N$ for some $N \in \NN$.
\end{corollary}

\begin{proof}
    The only subtlety is that the Artin fan of the subdivision need not be the same as the subdivision of the Artin fan $\af{\tilde X} \neq \tilde{\scr B}$. Nevertheless, there is a map $\af{\tilde X} \to \tilde{\scr B}$ which is strict \'etale and representable because they are both strict, \'etale, and representable over $\Log$. 
\end{proof}

We recall a lemma regarding the combinatorial nature of the characteristic monoid.

\begin{lemma}[{\cite[Lemma A.20]{firmlogrationalpointsherrpieropanthibaultsara}}]\label{lem:charmonPLfns}
    Let $X$ be a log algebraic stack and let $\af{X}$ be its associated Artin fan. Global sections of the characteristic monoid are identified with piecewise linear functions on $\af{X}$: 
    \[
        \Gamma(X, \bar M_X) = \Gamma(\af{X}, \bar M_{\af{X}}).
    \]
    This identification means that any morphism $X \to \af{}$ factors uniquely through $\af{X}$. The same goes for maps to $\af{}^N$ for $N \in \NN$.  
\end{lemma}

We need a weak form of the ``universal property'' for later; over a separably closed field, one can simply use the universal property definition of the Artin fan. Consider the map of topoi $(\af{P})_{\et} \to \Spec P$ from the site of \'etale, representable maps to $\af{P}$ to the Zariski site (i.e., the opens) of $\Spec P$. 

\begin{lemma}\label{lem:presheafmonoscheme}
    Suppose a log algebraic stack $X$ admits a strict map $X \to \af{P}$ to an Artin cone. There is an induced \'etale and representable map $\af{X} \to \af{P}$ from the Artin fan. It represents a sheaf pulled back from $\Spec P$.  
\end{lemma}

\begin{proof}
    Obtain the map $\af{X} \to \af{P}$ from Lemma \ref{lem:charmonPLfns} by presenting $P$ as a coequalizer of free monoids. It is necessarily \'etale and representable because its source and target are over $\Log$. 

    Take a hypercover $V_1 \rightrightarrows V_0 \to X$ where $V_i$ is a disjoint union of atomics. The Artin fan $\af{X}$ is the coequalizer on $\Log$ of 
    \[
        \af{V_1} \rightrightarrows \af{V_0}. 
    \]
    By strictness of the maps $V_i \to \af{P}$, each direct summand of $\af{V_i}$ corresponds to the inclusion of a face $\af{\bar P} \to \af{P}$. This is clearly pulled back from the sheaf represented by the open $\Spec \bar P \subseteq \Spec P$. The coequalizer of sheaves on $\Log$ or equivalently on $\af{P}$ is pulled back from that on $\Spec P$ because the pullback functor is exact. 
\end{proof}

\begin{lemma}\label{lem:factorthroughpresheafmonoscheme}
    Let $X$ be a log algebraic stack and consider a map $\af{X} \to \af{P}$ of log algebraic stacks which need not be strict. If $F$ is a sheaf on $\Spec P$ equipped with a section $s \in \Gamma(X, F|_X)$ of the pullback along the maps 
    \[
        X_{\et} \to (\af{X})_{\et} \to (\af{P})_{\et} \to \Spec P
    \]
    of \'etale-representable topoi, it is the restriction of a unique section $\tilde s \in \Gamma(\af{X}, F|_{\af{X}})$. 
\end{lemma}

\begin{proof}
    Choose a hypercover of $X$ by atomics, yielding a hypercover of $\af{X}$. By descent, a section of $F$ on $\af{X}$ is equivalent to a section over this hypercover of $\af{X}$ by Artin cones. It suffices to assume $X$ is atomic, in which case one can base change the sheaf $F$ to assume $\af{X} = \af{P}$. A section of a presheaf on $\Spec P$ is equivalent to a section over the closed point, as the only open neighborhood of the closed point is all of $\Spec P$. The pullback $F|_X$ is constructible with respect to the same stratification as $\bar M_X$ and constant on each stratum. So global sections of $F|_X$ on $X$ are also given by its stalk at a geometric point in the deepest closed stratum. As this deepest closed stratum is connected, we can identify sections of $F$ over the closed point of $\Spec P$ with sections of $F|_X$ over a geometric point in the closed stratum of $X$. 
    
    % in which there is a deepest stratum which every other stratum specializes to. It results that global sections of $F|_X$ are identified with the stalk at a separably closed point of this deepest stratum. 
    
    % which is provided by the stalk of the global section $s \in \Gamma(X, F|_X)$. 
\end{proof}

\begin{proposition}\label{prop:AFfunctorialifYfree}
    Let $X \to Y$ be a map of log algebraic stacks and suppose the Artin fan of the target $Y$ admits a strict \'etale representable map $\af{Y} \to \af{}^N$ for some $N \in \NN$. There is a unique map of Artin fans fitting in a commutative square 
    \[
    \begin{tikzcd}
        X \ar[r] \ar[d]       &\af{X} \ar[d]         \\
        Y \ar[r]       &\af{Y}.
    \end{tikzcd}
    \]
\end{proposition}

\begin{proof}
    By Lemma \ref{lem:charmonPLfns}, we have the outer solid trapezoid
    \[
    \begin{tikzcd}
        X \ar[r] \ar[d]       &\af{X} \ar[d, dashed] \ar[dr]         \\
        Y \ar[r]       &\af{Y} \ar[r]      &\af{}^N.
    \end{tikzcd}
    \]
    Write $\scr B \coloneqq \af{Y} \times_{\af{}^N}^{\msout{\ell}} \af{X}$ for the strict \'etale Artin fan over $\af{X}$ through which $X$ factors:
    \[
        X \to \scr B \to \af{X}. 
    \]
    Over a separably closed field, the universal property of the Artin fan provides a unique section $\af{X} \dashrightarrow \scr B$, which is equivalent to the desired dashed map. Over $\Spec \ZZ$, one must combine Lemmas \ref{lem:presheafmonoscheme}, \ref{lem:factorthroughpresheafmonoscheme}. 
\end{proof}

\begin{proof}[{Proof of Theorem \ref{thm:kvfcsinlogkthy}}]
    By \cite[Proposition 2.5]{kthylogprodfmla}, the theorem holds if there exists a log alteration $\tilde{\scr B} \to \scr B$ of Artin fans and a pullback square 
    \[
    \begin{tikzcd}
        \tilde X \ar[r] \ar[d] \lpbstrict        &\tilde{\scr B} \ar[d]         \\
        X \ar[r]       &\scr B
    \end{tikzcd}
    \]
    in the f.s.\ and underlying stack categories. We reduce the theorem to this special case. These pullback squares always exist strict-\'etale locally in $X$ by definition of a log alteration. 

    Arguing as in \cite[Lemma 3.5]{kthylogprodfmla}, there is a log alteration pulled back from the Artin fan $\af{X}$ that refines any fixed log alteration of $X$. So we have a commutative square of log alterations of $X$ where the vertical maps are pulled back from log alterations of Artin fans:
    \[
    \begin{tikzcd}
        \tilde{\scr B}' \ar[d]         &\tilde X' \ar[r] \ar[d] \ar[l] \ar[dl, very near start, phantom, "{\msout{\ell} \,\urcorner}"]       &X' \ar[d] \ar[r] \lpbstrict        &\scr B' \ar[d]       \\
        \af{\tilde X}       &\tilde X \ar[r] \ar[l]       &X \ar[r]        &\af{X}.
    \end{tikzcd}
    \]
    Applying the above special case of the theorem to $\tilde X' \to \tilde X$ and $X' \to X$, we can assume $\tilde X' = \tilde X$ and $X' = X$. Use this reduction and Corollary \ref{cor:wiseblowupforstacks} to assume $\af{X}$ admits a strict, \'etale, representable map to $\af{}^N$ for some $N \in \NN$. 
    
    Proposition \ref{prop:AFfunctorialifYfree} provides a commutative square 
    \[
    \begin{tikzcd}
        \tilde X \ar[r] \ar[d]        &\af{\tilde X} \ar[d]      \\
        X \ar[r]       &\af{X}.
    \end{tikzcd}
    \]
    Again, there is a log alteration $\tilde{\scr B} \to \af{X}$ which pulls back to a refinement $X' \to X$ of $\tilde X \to X$. Express the f.s.\ pullback two ways: 
    \[X' = \tilde X \times_X^\ell X' %= \tilde X \times_{\af{X}}^\ell \tilde{\scr B} 
    = \tilde X \times^{\msout{\ell}}_{\af{\tilde X}} \pra{\af{\tilde X} \times_{\af{X}}^\ell \tilde{\scr B}}\] 
    to see that the log alterations $X' \to \tilde X$ and $X' \to X$ are each pulled back from the Artin fan of the target. We are done by the above special case. 
\end{proof}

\bibliographystyle{alpha}%Used BibTeX style is unsrt
\bibliography{zbib}

\end{document}